\documentclass[12pt]{article}

\usepackage[margin=1in]{geometry}
\usepackage[utf8]{inputenc}
\usepackage{listings}
\usepackage[framed,numbered]{matlab-prettifier}

\usepackage[normalem]{ulem}
\usepackage{ftnxtra}

\usepackage{rotating}
\usepackage{graphicx}
\usepackage{adjustbox}

\usepackage{amsmath,amssymb,amsthm,enumerate, booktabs, color, multirow, scrextend,url}
\newcommand{\rev}[1]{\textcolor{black}{#1}}
\newcommand{\Fir}[1]{\textcolor{black}{#1}}
\newcommand{\Cag}[1]{\textcolor{black}{#1}}

\usepackage{comment}

\usepackage[caption=false]{subfig}
\usepackage{graphicx,epstopdf}
\usepackage{cleveref}
\usepackage[fleqn,tbtags]{mathtools}

\usepackage{footnote}
\makesavenoteenv{algorithm}
\usepackage{algorithm}
\usepackage{algorithmic}

\DeclareMathOperator*{\argmin}{arg\,min}

\newcommand{\N}{\mathbb{N}}
\newcommand{\R}{\mathbb{R}}

\renewcommand{\S}{\mathbb{S}}
\newcommand{\T}{\mathsf{T}} % for transposition
\newcommand{\X}{\mathcal{X}}
\newcommand{\D}{\mathcal{D}}
\renewcommand{\H}{\mathcal{H}}
\newcommand{\W}{\mathcal{W}}
\newcommand{\Wkn}{\mathcal{W}_{\text{known}}}

\newcommand{\Dkn}{\mathcal{D}_{\text{known}}} % This is new.
\newcommand{\Dunkn}{\mathcal{D}_{\text{unknown}}} % This is new.
\newcommand{\V}{\mathcal{V}}
\newcommand{\Vkn}{\mathcal{V}_{\text{known}}}
\newcommand{\Vunkn}{\mathcal{V}_{\text{unknown}}}
\renewcommand{\P}{\mathcal{P}}

\DeclareMathOperator{\wMin}{wMin}
\DeclareMathOperator{\Min}{Min}
\DeclareMathOperator{\Max}{Max}

\DeclareMathOperator{\bd}{bd}
\DeclareMathOperator{\Int}{int}

\DeclareMathOperator{\rec}{recc}
\DeclareMathOperator{\conv}{conv}
\DeclareMathOperator{\cl}{cl}
\DeclareMathOperator{\cone}{cone}
\DeclareMathOperator{\ri}{ri}
\DeclareMathOperator{\epi}{epi}
\DeclareMathOperator{\ind}{Ind}

\newcommand{\norm}[1]{\left\lVert#1\right\rVert}
\newcommand{\abs}[1]{\lvert#1\rvert}
\newtheorem{theorem}{Theorem}
\numberwithin{theorem}{section}
\newtheorem{proposition}[theorem]{Proposition}
\newtheorem{lemma}[theorem]{Lemma}
\newtheorem{corollary}[theorem]{Corollary}
\newtheorem{definition}[theorem]{Definition}
\newtheorem{remark}[theorem]{Remark}

\newtheorem{assumption}[theorem]{Assumption}

\newtheorem{example}[theorem]{Example}

\newcounter{algo}
\newcommand{\of}[1]{\ensuremath{\left( #1 \right)}}

\author{\c{C}a\u{g}in Ararat\thanks{Bilkent University, Department of Industrial Engineering, Ankara, 06800 Turkey, cararat@bilkent.edu.tr}
	\and 	S{\.i}may Tekg\"ul\thanks{The University of Edinburgh, School of Mathematics, Edinburgh, EH9 3FD United Kingdom, s.tekgul@ed.ac.uk} \and F{\.i}rdevs Ulus\thanks{Corresponding author, Bilkent University, Department of Industrial Engineering, Ankara, 06800 Turkey, firdevs@bilkent.edu.tr}}

\title{Geometric duality results and approximation algorithms for convex vector optimization problems}

\date{\today}

\begin{document}
	
	\maketitle
	
	% REQUIRED
	\begin{abstract}
		\rev{We study geometric duality for convex vector optimization problems. For a primal problem with a $q$-dimensional objective space, we formulate a dual problem with a $(q+1)$-dimensional objective space. Consequently, different from an existing approach, the geometric dual problem does not depend on a fixed direction parameter and the resulting dual image is a convex cone. We prove a one-to-one correspondence between certain faces of the primal and dual images. In addition, we show that a polyhedral approximation for one image gives rise to a polyhedral approximation for the other. Based on this, we propose a geometric dual algorithm which solves the primal and dual problems simultaneously and is free of direction-biasedness. We also modify an existing direction-free primal algorithm in a way that it solves the dual problem as well. We test the performance of the algorithms for randomly generated problem instances by using the so-called primal error and hypervolume indicator as performance measures.}
		
	\end{abstract}

	
	\section{Introduction}\label{sect:introduction}
	Vector optimization is a generalization of multiobjective optimization (MO) where the order relation over the objective vectors is determined by a general ordering cone. It has been applied in many fields including risk-averse dynamic programming \cite{birgitor}, financial mathematics \cite{nurtai,recursiveBellman,avar,superhedging}, economics \cite{rudloffulus,econometrica}, game theory \cite{birgitNash,hamellohnegame}.
	
	\rev{In vector optimization, generating efficient solutions which correspond to minimal elements in the objective space can be done either by solving single-objective optimization problems formed by the structure of the original problem, called scalarizations (see \cite{eichfelder2008adaptive} and references therein), or by iterative algorithms which work directly in the decision space (see e.g. \cite{augmentedlagrangian,projectedgradient,newton,inexactprojected,barriertype}).} The \rev{overall aim for many applications is however} to generate either (an approximation of) the whole set of minimal elements in the objective space \cite{benson_1998,psimplex} or the set of all efficient solutions in the decision space \cite{armand}. \rev{Since the dimension of the decision space is in general much higher than that of the objective space and consequently many efficient solutions might map into a single objective value, the former is usually a more affordable goal.} 
	
	By this motivation, Benson \cite{benson_1998} introduced an outer approximation algorithm in 1998 that aims to generate the Pareto frontier of linear multiobjective optimization problems. This has led to vast literature in vector optimization. In the linear case, many variations of Benson's algorithm have been proposed \cite{csirmaz_2016,ehrgott_2012,hamel_2014,lohne_2011,shaoehrgott_2008_1,shaoehrgott_2008_2}.  
	We discuss the \rev{objective space-based} algorithms in the literature for convex vector optimization problems (CVOPs) in more detail below. For non-convex vector optimization problems with special structure, we refer the reader to the recent works \cite{desantis,tradeoffs,bbeichfelder}.
	
	\subsection{\rev{Literature review on objective space-based CVOP algorithms}}
	For CVOPs, there are various approximation techniques in the literature\rev{, see e.g.} \cite{ruzika_2005}. \rev{Here, we focus on the approaches which generate approximations to the entire minimal set in the objective space.}
	In 2011, Ehrgott, Shao and Sch{\"o}bel \cite{ess_2011} proposed an extension of Benson's algorithm for convex MO problems. Later L{\"o}hne, Rudloff and Ulus \cite{ulus_2014} generalized this algorithm for CVOPs \rev{and recently, different variants have been proposed in \cite{dorfler2021benson,irem}.} \rev{These algorithms solve} a Pascoletti-Serafini \cite{ps_1984} scalarization and a vertex enumeration problem in each iteration. This scalarization problem depends on two parameters: a direction vector and a reference point. Specifically, given a point $v$ in the objective space, it determines the closest point (along \rev{a} direction $c$) to $v$ in the objective space as well as a corresponding feasible solution in the decision space. \rev{In \cite{ulus_2014}, the reference point $v$ is selected arbitrarily among the vertices of the current approximation and $c$ is fixed throughout the algorithms, whereas in \cite{dorfler2021benson}, the reference point $v$ and a corresponding direction vector $c$ are selected based on an additional procedure. In \cite{irem}, several additional rules for selecting the two parameters of the Pascoletti-Serafini scalarization for the same algorithmic setup are proposed and compared.}
	
	Recently, for CVOPs, Ararat, Ulus and Umer \cite{umer_2022} considered a norm-minimizing scalarization which requires only a reference point but no direction parameter, and proposed an outer approximation algorithm for CVOPs. As the proposed algorithm does not require a direction parameter, direction-biasedness is not a concern. However, the norm-minimizing scalarization has a nonlinear convex \rev{objective} function, which is not the case for the Pascoletti-Serafini scalarization that is used in the primal algorithm in \cite{ulus_2014}. 
	
	\rev{Apart from these \emph{primal} algorithms,} in \cite{ulus_2014}, the authors also proposed a \rev{\emph{geometric dual}} variant of their primal algorithm, which is based on a geometric duality theory. The general theory of convex polytopes indicates that two polytopes are dual to each other if there exists an inclusion-reversing one-to-one mapping between their faces \cite{grunbaum_2013}. Similar to the duality relation between polytopes, a duality relation between the polyhedral image of the primal problem and the polyhedral image of a dual problem was introduced in \cite{heydelohne_2008} for multiobjective linear programming problems. \rev{Later, Luc \cite{luc-duality} introduced \emph{parametric duality} for these problems and \Cag{studied} the relationship between the parametric and geometric duality notions. \Cag{More recently, the equivalence between the two duality notions has been shown in \cite{dorflerlohne_duality}.}} 
	
	\rev{In \cite{heyde_2013}, a general duality theory is developed for the epigraph of a closed convex function and that of its conjugate function. As a special case, this theory is applied for CVOPs to obtain a duality relation between the primal and dual images; hence, the geometric duality theory \rev{in \cite{heydelohne_2008}} is generalized to the non-polyhedral case.} The geometric dual algorithm that is proposed in \cite{ulus_2014} is based on the \rev{duality relation in \cite{heyde_2013}}. Accordingly, the fixed direction $c$ that is used within the primal solution concept and algorithm is also used in the design of the geometric dual problem and the algorithm. \rev{Different from the primal algorithms, in the geometric dual algorithm of \cite{ulus_2014}, instead of Pascoletti-Serafini or norm-minimizing scalarizations, only the well-known weighted sum scalarizations, which may be easier to deal with because of their simple structure, are solved.}

	\subsection{\rev{The proposed approach and contributions}}
	In this paper, we establish a \emph{direction-free} geometric duality theory and a dual algorithm to solve CVOPs. \rev{More precisely, after recalling the primal solution concept in \cite{umer_2022}, we propose a direction-free geometric dual problem and a solution concept for it. Then, we prove a geometric duality relation between the images of the primal and dual problems, which is based on an inclusion-reversing one-to-one duality mapping. The proof of this relation is based on the general duality theory for epigraphs in \cite{heyde_2013}. We also prove that a polyhedral approximation for the primal image yields a polyhedral approximation for the dual image, and vice versa.} Different from \cite{heyde_2013}, the proposed geometric dual image does not depend on a fixed direction; but in order to handle this issue, the dimension of the objective space for the dual problem is increased by one. \rev{In this sense, it can also be seen as a generalization of the parametric duality for linear MO from \cite{luc-duality} to CVOPs.} Accordingly, the image of the proposed dual problem is not only a convex set but a convex cone. Due to this conic structure of the dual image, the dimension increase compared to \cite{heyde_2013} is not an additional source of computational burden, as also confirmed by the numerical results.
	
	\rev{Based on the geometric duality theory, we propose a dual algorithm which solves the primal and dual problems simultaneously by solving only weighted sum scalarization problems. More precisely, the algorithm gives a finite $\epsilon$-solution to the dual problem; moreover, it gives a finite weak $\tilde{\epsilon}$-solution to the primal problem, where $\tilde{\epsilon}$ is determined by $\epsilon$ and the structure of the underlying ordering cone. We also modify the primal algorithm in \cite{umer_2022} in a way that it returns a finite $\epsilon$-solution to the dual problem as well.}
	
	\rev{We compare the proposed geometric duality and the dual algorithm with the ones in \cite{heyde_2013} and \cite{ulus_2014}, respectively. In particular, we show that the proposed dual image and the dual image in \cite{heyde_2013} can be recovered from each other. Moreover, by fixing a suitable norm, we show that the dual algorithm in \cite{ulus_2014} can be seen as a special case of the proposed dual algorithm.} 
	
	\rev{Finally, we test the performance of the proposed algorithm in comparison to the existing ones using two performance measures: primal error (PE) and hypervolume indicator (HV). While PE simply measures the Hausdorff distance between the primal image and its returned polyhedral approximation, novel to this work, we define HV as a hypervolume-based performance metric for convex vector optimization. Our definition is \Cag{similar to hypervolume-based metrics} for MO, see e.g. \cite{auger_2012,zitzler2007hypervolume,zitzler1999multiobjective,zitzler2003performance}. The computational results suggest that the proposed dual algorithm has promising performance.}
	
	The rest of the paper is structured as follows. \Cref{sec:preliminaries} presents the basic concepts and notation. \Cref{sec:problems} introduces the primal problem, dual problem, and solution concepts for these problems. The geometric duality between the primal and dual problems is studied in \Cref{sec:geomdual}. \rev{\Cref{sec:algorithms} provides the primal and dual algorithms. In \Cref{sect:relation}, we compare the proposed approach with the existing ones in the literature. In \Cref{sec:numres}, we provide some test results for performance comparisons.}
	
	\section{Preliminaries}\label{sec:preliminaries}
	In this section, we provide the definitions and notations used throughout the paper. Let $q\in\N\coloneqq\{1,2,\ldots\}$ be a positive integer. We denote the $q$-dimensional Euclidean space by $\R^q$. Let $\norm{\cdot}$ be an arbitrary norm on $\R^q$. The associated dual norm is denoted by $\norm{\cdot}_\ast $, that is, $\norm{z}_\ast=\textnormal{sup}\{z^{\T}x \mid\norm{x}\leq1\}$ for each $z\in\R^q$. Throughout, $B(0,\epsilon):=\{z \in \mathbb{R}^q \mid \|z\|\leq\epsilon \}$ denotes the norm ball centered at $0\in \R^q$ with radius $\epsilon >0$.
	
	For a set $A \subseteq \R^q$, we denote the convex hull, conic hull, interior, relative interior and closure of $A$ by $\conv A$, $\cone A$, $\Int A$, $\ri A$ and $\cl A$, respectively. Recall that $\cone A:=\{\lambda a \mid a\in A, \lambda \geq 0\}$. \rev{The closed convex cone defined by $A^+=\{y\in \R^q\mid \forall a\in A\colon y^{\T}a\geq 0\}$ is called the \emph{dual cone} of $A$. It is well-known that the dual cone of $A^+$ is given by $A^{++}\coloneqq (A^+)^+=\cl\cone\conv A$ whenever $A$ is nonempty.} The \emph{recession cone} of $A$ is defined by $\rec A =\{c\in\R^q \mid \forall \lambda\geq0,  a\in A \colon a+\lambda c \in A\}$. An element $d \in \rec A$ is a \emph{(recession) direction} of $A$. Let $A\subseteq\R^q $ be convex and $F\subseteq A$ be a convex subset. If $\lambda y^1+(1-\lambda) y^2\in F$ for some $0 < \lambda <1$ holds only if both $y^1$ and $y^2$ are elements of $F$, then $F$ is called a \textit{face} of $A$. A zero-dimensional face is called an \emph{extreme point} \rev{and} a one-dimensional face is called an \emph{edge} \rev{of A. If $A$ is $q$-dimensional, then} a $(q-1)$-dimensional face is called a \emph{facet} of \rev{$A$}. A face of $A$ that is not the empty set and not $A$ itself is called a \emph{proper face} of $A$. \rev{We call $F$ an \emph{exposed face} of $A$ if it can be written as the intersection of $A$ and a supporting hyperplane of $A$. If $A$ is $q$-dimensional, then an exposed face of $A$ is also a proper face; the converse does not hold in general \cite[Section 2.2]{heyde_2013}.}
	
	Given nonempty sets $A,B\subseteq\R^q$, their Minkowski sum is defined by $A+B\coloneqq \{a+b\mid a\in A, b\in B\}$. For $\lambda \in\R$, we also define $\lambda A\coloneqq \{\lambda a\mid a\in A\}$. In particular, we have $A-B = A + (-1)B$.
	
	Let $C\subseteq \R^q$ be a cone. \rev{It} is called \emph{pointed} if \rev{$C\cap -C = \{0\}$}, \emph{solid} if it has nonempty interior, and \emph{nontrivial} if $C \neq\emptyset$ and $C\neq\R^q$. If $C \subseteq \R^q$ is a convex pointed cone, then the relation $\leq_C\coloneqq \{(x,y)\in\R^q\times\R^q \mid y-x\in C\} $ is an antisymmetric partial order on $\R^q$ \cite[Theorem 1.18]{jahn_2004}; we write $x\leq_C y$ whenever $(x,y)\in\leq_C$.
	
	Let $m\in\N$ and $\X\subseteq \R^m$ be a nonempty convex set. A function $f\colon\X \to \R^q$ is said to be \emph{$C$-convex on} $\X$ if
	$f(\lambda x^1+(1-\lambda)x^2)\leq_C \lambda f(x^1)+(1-\lambda)f(x^2)$ for every $x^1,x^2\in\X$ and $\lambda\in [0,1]$ \cite[Definition 2.4]{jahn_2004}. \rev{Given a function $g\colon \R^q\to \R$, the function $g^\ast\colon \R^q\to [-\infty,+\infty]$ defined by $g^\ast (w)\coloneqq\sup_{z\in\R^q}(w^\T z-g(z))$, $w\in\R^q$, is called the \emph{conjugate function} of $g$. For a set $A\subseteq\R^q$, the function $I_A$ defined by $I_A(z)\coloneqq 0$ for $z\in A$, and by $I_A(z)\coloneqq +\infty$ for $z\notin A$ is called the \emph{indicator function} of $A$. In this case, taking $g=I_A$ gives $g^\ast(w)=\sup_{z\in A}w^\T z$ for each $w\in \R^q$; $g^\ast$ is called the \emph{support function} of $A$; we also define the \emph{polar} of $A$ as the set $A^\circ \coloneqq  \{w\in \R^q\mid g^\ast (w)\leq 1\}$.}
	
	Let $A\subseteq \R^q $ and let $C\subseteq\R^q $ be a closed convex pointed cone. The sets
	$\Min_C A\coloneqq \{y\in A\mid (\{y\}-C\setminus\{0\})\cap A=\emptyset\}$, $\wMin_C A\coloneqq \{y\in A\mid (\{y\}-\Int C)\cap A=\emptyset\}$,
	$\Max_C A\coloneqq \{y\in A\mid (\{y\}+C\setminus\{0\})\cap A=\emptyset\}$ are called the sets of \emph{$C$-minimal}, \emph{weakly $C$-minimal}, \emph{$C$-maximal} elements of $A$, respectively \cite[Definition 1.41]{lohne_2011}. \rev{An exposed} face of $A $ that only consists of (weakly) $C$-minimal elements is called a \emph{(weakly) $C$-minimal \rev{exposed} face}. \rev{An exposed} face of $A$ that consists of only $C$-maximal elements is called a \emph{$C$-maximal \rev{exposed} face} \cite[Section 4.5]{lohne_2011}.

	\begin{remark}\label{rem:w_inC+}
		Let $A\subseteq \R^q$ be a nonempty convex set and $C\subseteq \R^q $ be a convex cone. If $A=A+C$ and $w\in \R^q$ such that $\inf_{a\in A}w^{\T}a>-\infty$, then $w\in C^+$. Note that $A=A+\rec A$ holds \cite[Theorem 5.6]{soltan_2015}. 
		If $A\subseteq H\coloneqq\{z\in\R^q\mid w^{\T}z\geq r \} $ for some $w\in \R^q $ and $r\in\R$, then we have $w\in (\rec A)^+ $. 
	\end{remark}
	
	\section{Primal and dual problems}\label{sec:problems}
	In this paper, we consider a convex vector optimization problem and its geometric dual. The primal problem is defined as
	\begin{equation}\tag{P}\label{P}
	\text{minimize } f(x) \text{ with respect to } \leq_C \text{subject to } x \in \X,
	\end{equation}
	where the ordering cone $C\subseteq\R^q $ is nontrivial, pointed, solid, closed and convex; the vector-valued objective function $f\colon \X\to\R^q $ is $C$-convex and continuous; and the feasible set $\rev{\emptyset \neq} \X\subseteq \R^{m}$ is compact and convex. The \emph{upper image} of \eqref{P} is defined as
	\[
	\P \coloneqq \cl(f(\X)+C),
	\]
	where $f(\X)\coloneqq \{f(x)\mid x\in\X\}$ is the image of $\X$ under $f$. The following proposition collects some basic facts about the problem structure. Its proof is straightforward \rev{from the fact that $\X$ is compact}, hence we omit it. 
	
	\begin{proposition}\label{prop:upperimage}
		The upper image $\P$ is a closed convex set, the image $f(\X)$ is a compact set, and it holds
		$\P=f(\X)+C$. Moreover, the primal problem \eqref{P} is \emph{bounded} in the sense that $\{y\}+C\subseteq\P$ for some $y\in\R^q$.
	\end{proposition}
	
	For a parameter vector $w \in C^+ $, the convex program
	\begin{equation*}\tag{WS$(w)$}\label{WS(w)}
	\text{minimize } w^{\T}f(x)  \text{ subject to } x \in \X
	\end{equation*}
	is called the \emph{weighted sum scalarization} of \eqref{P}. Let $p^w $ be the optimal value of \eqref{WS(w)}, that is, 
	%\begin{equation} \label{eq:pw}
	$p^w:=\inf_{x\in\X} w^{\T}f(x)$. 
	%\end{equation} 
	Since $\X$ is a \rev{nonempty} compact set and $f$ is a continuous function, it follows that $p^w\in\R$. The next proposition is a well-known result that will be used in the design of the geometric dual problem.
	\begin{proposition}\label{prop:WS}
		\cite[Corollary 5.29]{jahn_2004} Let $w\in C^+\setminus \{0\}$. Then, an optimal solution $x^w$ of \eqref{WS(w)} is a weak minimizer of \eqref{P}. The converse also holds: for each weak minimizer $x\in \X$ of \eqref{P}, there exists $w \in C^+\setminus\{0\} $ such that $x$ is an optimal solution of \eqref{WS(w)}.
	\end{proposition}
	
	Now, let us define the geometric dual problem of \eqref{P} as
	\begin{equation*}\tag{D}\label{D}
	\text{maximize } \xi(w) \text{ with respect to } \leq_K \text{subject to } w\in \W.
	\end{equation*}
	In this problem, the objective function $\xi\colon\R^q\to\R^{q+1}$ is defined by 
	%\begin{equation}\label{eq:dualObj}
	%\xi(w)\coloneqq \Big(w_1, \dots, w_q, \inf_{x\in \X} w^{\T}f(x)\Big)^{\T},\quad w\in \W;
	%\end{equation}
	\begin{equation}\label{eq:dualObj}
	\xi(w)\coloneqq (w_1, \dots, w_q, p^w)^{\T},\quad w\in \W;
	\end{equation}
	the ordering cone $K$ is defined by $
	K\coloneqq \cone\{e^{q+1}\}=\{\lambda e^{q+1} \mid \lambda \geq 0 \}$,
	where $e^{q+1}=(0,\dots ,0,1)^{\T}\in \R^{q+1}$; and the feasible set is $\W\coloneqq C^+$. The \emph{lower image} of \eqref{D} is defined as
	\[
	\D\coloneqq \xi(\W)-K=\{(w^{\T},\alpha)^{\T} \in \R^{q+1} \mid w \in \W, \ \alpha \leq p^w\}.
	\] \begin{remark}\label{rem:D_alternative}
		\rev{Note that the decision space of the dual problem has dimension $q$, which in general is much less than $n$, the number of variables of the primal problem. {However, the dual objective function involves solving another optimization problem. Indeed,} if the feasible region $\X$ of \eqref{P} is given by explicit constraints, then it is also possible to define a geometric dual problem which includes additional dual variables corresponding to these explicit constraints. In particular, the last component of the dual objective function can be defined using the Lagrangian of \eqref{WS(w)} instead of its value directly. This construction would lead to the same lower image, see also \cite[Remark 3.6]{ulus_2014}.}
	\end{remark}
	The following proposition follows from the definition of $\D$, we omit its proof.
	
	\begin{proposition}\label{prop:Dcone}
		The lower image $\D$ is a closed convex cone.
	\end{proposition}
	%\begin{proof}
	%	Let $((w^i)^{\T},\alpha_i)^{\T}\in\D$ and $\lambda_i\geq 0$ for each $i\in\{1,\dots,n\} $, where $n\in\mathbb{N}$. Since $\W=C^+ $ is a convex cone and we have $w^i\in C^+ $ for each $i\in\{1,\dots,n\}$, we have $\sum_{i=1}^{n}\lambda_iw^i \in \W $. Moreover, we have $\alpha_i\leq\inf_{x\in\X} (w^i)^{\T}f(x)$ for each $i\in\{1,\ldots,n\}$, which implies that
	%	\begin{equation*}
	%	\sum_{i=1}^{n}\lambda_i\alpha_i \leq \sum_{i=1}^{n}\lambda_i \inf_{x\in\X}(w^i)^{\T}f(x) \leq \inf_{x\in\X}\of{\sum_{i=1}^{n}\lambda_iw^i}^{\T}f(x).
	%	\end{equation*}
	%	Hence, $\sum_{i=1}^{n}\lambda_i((w^i)^{\T},\alpha_i)^{\T}\in \D$. It follows that $\D$ is a convex cone.
	
	%\Fir{(Remark'ı buraya ekledim.)}\\
	%	Note that the function $w\mapsto p^w$ is continuous as a concave function on $\mathcal{W}$ with finite values. Moreover, $\D \subseteq \W\times\R$ coincides with the hypograph of this function. Hence, $\D$ is a closed set. %Using this property, it can be shown that the lower image $\mathcal{D} $ is a closed set \cite[Theorem 4.13]{fuente_2000}. \Cag{(Honestly, I don't see the relationship between the two.) \Fir{(Ben de direkt bağlantı kuramadım.)}} 
	%\end{proof}
	
	We define exact and approximate solution concepts for the primal problem \eqref{P}. %Let us start with the following definition.
	
	\begin{definition}\label{def:solnfull}
		\cite[Definition 2.20, Proposition 4.7]{lohne_2011} A point $\bar{x} \in \X$ is said to be a \emph{(weak) minimizer} for \eqref{P} if $f(\bar{x})$ is a (weakly) C-minimal element of $f(\X)$. A nonempty set $\bar{\X}\subseteq\X$ is called an \emph{infimizer} of \eqref{P} if 
		$\cl \conv (f(\bar{\X})+C)=\P$.	An infimizer $\bar{\X}$ of \eqref{P} is called \emph{(weak) solution to \eqref{P}} if it consists of only (weak) minimizers.
	\end{definition}
	
	\rev{Since} the upper image of a convex vector optimization problem is %\out{considered}
	\rev{not a polyhedral set in general}, a finite set $\bar{\X}$ \rev{may} not satisfy the exact solution concept in \Cref{def:solnfull}. Hence, we give an approximate solution concept for a fixed $\epsilon>0$ below. %The property given in the following definition is necessary both for that solution concept and also for the construction of the primal algorithm in Section \ref{sec:algorithms}. \Fir{(Bu son cümle ile neyi kastediyoruz? Çıkartsak da olur mu?)}
	%\begin{definition}\label{def:1}
	%	\eqref{P} is said to be \emph{bounded} if $\mathcal{P}\subseteq\{y\}+C$ for some $y\in\mathbb{R}^q$.
	%\end{definition}
	%In this paper, we consider bounded convex vector optimization problems. The following definition gives a solution concept for such problems. 
	
	\begin{definition}\label{def:solnPrimal}
		\cite[Definition 3.5]{umer_2022} A nonempty finite set $\bar{\X}\subseteq\mathcal{X}$ is called a \emph{finite $\epsilon$-infimizer} of \eqref{P} if
		$\conv f(\bar{\X})+C+B(0,\epsilon) \supseteq \P$. A finite $\epsilon$-infimizer $\bar{\X}$ of \eqref{P} is called a \emph{finite (weak) $\epsilon$-solution} to \eqref{P} if it consists of only (weak) minimizers.
	\end{definition}
	
	Note that if $\bar{\X}$ is a finite (weak) $\epsilon$-solution, then we have the following inner and outer approximations of the upper image:
	\begin{equation*}
	\conv f(\bar{\X})+C+B(0,\epsilon) \supseteq \P \supseteq \conv f(\bar{\X})+C.
	\end{equation*}
	
	Now, similar to \Cref{def:solnfull}, we provide an exact solution concept for the dual problem \eqref{D}.
	
	\begin{definition}\cite[Definition 2.53, Corollary 2.54]{lohne_2011}\label{def:solnfull2}
		A point $\bar{w} \in \W$ is called a \emph{maximizer} for \eqref{D} if $\xi(\bar{w})$ is a $K$-maximal element of $\xi(\W)$. A nonempty set $\bar{\W} \subseteq \W $ is called a \emph{supremizer} of \eqref{D} if
		$\cone\conv\xi(\bar{\W})-K=\D$.
		A supremizer $\bar{\W}$ of \eqref{D} is called a \emph{solution} to \eqref{D} if it consists of only maximizers.
	\end{definition}
	
	As for the upper image, in general, the lower image cannot be represented by a finite set $\bar{\W}$ using the exact solution concept in \Cref{def:solnfull2}. In the next definition, we propose a novel approximate solution that is tailor-made for the lower image $\D$, which is a convex cone; see \Cref{rem:solnDual} below for the technical motivation.
	
	\begin{definition}\label{def:solnDual} 
		A nonempty finite set $\bar{\W}\subseteq\W\cap \S^{q-1} $ is called a \emph{finite $\epsilon$-supremizer} of \eqref{D} if
		%	\begin{equation}\label{eq:sup}
		$\cone(\conv \xi(\bar{\W})+\epsilon \{e^{q+1}\} )- K  \supseteq \D$,
		%	\end{equation}
		where $\S^{q-1}\coloneqq \{z\in\R^q\mid \norm{z}_\ast=1\}$ denotes the unit sphere in $\R^q$ with respect to the dual norm. A finite $\epsilon$-supremizer $\bar{\W} $ of \eqref{D} is called a \emph{finite $\epsilon$-solution} to \eqref{D} if it consists of only maximizers.
	\end{definition}
	
	If $\bar{\mathcal{W}}$ is a finite $\epsilon$-solution of \eqref{D}, then one obtains the following inner and outer approximations of the lower image:
	\begin{equation*}
	\cone(\conv \xi(\bar{\W})+\epsilon \{e^{q+1}\} )
	- K   \supseteq \D \supseteq  \cone \conv \xi(\bar{\W}) - K.
	\end{equation*}
	
	\begin{remark}\label{rem:solnDual}
		Let us comment on the particular structure of \Cref{def:solnDual}. %First, we add the error term $\epsilon(e^{q+1}) $ to the set $\xi(\bar{\W})$ of objective values produced by the candidate solution $\bar{\W}$. 
		Since the lower image $\D$ is a convex cone by \Cref{prop:Dcone}, we evaluate the conic hull of the Minkowski sum $\conv \xi(\bar{\W})+\epsilon \{e^{q+1}\}$ so that the error values are scaled properly. With this operation, we ensure that the resulting conic hull is comparable with $\D$ (up to the subtraction of the ordering cone $K$).
	\end{remark}
	
	%\Cag{(What do you think about moving the part about weighted sum scalarization (except the next proposition) right before the dual problem? Then, we can use the $p^w$ notation in the definition of $\xi(w)$, etc. Also the weighted sum scalarization is immediately related to the primal problem as well. Not sure.)}\Fir{(Sounds better so I have arranged it that way.)}
	
	The next proposition will be used later to prove some geometric duality results.
	
	\begin{proposition}\label{prop:WS_Kmax}
		Let $w \in \W\rev{\setminus\{0\}}$. Then, $\xi(w)$ is a $K$-maximal element of $\D$. 
	\end{proposition}
	\begin{proof}
		Let $\varepsilon>0$. We prove that $\xi(w)+\varepsilon e^{q+1} \notin \D=\xi(\W)-K$. To get a contradiction, assume the existence of $\bar{w} \in \W$ and $\bar{\varepsilon}\geq0$ with $\xi(w)+\varepsilon e^{q+1}=\xi(\bar{w})-\bar{\varepsilon}e^{q+1}$. By %\eqref{eq:pw} and 
		\eqref{eq:dualObj}, we have
		$
		(w_1,\dots, w_q, p^w)^{\T} + \varepsilon e^{q+1}
		=(\bar{w}_1, \dots, \bar{w}_q, p^{\bar{w}})^{\T} -\bar{\varepsilon}e^{q+1}$,
		that is, $
		(w_1,\dots, w_q, p^w+\varepsilon )^{\T} 
		=(\bar{w}_1, \dots, \bar{w}_q, p^{\bar{w}}- \bar{\varepsilon})^{\T}$.
		Hence, $\inf_{x\in \X} w^{\T}f(x) = \inf_{x\in \X} \bar{w}^{\T}f(x)$ and $\varepsilon=-\bar{\varepsilon}$. This contradicts $\varepsilon>0$ and $\bar{\varepsilon}\geq0$. Hence, $\xi(w)+\varepsilon e^{q+1} \notin \D$. Since $\varepsilon>0$ is arbitrary, $\xi(w)$ is a $K$-maximal element of $\D$.
	\end{proof}
	
	\section{Geometric duality}\label{sec:geomdual}
	
	In this section, we will investigate the duality relation between the primal and dual problems, \eqref{P} and \eqref{D}. First, we will provide the main duality theorem which relates the weakly $C$-minimal \rev{exposed} faces of $\P$ and the $K$-maximal \rev{exposed} faces of $\D$. Then, we will establish some duality properties regarding $\P$ and $\D$, as well as their polyhedral approximations.
	
	\subsection{Geometric duality between $\P$ and $\D$}\label{subsec:geomdual_1}
	
	We will prove that there is a one-to-one correspondence between the set of all weakly $C$-minimal \rev{exposed} faces of the upper image $\P$ and the set of all $K$-maximal \rev{exposed} faces of the lower image $\D$. We will also show that the upper and lower images can be recovered from each other.
	
	%Moreover, some duality properties regarding polyhedral approximations of these sets will be investigated. Throughout, we call a set $A\subseteq \R^q$ ($A\subseteq \R^{q+1}$) an \emph{upper set} (a \emph{lower set}) if $A=A+C$ ($A = A-K$). \Fir{Bu tanımları kullanmışız bu section'da. Kullanmamaya karar verirsek son cümleyi sileriz.})
	
	Let us start by defining 
	\begin{equation}\label{eq:varphi}
	\varphi\colon \R^q \times \R^{q+1}\to\R,\quad  \varphi(y,w,\alpha)\coloneqq w^{\T}y-\alpha,
	\end{equation}
	and the following set-valued maps: 
	\begin{align}
	&\H, H\colon \R^{q+1}\rightrightarrows \R^q, \H(w,\alpha)\coloneqq \{y \in \R^q\mid \varphi(y,w,\alpha)\geq 0\},  H(w,\alpha)\coloneqq \bd \H(w,\alpha),\notag \\
	&\H^{\ast},H^\ast\colon\R^q\rightrightarrows \R^{q+1}, \H^\ast(y)\coloneqq\{(w^{\T},\alpha)^{\T} \in \R^{q+1}\negthinspace\mid\negthinspace\varphi(y,w,\alpha) \geq 0 \},
	H^\ast(y)\coloneqq \bd \H^\ast(y).\notag %= \{(w^\T,\alpha)^{\T} \in \R^{q+1}\mid\varphi(y,w,\alpha) = 0 \}.\notag 
	\end{align}
	%Notice that \Fir{for any $w, y\in \R^q$ and $\alpha \in \R$,} $H(w,\alpha)=\textnormal{bd }\mathcal{H}(w,\alpha) $ and $ H^*(y)=\textnormal{bd }\mathcal{H}^*(y) $ holds. 
	These functions are essential to define a duality map between $\P$ and $\D$. Moreover, for arbitrary $y \in \R^q $ and $(w^{\T},\alpha)^{\T}\in \R^{q+1} $, we have the following statement:
	\begin{equation}\label{eqn:H_H*}
	(w^{\T},\alpha)^{\T}\in H^\ast(y) \iff y\in H(w,\alpha).
	\end{equation}

	With the next proposition we show that, by solving a weighted sum scalarization, one finds supporting hyperplanes to the upper and lower images.
	
	\begin{proposition}\label{prop:supp_halfsp}
		Let $w\in \W\rev{\setminus\{0\}}$ and $x^w$ be an optimal solution to \eqref{WS(w)}. Then, $H(\xi(w))$ is a supporting hyperplane of $\P$ at $f(x^w)$ such that $\P\subseteq \H(\xi(w))$, and $H^\ast(f(x^w))$ is a supporting hyperplane of $\D$ at $\xi(w)$ such that $\D\subseteq \mathcal{H}^\ast(f(x^w))$.
		
	\end{proposition}
	
	\begin{proof}
		To prove the first statement, we show that $f(x^w) \in H(\xi(w)) $ and $\P\subseteq \H(\xi(w))$. As $\xi(w)=(w^{\T},w^{\T}f(x^w))^{\T} $, we have $\varphi(f(x^w),\xi(w)) = %w^{\T}f(x^w)-w^{\T} f(x^w) 
		0$, which implies that $f(x^w) \in H(\xi(w)) $. To prove $\P\subseteq \H(\xi(w))$, let $x \in \X$, $c \in C$ be arbitrary. Then, $\varphi(f(x)+c,\xi(w))=w^ {\T}f(x)+w^{\T}c-w^{\T}f(x^w)\geq 0$ as $w \in C^+$ and $x^w$ is an optimal solution of \eqref{WS(w)}. Therefore, $\H(\xi(w)) \supseteq  f(\X)+C =\P$, where the equality follows by \Cref{prop:upperimage}. On the other hand, having $\varphi(f(x^w),\xi(w))=w^{\T}f(x^w)-w^{\T}f(x^w)=0$ also implies that $\xi(w) \in H^\ast(f(x^w))$. To complete the proof, it is enough to show that $\D\subseteq \H^\ast(f(x^w))$. Let $(w^{\T},\alpha)^{\T}\in \D$ be arbitrary. It follows from the definition of $\D$ that $\alpha\leq \inf_{x\in \X}w^{\T}f(x)=w^{\T}f(x^w)$. Hence, $\varphi(f(x^w),w,\alpha)=w^{\T}f(x^w)-\alpha\geq 0$, that is, $(w^{\T},\alpha)^{\T}\in \H^\ast(f(x^w))$.
	\end{proof}
	
	The following \rev{propositions} show the relationship between the weakly $C$-minimal elements of $\P$ and the $K$-maximal \rev{exposed} faces of $\D$, as well as that between the $K$-maximal elements of $\D$ and the weakly $C$-minimal \rev{exposed} faces of $\P$. \rev{Their proofs are based on elementary arguments, hence omitted for brevity.}
	
	%\Fir{(İki lemmayı birleştirdim. Lemma 4.3 ve .4.4 de birleştirilebilir. Proof'ları appendix'e koyabiliriz belki?)}
	
	\begin{proposition}\label{lem:duality_1}
		(a) Let $y\in\R^q$. Then, $y$ is a weakly $C$-minimal element of $\P$ if and only if $H^{\ast}(y)\cap\D$ is a $K$-maximal \rev{exposed} face of $\D$.
		%\item
		(b) For every $K$-maximal \rev{exposed} face $F^\ast$ of $\D$, there exists some $y\in \P$ such that $F^\ast=H^\ast(y)\cap\D$. 
	\end{proposition}

	\begin{proposition}\label{lem:duality_2}
		%	The following statements hold true:
		%	\begin{enumerate}[(a)]
		%	\item 
		(a) Let $(w^{\T},\alpha)^{\T}\in\R^{q+1}\rev{\setminus\{0\}}$. Then, $(w^{\T},\alpha)^{\T}$ is a $K$-maximal element of $\D$ if and only if $H(w,\alpha)\cap\P$ is a weakly $C$-minimal \rev{exposed} face of $\P$ satisfying $\H(w,\alpha)\supseteq\P$.
		%	\item 
		(b) For every $C$-minimal \rev{exposed} face $F$ of $\P$, there exists some $(w^{\T},\alpha)^{\T}\in \D$ such that $F=H(w,\alpha)\cap\P $. %\Fir{(Teorem 4.4 için $(w^{\T},\alpha)^{\T}\in\D$ olduğunu göstermemiz gerekiyor sanıyorum. O yüzden ifadeyi değiştirip ispatı ekledim.)}
		%\end{enumerate}
	\end{proposition}
	
	We now proceed with the main geometric duality result. Let $\mathcal{F}_{\P} $ be the set of all weakly $C$-minimal \rev{exposed} faces of $\P$ and $\mathcal{F}_{\D}^\ast$ be the set of all $K$-maximal \rev{exposed} faces of $\D$. Consider the set-valued function
	\begin{equation}\label{eqn:Psi}
	\Psi \colon \mathcal{F}_{\D}^\ast \rightrightarrows \R^q, \quad  \Psi(F^\ast)\coloneqq\bigcap\limits_{(w^{\T},\alpha)^{\T}\in F^\ast} H(w,\alpha)\cap\P.
	\end{equation}
	
	\begin{theorem}\label{thm:duality}
		$\Psi$ is an inclusion-reversing one-to-one correspondence between \rev{$\mathcal{F}^\ast_{\D}\setminus\{\{(0^\T,0)^\T\}\}$} and $\mathcal{F}_{\P}$. The inverse map is given by 
		%\begin{equation*}
		%\begin{aligned}
		$\Psi^\ast(F)\coloneqq\bigcap_{y\in F} H^\ast(y)\cap\D$.
		%\end{aligned}
		%\end{equation*}
	\end{theorem}

	\begin{proof}
		We use the geometric duality theory for the epigraphs of closed convex functions developed in \cite{heyde_2013}. To be able to use this theory, we express $\P$ and $\D$ as the epigraphs of closed convex functions, up to transformations. Observe that
		\[
		\D\negthinspace=\negthinspace\{(w^\T,\alpha)^\T\in \R^{q+1}\negthinspace \mid p^w\geq \alpha, w\in C^+\}\negthinspace =\negthinspace -\{(w^\T,\alpha)^\T\in \R^{q+1}\negthinspace \mid \tilde{p}(w)\leq \alpha\}\negthinspace =\negthinspace -\epi \tilde{p},
		\]
		where $\tilde{p}\colon \R^{q}\to \R\cup\{+\infty\}$ is %\out{defined by $\tilde{p}(w)\coloneqq \sup_{x\in\X}w^\T f(x)$ for $w\in - C^+$, and by $\tilde{p}(w)\coloneqq +\infty$ for $w\notin -C^+$. Note that $\tilde{p}$ is indeed} 
		the support function of $\P$, i.e., $\tilde{p}(w)=\sup_{y\in\P}w^\T y$ for every $w\in\R^q$. Then, the well-known duality between support and indicator functions (\cite[Theorem 13.2]{rockafellar_1970}) yields that $\tilde{p}=g^\ast$, where $g\coloneqq I_{\P}$ is the indicator function of $\P$. We also have $\epi g = \P\times\R_+$ and \Cag{$\P=P(\epi g)$, where $P\colon 2^{\R^{q+1}}\rightrightarrows\R^q$ is defined by $P(F^\ast)=\{y\in\R^{q}\mid (y^\T,0)^\T\in F^\ast\}$ for $F^\ast\subseteq \R^{q+1}$.} Let us define a set-valued function $\tilde{\Psi}\colon 2^{\R^{q+1}}\rightrightarrows \R^{q+1}$ by
		\[
		\tilde{\Psi}(F^\ast)\coloneqq\bigcap_{(w^\T,g^\ast(w))^\T\in F^\ast}\{(y^\T,g(y))^\T\in \R^q\times \R\mid w\in \partial g(y)\}.
		\]
		Then, by \cite[Theorem 3.3]{heyde_2013}, $\tilde{\Psi}$ is an inclusion-reversing one-to-one correspondence between the set of all $K$-minimal exposed faces of $\epi g^\ast = -\D$ and the set of all $K$-minimal exposed faces of $\epi g = \P\times\R_+$. Clearly, a $K$-minimal exposed face of $-\D$ is of the form $-F^\ast$, where $F^\ast\in \mathcal{F}_{\D}^\ast$; the converse holds as well. The one-to-one correspondence here is inclusion-preserving.

		Next, we show that \Cag{$P(\tilde{F})\in \mathcal{F}_\P\cup \{\P\}$} whenever $\tilde{F}$ is a $K$-minimal exposed face of $\epi g$. Note that a point $(y^\T,r)^\T\in\R^{q+1}$ is a $K$-minimal element of $\epi g=\P\times\R_+$ if and only if $y\in\P$ and $r=0$. Let $\tilde{F}$ be a $K$-minimal exposed face of $\epi g$. The previous observation implies that $\tilde{F}\subseteq \P\times\{0\}$. By definition, we may write $\tilde{F}=(\P\times\R_+)\cap \tilde{H}$ for some supporting hyperplane $\tilde{H}\subseteq \R^{q+1}$ of $\P\times\R_+$. Then, \Cag{$P(\tilde{F})=P(\P\times\R_+)\cap P(\tilde{H})=\P\cap P(\tilde{H})$.} By \cite[Lemma 3.1(ii)]{heyde_2013}, there exist $w\in \R^q$ and $\alpha\in\R$ such that $\tilde{H}=\{(y^\T,r)\in\R^{q+1}\mid w^\T y - r = \alpha\}$. If $w=0$, then $\tilde{F}\subseteq \P\times\{0\}$ implies that $\alpha=0$ and we have $\tilde{H}=\R^q\times\{0\}$ and \Cag{$P(\tilde{F})=\P$.} Suppose that $w\neq 0$. \Cag{Then, $P(\tilde{H})\neq \R^q$ is a supporting hyperplane of $P(\tilde{F})$ since $\tilde{H}$ is a supporting hyperplane of $\tilde{F}$. It follows that $P(\tilde{F})=\P\cap P(\tilde{H})\in\mathcal{F}_\P$.}
		
		%\out{Next, we show that $P[\tilde{F}]\in \mathcal{F}_\P\cup \{\P\}$ whenever $\tilde{F}$ is a $K$-minimal exposed face of $\epi g$. Note that a point $(y^\T,r)^\T\in\R^{q+1}$ is a $K$-minimal element of $\epi g=\P\times\R_+$ if and only if $y\in\P$ and $r=0$. Let $\tilde{F}$ be a $K$-minimal exposed face of $\epi g$. The previous observation implies that $\tilde{F}\subseteq \P\times\{0\}$. By definition, we may write $\tilde{F}=(\P\times\R_+)\cap \tilde{H}$ for some supporting hyperplane $\tilde{H}\subseteq \R^{q+1}$ of $\P\times\R_+$. Then, $P[\tilde{F}]=P[\P\times\R_+]\cap P[\tilde{H}]=\P\cap P[\tilde{H}]$. By \cite[Lemma 3.1(ii)]{heyde_2013}, there exist $w\in \R^q$ and $\alpha\in\R$ such that $\tilde{H}=\{(y^\T,r)\in\R^{q+1}\mid w^\T y - r = \alpha\}$. If $w=0$, then $\tilde{F}\subseteq \P\times\{0\}$ implies that $\alpha=0$ and we have $\tilde{H}=\R^q\times\{0\}$ and $P[\tilde{F}]=\P$. Suppose that $w\neq 0$. Then, $P[\tilde{H}]\neq \R^q$ is a supporting hyperplane of $P[\tilde{F}]$ since $\tilde{H}$ is a supporting hyperplane of $\tilde{F}$. It follows that $P[\tilde{F}]=\P\cap P[\tilde{H}]\in\mathcal{F}_\P$.}
		
		\Fir{Conversely, we define a function $G$ on $2^{\R^q}$ which maps $\mathcal{F}_\P\cup\{\P\}$ into the set of all $K$-minimal exposed faces of $\epi g$. %For each $(w^\T,\alpha)\in \D$, let us define $\tilde{H}(w,\alpha)\coloneqq \{(y^\T,r)^\T\in\R^{q+1} \mid w^\T y - r =\alpha\}$. 
			Let $G(F)\coloneqq F \times \{0\}$ for every $F\subseteq\R^q$. By the definitions of $P$ and $G$, we obtain $P(G(F)) = F$ for every $F\subseteq \R^q$ and $G(P(\tilde{F})) = \tilde{F}$ for every $\tilde{F} \subseteq \R^q \times \{0\}$. In particular, $P(G(F))=F$ for every $F\in \mathcal{F}_\P\cup\{\P\}$ and $G(P(\tilde{F}))=\tilde{F}$ for every $K$-minimal exposed face $\tilde{F}$ of $\epi g$. It remains to show that $G(F)$ is a $K$-minimal exposed face of $\epi g$ for all $F \in \mathcal{F}_{\P}\cup \{\P\}$. First, note that $G(\P) = \P \times\{0\} = \epi g \cap (\R^q \times \{0\})$ is a $K$-minimal exposed face of $\epi g$. Next, suppose that $F\in\mathcal{F}_\P$. Then, by \Cref{lem:duality_2}(b), there exists at least one $(w^\T,\alpha)^\T\in\D$ such that $F=H(w,\alpha)\cap \P$. Moreover, since $H(w,\alpha)$ is a supporting hyperplane of $\P$, it follows that $\tilde{H}(w,\alpha)\coloneqq \{(y^\T,r)^\T\in\R^{q+1} \mid w^\T y - r =\alpha\}$ is a supporting hyperplane of $\epi g$ and $\tilde{F}\coloneqq \tilde{H}(w,\alpha)\cap \epi g$ is a $K$-minimal exposed face of $\epi g$ by \cite[Lemma 3.1(ii)]{heyde_2013}.	
			Note that $P(\tilde{F}) = P(\tilde{H}(w,\alpha))\cap P(\epi g) = H(w,\alpha) \cap \P = F$. Hence, $G(F) = G(P(\tilde{F})) = \tilde{F}$ is a $K$-minimal exposed face of $\epi g$.}

		Let $F^\ast\in \mathcal{F}^\ast_\D$. Then, $-F^\ast$ is a $K$-minimal exposed face of $\epi g^\ast$, $\tilde{\Psi}(-F^\ast)$ is a $K$-minimal exposed face of $\epi g$, and \Cag{$P(\tilde{\Psi}(-F^\ast))\in \mathcal{F}_\P\cup\{\P\}$. Moreover, we have
			\begin{align*}		
			P(\tilde{\Psi}(-F^\ast))&=\bigcap_{(w^\T,g^\ast(w))^\T \in -F^\ast} P(\{(y^\T,g(y))^\T\in\R^{q+1}\mid w\in \partial g(y)\})\\
			&=\bigcap_{(-w^\T,-\tilde{p}(w))^\T \in F^\ast} P(\{(y^\T,0)^\T\in\R^{q+1}\mid y\in\P, w\in \mathcal{N}_\P(y)\})\\
			&=\bigcap_{(w^\T,p^w)^\T \in F^\ast}\{y\in\P\mid w\in -\mathcal{N}_\P(y)\}\\
			%	&\Fir{?=\bigcap_{(w^\T,p^w)^\T \in F^\ast}\{y\in\P\mid y\in H(w,p^w)\}}\\
			&=\bigcap_{(w^\T,\alpha)^\T \in F^\ast} H(w,\alpha)\cap \P
			= \Psi(F^\ast).
			\end{align*}
		}
		%	\out{Let $F^\ast\in \mathcal{F}^\ast_\D$. Then, $-F^\ast$ is a $K$-minimal exposed face of $\epi g^\ast$, $\tilde{\Psi}(-F^\ast)$ is a $K$-minimal exposed face of $\epi g$, and $P[\tilde{\Psi}(-F^\ast)]\in \mathcal{F}_\P\cup\{\P\}$. Moreover, we have
		%\begin{align*}		
		%		P[\tilde{\Psi}(-F^\ast)]&=\bigcap_{(w^\T,g^\ast(w))^\T \in -F^\ast} P[\{(y^\T,g(y))^\T\in\R^{q+1}\mid w\in \partial g(y)\}]\\
		%		&=\bigcap_{(-w^\T,-\tilde{p}(w))^\T \in F^\ast} P[\{(y^\T,0)^\T\in\R^{q+1}\mid y\in\P, w\in \mathcal{N}_\P(y)\}]\\
		%		&=\bigcap_{(w^\T,p^w)^\T \in F^\ast}\{y\in\P\mid w\in -\mathcal{N}_\P(y)\}\\
		%	&\Fir{?=\bigcap_{(w^\T,p^w)^\T \in F^\ast}\{y\in\P\mid y\in H(w,p^w)\}}\\
		%		&=\bigcap_{(w^\T,\alpha)^\T \in F^\ast} H(w,\alpha)\cap \P
		%		= \Psi(F^\ast).
		%		\end{align*}
		%	}
		
		By \cite[Theorem 3.3]{heyde_2013}, the inverse of $\tilde{\Psi}$ is given by
		\[
		\tilde{\Psi}^{-1}(\tilde{F})=\bigcap_{(y^\T,g(y))^\T\in \tilde{F}}\{(w^\T,g^\ast(w))^\T\in\R^{q}\times\R\mid w\in \partial g(y)\}.
		\]
		Let $F\in \mathcal{F}_\P\cup\{\P\}$. Then, $-\tilde{\Psi}^{-1}(G(F))\in \mathcal{F}^\ast_\D$ by the above constructions. Moreover, we have 
		\begin{align}
		-\tilde{\Psi}^{-1}(G(F))&=\bigcap_{(y^\T,g(y))^\T\in G(F)}\{(-w^\T,-g^\ast(w))^\T\in\R^{q}\times\R\mid w\in \partial g(y)\}\notag \\
		&=\Fir{\bigcap_{(y^\T,0)^\T\in F\times\{0\}}\{(-w^\T,-\tilde{p}(w))^\T\in\R^{q}\times\R\mid w\in\mathcal{N}_\P(y)\}}\notag \\
		&=\bigcap_{y\in F} \{(w^\T,p^w)^\T\in\R^{q}\times\R\mid w\in - \mathcal{N}_\P(y)\}\label{psiFast3}\\
		&=\bigcap_{y\in F} \{(w^\T,p^w)^\T\in\R^{q}\times\R\mid y\in H(w,p^w)\}\notag\\
		&=\bigcap_{y\in F} \{(w^\T,p^w)^\T\in\R^{q}\times\R\mid (w^\T,p^w)^\T\in H^\ast(y)\}\notag\\
		&=\bigcap_{y\in F} H^\ast(y)\cap \D = \Psi^\ast(F).\notag
		\end{align}
		\rev{Finally, note that $\Psi (\{(0^\T,0)^\T\})=\P$. Then, by combining the three one-to-one correspondences established above and excluding the pair formed by $\{(0^\T,0)^\T\}$ and $\P$, we conclude that $\Psi$ is an inclusion-reversing one-to-one correspondence between $\mathcal{F}_{\D}^\ast\setminus \{\{(0^\T,0)^\T\}\}$ and $\mathcal{F}_\P$, and $\Psi^\ast$ is its inverse mapping.}
	\end{proof}

\begin{remark}\label{rem:indicatrix}
	\rev{Using the notions of second-order subdifferential and indicatrix for convex functions, one can obtain the following polarity relationship between the curvatures of $\P$ and $\D$, which is in the spirit of \cite[Theorem 5.7]{heyde_2013}. Suppose that the function $g=I_{\P}$ (see the proof of \Cref{thm:duality}) is twice epi-differentiable (\cite[Definition 2.2]{rockafellar-epi}). Then, for every $K$-maximal exposed face $F^\ast$ of $\D$, $(w^\T,p^w)^\T\in F^\ast$, $y\in \Psi(F^\ast)$, we have
		\[
		\ind_\D((w^\T,p^w)\mid (y^\T,1))\coloneqq \ind g^\ast(-w\mid y)=(\ind g(y\mid -w))^\circ \eqqcolon (\ind_\P(y\mid-w))^\circ.
		\]
		Here, $\ind g, \ind g^\ast$ are defined as the polars of the corresponding second-order subdifferentials of $g, g^\ast$, respectively; see \cite[Section 4.1]{heyde_2013}, \cite[Proposition 4.1]{nd}, \cite[Section 4]{seeger}. Moreover, we define $\ind_\P, \ind_\D$ as indicatrices of $\P,\D$ (with suitable dimensions) by using $\ind g, \ind g^\ast$, similar to the construction in \cite[Section 5]{heyde_2013}. The result is a direct consequence of \cite[Lemma 4.6(b)]{seeger}, where it is needed to work with $(w,y)$ such that $-w\in \partial g (y)$. In our case, this condition is verified thanks to the structure of $\Psi(F^\ast)$ in \eqref{psiFast} whenever $(w^\T,p^w)^\T\in F^\ast$, $y\in \Psi(F^\ast)$. We verify the twice epi-differentiability of $g$ for the problem that will be considered in \Cref{example:3}. To that end, let us take $f(x)=A^\T x$ and $\X=\{x\in \R^n\mid x^{\T}Px-1\leq0\}$, where $P \in \R^{n\times n}$ is a symmetric positive definite matrix, and $A\in\R_+^{n\times q}$. By \cite[Lemma 4.6(a)]{seeger}, $g$ is twice epi-differentiable at $y$ relative to $-w$ if and only if $g^\ast$ is so at $-w$ relative to $y$. To check the latter, note that, for $w\in C^+$, we have
		\begin{align}
		g^\ast(-w)&=\sup_{z\in\P}-w^\T z = \sup_{x\in\X}-w^\T f(x)=\sup\{-w^\T A^\T x\mid x^\T P x \leq 1\} \notag\\
		& = \sqrt{(Aw)^\T P^{-1} Aw} = \sqrt{w^\T(A^\T P^{-1}A)w}, \notag
		\end{align}
		which follows from standard calculations for finding the support function of an ellipsoid. It follows that $(g^\ast)^2$ on $\R^q$ is a piecewise linear-quadratic function in the sense of \cite[Definition 1.1]{rockafellar-epi}, which is twice epi-differentiable by \cite[Theorem 3.1]{rockafellar-epi}. Hence, $g^\ast$ is also twice epi-differentiable.
	}
\end{remark}

Now, we will see that the upper and dual images $\P$ and $\D$ can be recovered from each other using the function $\varphi$ introduced in \eqref{eq:varphi}. The following definition will be used to simplify the notation in later steps.
%%%%%%%%%%%%%%%%%%%%%%%%%%%%%%%%%%%
\begin{definition}\label{def:PD-DP}
	For closed and convex sets $\bar{\P}$ and $\bar{\D}$, we define
	\begin{align}\label{eqn:6}
	&\D_{\bar{\P}}\coloneqq\{(w^{\T},\alpha)^{\T} \in \R^{q+1} \mid \forall y \in \bar{\P} \colon \varphi(y,w,\alpha) \geq 0\}, \notag \\
	&\P_{\bar{\D}}\coloneqq\{y \in \R^q \mid \forall (w^{\T},\alpha)^{\T} \in \bar{\D} \colon \varphi(y,w,\alpha)\geq 0\}. \notag 
	\end{align}
\end{definition}
\begin{remark}\label{rem:PDdualcone}
	\rev{In view of \eqref{eq:varphi}, the sets in \Cref{def:PD-DP} can be rewritten as
		\begin{align}
		&\D_{\bar{\P}}=\{z\in \R^{q+1} \mid \forall y \in \bar{\P} \colon (y^{\T},-1)z\geq 0\}=(\bar{\P}\times\{-1\})^+, \notag \\
		&\P_{\bar{\D}}=\{y\in\R^q\mid \forall z\in \bar{\D}\colon (y^{\T},-1)z\geq 0\}=\{y\in\R^q\mid  (y^{\T},-1)^{\mathsf{T}}\in \bar{\D}^+\}. \notag 
		\end{align}}
\end{remark}
%We start by showing upper and lower images can be recovered from transformations introduced in \Cref{def:PD-DP}.
\begin{proposition}\label{prop:DPequalsD}
	%For the upper image $\P$ and the lower image $\D$, i
	It holds (a) $\D_{\P}=\D$, (b) $\P_{\D}=\P$.
\end{proposition}
%\Fir{(item halinde yazılmayabilir belki? Ya da yanyana yazabiliriz iki item'ı?)}
\begin{proof}
	%	\begin{enumerate}[(a)]
	%		\item 
	(a) follows directly from the definitions, \Cref{rem:w_inC+} and \Cref{prop:upperimage}.
	%\Cag{To prove (a),} let $(w^\T,\alpha)^\T \in \D_{\P}$ be arbitrary. Since $\varphi(y,w,\alpha)=w^\T y-\alpha \geq 0 $ holds for every $y \in \P$, we have $\alpha\leq\inf_{y\in\P} w^\T y$. By \Cref{rem:w_inC+}, $w\in C^+$. Therefore, $(w^\T,\alpha)^\T \in \D$, which implies $\D_\P \subseteq \D$. Now let $(w^\T,\alpha)^\T \in \D$. Since $\alpha\leq \inf_{x\in\X}w^\T f(x)$, $w\in C^+$ and $\P = f(\X)+C$ by \Cref{prop:upperimage}, we have $\alpha\leq w^\T y $ for every $y\in \P$. Therefore, $(w^\T,\alpha)^\T \in \D_{\P}$.
	%\item
	To prove (b), let $y\in\P$ be arbitrary. By \Cref{prop:upperimage}, $y= f(x)+c \in \P$ for some $x\in\X, c\in C$. For every $(w^\T,\alpha)^\T\in\D$, we have $\varphi(f(x)+c,w,\alpha)
	%=w^\T (f(x) + c) -\alpha \geq \inf_{x\in\X} w^\T f(x)- \alpha +w^\T c 
	\geq 0$, which follows from $w\in C^+$ and $(w^\T,\alpha)^\T\in\D$. This implies
	%$w^\T f(x)\geq \inf_{x\in\X} w^\T f(x)\geq\alpha $. Also, we have $w^\T c \geq 0$ for each $w\in\W$ since $w\in C^+ $. Therefore, $w^\T f(x)+w^\T c-\alpha\geq0 $ and we have $\varphi(f(x)+c,w,\alpha)\geq0 $ and 
	$f(x)+c \in \P_{\D}$, hence, $\P_{\D}\supseteq\P$. 
	\rev{For the reverse inclusion, let $y\in \P_{\D}$. By \Cref{rem:PDdualcone} and the definition of $\D$, we have $(y^{\T},-1)^{\T}\in \D^+=(\xi(\W)-K)^+$. Hence, for every $w\in \W$ and $\alpha\leq p^w$, we have $(y^{\T},-1)(w^{\T},\alpha)^{\T}=w^{\T}y-\alpha \geq 0$. In particular, by choosing $w\in C^+\setminus \{0\}$ and setting $\alpha=p^w$, we get $w^{\T}y\geq p^w$ for every $w\in C^+\setminus\{0\}$. Note that $p^w=\inf_{x\in \X}w^{\T}f(x)=\inf_{y\in\P}w^{\T}y$ for every $w\in C^+\setminus\{0\}$. Since $\P$ is a nonempty closed convex set such that $\P=\P+C$, having $w^{\T}y\geq p^w$ for every $w\in C^+\setminus\{0\}$ is equivalent to $y\in \P$. Hence, $\P_\D\subseteq\P$.}
	%	\out{For the reverse inclusion, let $y\in\P_\D$. Assume that $y\notin \P$.
	%		To show $\P_{ \D}\subseteq\P$, assume the inclusion does not hold. Then, there exists $\bar{y}\in\P_{ \D}\setminus\P$. 
	%		By separation theorem, there exists $\bar{w}\in \R^q\setminus\{0\} $ with $\bar{w}^\T\bar{y}<\inf_{\bar{y}\in\P}\bar{w}^\T \bar{y}\eqqcolon \bar{\alpha}$. %Let $\bar{\alpha}\coloneqq\inf_{y\in\P}\bar{w}^\T y $. 		By the definition of $\bar{\alpha}$, $\bar{w}^\T \bar{y}\geq \bar{\alpha}$ for every $\bar{y}\in \P$. 
	%For every $y\in\P$, $\bar{w}^\T y\geq\bar{\alpha} $ holds, which means $\varphi(y,\bar{w},\bar{\alpha})\geq 0 $. 		Therefore, $(\bar{w}^\T,\bar{\alpha})^\T\in\D_\P$ and using the result from part (a), $(\bar{w}^\T,\bar{\alpha})^\T\in{ \D}$ as well. Now, $\varphi(y,\bar{w},\bar{\alpha})%=\bar{w}^\T y-\bar{\alpha}		<0 $, which contradicts with $y\in \P_{ \D}$. Hence, $\P_\D\subseteq\P$.}
	%	\end{enumerate}	
\end{proof}

	\subsection{Geometric duality between the approximations of $\P$ and $\D$} \label{subsec:geomdual_2}
	
	In \Cref{thm:duality}, we have seen that the upper and lower images can be recovered from each other. In this section, we show that similar relations hold for polyhedral approximations of these sets. Throughout, we call a set $A\subseteq \R^q$ ($A\subseteq \R^{q+1}$) an \emph{upper set} (a \emph{lower set}) if $A=A+C$ ($A = A-K$). %\Fir{(Bu tanımları kullanmışız bu section'da. Kullanmamaya karar verirsek son cümleyi sileriz.)} 
	We start by showing that a closed convex upper set can be recovered using the transformations introduced in \Cref{def:PD-DP}.
	
	\begin{proposition}\label{prop:7}
		Let $\emptyset \neq \bar{\P} \subsetneq \R^q $ be a closed convex set. Then, $\bar{\P}=\P_{\mathcal{D}_{\bar{\P} }} $.%, that is,
		%\begin{equation*}
		%\begin{aligned}
		%\bar{\P}=\{y \in \R^{q} \mid \forall (w^{\T},\alpha)^{\T} \in \D_{\bar{\P} }\colon  \varphi(y,w,\alpha)\geq0\}.
		%\end{aligned}
		%\end{equation*}
	\end{proposition}
	\begin{proof}
		\rev{Let $y\in\R^q$. \Cref{rem:PDdualcone} implies that $y\in \P_{\D_{\bar{\P}}}$ is equivalent to $(y^{\T},-1)^{\T}\in (\D_{\bar{\P}})^+=(\bar{\P}\times\{-1\})^{++}$. By the convexity of $\bar{\P}$ and \cite[Theorem 8.2]{rockafellar_1970}, we have
			\[
			(\bar{\P}\times\{-1\})^{++}=\cl\cone (\bar{\P}\times\{-1\})=\cone (\bar{\P}\times\{-1\})\cup (\rec \bar{\P}\times\{0\}).
			\]
			Hence, $(y^{\T},-1)^{\T}\in (\bar{\P}\times\{-1\})^{++}$ is equivalent to $(y^{\T},-1)^{\T}\in \cone(\bar{\P}\times\{-1\})$, which is equivalent to $y\in \bar{\P}$. Therefore, $\P_{\D_{\bar{\P}}}=\bar{\P}$.
		}
		%\out{To show that $\bar{\P} \subseteq \P_{\D_{\bar{\P} }}$, let $y \in \bar{\P} $. For every $(w^{\T},\alpha)^{\T} \in \D_{\bar{\P} } $, we have $\varphi(y,w,\alpha) \geq 0$ as $y\in\bar{\P} $. Hence, $y \in \P_{\D_{\bar{\P} }} $. Next, we show that $\bar{\P} \supseteq \P_{\D_{\bar{\P} }}$. Assume that the inclusion does not hold. Then, there exists $\tilde{y} \in \P_{\D_{\bar{\P} }} \setminus \bar{\P} $. By the definition of $\P_{\D_{\bar{\P} }}$, $\varphi(\tilde{y},w,\alpha)%=w^{\T}\tilde{y}-\alpha
		%\geq0$ holds for every $(w^{\T},\alpha)^{\T} \in \D_{\bar{\P} }$. Using separation theorem, we may find $\bar{w} \in \R^q\setminus\{0\}$ such that $\bar{w}^{\T}\tilde{y}<\inf_{\bar{y}\in \bar{\P}}\bar{w}^{\T}\bar{y}=:\bar{\alpha} $. Clearly, for every $y \in \bar{\P} $, $\bar{w}^{\T}y \geq \bar{\alpha}$, that is, $\varphi(y,\bar{w},\bar{\alpha})\geq 0 $ holds. Hence, $(\bar{w}^{\T},\bar{\alpha}) \in \D_{\bar{\P} }$. But $\varphi(\tilde{y},\bar{w},\bar{\alpha})%=\bar{w}^{\T}\tilde{y}-\bar{\alpha}
		%<0$, contradicting the definition of $\P_{\D_{\bar{\P} }}$. Therefore, $\bar{\P} \supseteq \P_{\D_{\bar{\P} }}$.}
	\end{proof}

	Next, for a closed convex lower set $\bar{\D} $, we want to investigate the relationship between $\mathcal{\bar{D}} $ and $\mathcal{D}_{\mathcal{P}_{\mathcal{\bar{D}}}} $. While the equality of these sets may not hold in general, \rev{the next proposition shows that it holds if $\bar{\D}$ is a cone.} %\out{we will see by Propositions \ref{prop:8} and \ref{prop:25} that it holds for some special cases. Before these results, we provide two lemmas.} 
	\rev{\begin{proposition}\label{prop:DPD=D}
			Let $\emptyset\neq \bar{\D}\subseteq\R^{q+1}$ be a closed convex lower set. Suppose further that $\bar{\D}$ is a cone and $\P_{\bar{\D}}\neq\emptyset$. Then, $\D_{\P_{\bar{\D}}} = \D$.
		\end{proposition}
		\begin{proof}
			By \Cref{rem:PDdualcone}, we have
			\begin{align*}
			\D_{\P_{\bar{\D}}}=(\P_{\bar{\D}}\times\{-1\})^+&=(\{y\in\R^q\mid  (y^{\T},-1)^{\mathsf{T}}\in \bar{\D}^+\}\times\{-1\})^+\\
			&=\of{\bar{\D}^+\cap (\R^q\times\{-1\})}^+.
			\end{align*}
			From this observation, it follows that $\D_{\P_{\bar{\D}}}\supseteq \bar{\D}^{++}=\bar{\D}$ since $\bar{\D}$ is a nonempty closed convex cone. To prove the reverse inclusion, let $(w^{\T},\alpha)^{\T}\in \D_{\P_{\bar{\D}}}$. Then, by the above observation, we have
			\begin{equation}\label{eq:pos_prod}
			(y^{\T},-1)(w^{\T},\alpha)^{\T}=w^{\T}y-\alpha\geq 0
			\end{equation}
			for every $y\in\R^q$ such that $(y^{\T},-1)^{\T}\in \bar{\D}^+$. We show that $(w^{\T},\alpha)^{\T}\in \bar{\D}=\bar{\D}^{++}$. To that end, let $(y^{\T},\beta)^{\T}\in \bar{\D}^+$. We claim that $(y^{\T},\beta)(w^{\T},\alpha)^{\T}=w^{\T}y+\beta\alpha\geq 0$. First, note that $\beta\leq 0$ since $(y^{\T},\beta)^{\T}\in \bar{\D}^+$ and $\bar{\D}$ is a lower set. We consider the following cases.\\
			\emph{Case 1:} Suppose that $\beta<0$. Then, we have $(-\frac{1}{\beta}y^{\T},-1)^{\T}\in \bar{\D}^+$ since $\bar{\D}^+$ is a cone. In particular, applying \eqref{eq:pos_prod} for this point yields $-\frac{1}{\beta}w^{\T}y-\alpha \geq 0$, that is, $w^{\T}y+\beta\alpha\geq 0$. Hence, the claim follows.\\
			\emph{Case 2:} Suppose that $\beta = 0$. Since $\bar{\D}^+$ is a convex cone, $(y^{\T},0)^{\T}$ is a recession direction of $\bar{\D}^+$. In particular, for every $\bar{y}\in \P_{\bar{\D}}\neq \emptyset$, we have $(\bar{y}^{\T},-1)^{\T}\in\bar{\D}^+$ by \Cref{rem:PDdualcone}, which implies that $(\bar{y}^{\T},-1)^{\T}+t(y^{\T},0)^{\T}=((\bar{y}+ty)^{\T},-1)^{\T}\in\bar{\D}^+$ for every $t\in\R_+$. Hence, by \Cref{rem:PDdualcone} again, $\bar{y}+ty\in \P_{\bar{\D}}$ for every $\bar{y}\in\P_{\bar{\D}}$ and $t\in\R_+$, that is, $y\in \rec \P_{\bar{\D}}$. Furthermore, the proof of \cite[Theorem 8.2]{rockafellar_1970} states that $\rec \P_{\bar{\D}}\times\{0\}$ consists of the limits of all sequences of the form $(\lambda_n((y^n)^{\T},-1)^{\T})_{n\in\N}$, where $(\lambda_n)_{n\in\N}$ is a decreasing sequence in $\R_+$ whose limit is $0$, and $((y^n)^{\T})_{n\in\N}$ is a sequence in $\P_{\bar{\D}}$. Since $(y^{\T},0)^{\T}\in\rec \P_{\bar{\D}}\times \{0\}$, such $(\lambda_{n})_{n\in\N}$ and $((y^n)^{\T})_{n\in\N}$ exist with the additional property that $\lim_{n\rightarrow\infty}\lambda_n y^n=y$. For each $n\in\N$, having $y^n\in \P_{\bar{\D}}$ implies that $((y^n)^{\T},-1)^{\T}\in \bar{\D}^+$ by \Cref{rem:PDdualcone}, and hence $w^{\T}y^n-\alpha\geq 0$ by \eqref{eq:pos_prod}. Then, $w^{\T}\lambda_n y^n-\lambda_n \alpha \geq 0$ for each $n\in\N$. Letting $n\rightarrow\infty$ yields that $w^{\T}y=w^{\T}y+\beta\alpha\geq 0$. Hence, the claim follows.\\
			Therefore, $(w^{\T},\alpha)^{\T}\in \bar{\D}^{++}=\bar{\D}$, which completes the proof of $\D_{\P_{\bar{\D}}}\subseteq\bar{\D}$.
			%	
			%	is closed and consists of the limits of sequences of the form $\lambda_i(y_i,-1),$ where $y_i\in \P_{\bar{\D}}, \lambda_i\geq0 with \lim_i \lambda_i = 0$ % If $K = \cone(\P_{\bar{\D}}\times \{-1\}  )$, then $\cl K = K \cup (\rec \P_{\bar{\D}}\times \{0\})$.\\Hence, $(y^\T, 0)^\T \in \cl \cone (\P_{\bar{\D}}\times \{-1\}) = \cl\left\{\sum_k \lambda_k((y_k)^\T, -1)^\T \mid y_k \in \P_{\bar{\D}}, \lambda_k \geq 0\right\}$. Then, there exists $\lambda_{k^i} \geq 0, y_{k^i} \in \P_{\bar{D}}$ such that $\lim_{i \to \infty}\sum_{k^i}\lambda_{k^i}(p_{k^i}^\T,-1) = (y^\T,0)^\T$. \\	$w^\T (\sum_{k^i}\lambda_{k^i}) = \sum_{k^i}\lambda_{k^i}w^\T y_{k^i} \geq \alpha \sum_{k^i}\lambda_{k^i}$ since $y_{k^i}\in \P_{\bar{\D}}$ 	
	\end{proof}}
	
	The \rev{next lemma} shows that, when computing $\P_{\bar{\D}}$, considering the extreme directions of $\bar{\D}$ is sufficient under a special structure for $\bar{\D}$. Its proof is quite routine, we omit it.
	\begin{lemma}\label{lem:6}
		Let $\bar{\D}=\cone \conv \xi(\bar{\W})-K$ %\Fir{(\sout{$\cone (\conv (\xi(\bar{\W})))-K$} daha önceki kullanımlarda, örn. Defn 3.7,  parantez kullanmamışız diye parantezleri kaldırdım.)}
		for some $\bar{\W}\subseteq\W$. Then,
		$%\begin{equation*}
		\P_{\bar{\D}}=\{y \in \R^q\mid 
		\forall(w^{\T},\alpha)^{\T}\in\xi(\bar{\W})\colon \varphi(y,w,\alpha)\geq0\}$.
		%\end{equation*}
	\end{lemma}
	
	\rev{The next corollary provides two special cases of \Cref{prop:DPD=D} which will be used later.
		\begin{corollary}\label{cor:special}
			Let $\bar{\mathcal{W}}\subseteq\mathcal{W}$ be a nonempty finite set. Then, $\bar{\D}=\D_{\P_{\bar{\D}}}$ in each of the following cases.\\
			(a) $\bar{\D}=\cone\conv\xi(\bar{\W})-K$.\quad 
			(b) $0\notin \bar{\W}$ and $\bar{\D}=\cone(\conv\xi(\bar{\W}) +\epsilon\{e^{q+1}\})-K$.
			%$\D_\epsilon=\cone(\conv\xi(\bar{\W} +\epsilon\{e^{q+1}\})-K$ for some nonempty $\bar{\mathcal{W}}\subseteq\mathcal{W}$ and let $\P_{\D_\epsilon}$ be defined by \Cref{def:PD-DP}. Then, $\D_\epsilon=\D_{\P_{\D_\epsilon}} $
		\end{corollary}
		\begin{proof}
			Since $\bar{\mathcal{W}}$ is finite, $\bar{\D}$ is a closed convex lower set that is also a cone in each case. Next, we show that $\P_{\bar{\D}}\neq\emptyset$.\\
			(a) Let $\bar{x}\in\X$. We have $w^{\T}f(\bar{x})\geq \inf_{x\in\X}w^{\T}f(x)=p^w$ for every $w\in \W$, in particular, for every $w\in\bar{\W}$. Hence, $f(\bar{x})\in\P_{\bar{\D}}$ by \Cref{lem:6}.\\
			(b) By the definition of $\P_{\bar{\D}}$ and simple algebraic manipulations, we have
			\begin{align*}
			\P_{\bar{\D}}&=\{y\in\R^q\mid \forall (w^{\T},\alpha)^{\T}\in \bar{\D}\colon w^{\T}y\geq \alpha\}\\
			&=\{y\in\R^q\mid \forall (w^{\T},\alpha)^{\T}\in \conv \xi(\bar{\W})+\epsilon \{e^{q+1}\} \colon w^{\T}y\geq \alpha\}\\
			&=\{y\in\R^q\mid \forall (w^{\T},\alpha)^{\T}\in \conv\xi(\bar{\W}) \colon w^{\T}y\geq \alpha+\epsilon\}\\
			&=\{y\in\R^q\mid \forall (w^{\T},\alpha)^{\T}\in \xi(\bar{\W}) \colon w^{\T}y\geq \alpha+\epsilon\}.
			\end{align*}
			Since $0\notin \bar{\W}$ and $\bar{\W}$ is finite, there exists $\bar{c}\in\Int C$ such that $w^{\T}\bar{c}\geq \epsilon$ for every $w\in\bar{\W}$. Let $\bar{x}\in \X$. Then, $w^{\T}(f(\bar{x})+\bar{c})\geq p^w+\epsilon$ for every $w\in\bar{\W}$. Hence, $f(\bar{x})+\bar{c}\in \P_{\bar{\D}}$.
	\end{proof}}

	\section{Algorithms}\label{sec:algorithms}
	
	In this section, we will present two approximation algorithms, namely the primal and dual algorithms, for solving the primal and dual problems, \eqref{P} and \eqref{D}, simultaneously. First, we will explain the primal algorithm, which is proposed in \cite{umer_2022} for solving \eqref{P} only, and show that by simple modifications, this algorithm also yields a solution to the dual problem \eqref{D}. 
	Next, we will describe the dual algorithm which uses the geometric duality results from \Cref{sec:geomdual}. %Both algorithms are outer approximation algorithms, the first one approximating the upper image of the primal problem and the latter approximating the lower image of the dual problem.
	
	Recall that $C$ is assumed to be a nontrivial, pointed, solid, closed and convex cone; the vector-valued objective function $f\colon \X\to\R^q $ is $C$-convex and continuous; and the feasible set $\X\subseteq \R^{m} $ is compact and convex. We further assume the following from now on.
	
	\begin{assumption}\label{ass:1}
		(a) The feasible region of \eqref{P} has nonempty interior, that is, $\Int \X \neq \emptyset$; and (b) the ordering cone $C$ is polyhedral. 
		%Let the following hold true:
		%	\begin{itemize}
		%		\item[(a)] $\mathcal{X}$ has non-empty interior.
		%		\item[(b)] $C$ is polyhedral.
		%	\end{itemize}
	\end{assumption}
	%In the following two sections, we describe the primal and dual algorithms, respectively. 
	
	Under \Cref{ass:1} (b), it is known that $C^+$ is polyhedral. We denote the generating vectors of $C^+$ by $w^1,\ldots, w^J$ and assume without loss of generality that $\norm{w^j}_\ast = 1$ for each $j\in\{1,\ldots,J\}$.
	
	\subsection{Primal algorithm}
	
	The primal algorithm in \cite{umer_2022} is an outer approximation algorithm, that is, it works with polyhedral outer approximations of the upper image $\P$. In particular, it starts by finding a polyhedral outer approximation $\P_0$ of the upper image and iterates by updating the outer approximation with the help of supporting halfspaces of $\P$ until the approximation is sufficiently fine.
	
	In each iteration of the primal algorithm (\Cref{alg1}), the following norm-minimizing scalarization problem is solved:
	\begin{equation*}\tag{$\textnormal{P}(v)$}\label{PS(v)}
	\text{minimize } \norm{z} \text{ subject to } f(x)-z-v \in -C,\ z \in \R^q, \ x\in\X,
	\end{equation*}
	where $v\in\R^q$ is a parameter to be set by the algorithm. The Lagrangian dual of \eqref{PS(v)} is given by 
	\begin{equation*}\tag{$\textnormal{D}(v)$}\label{D-PS(v)}
	\text{maximize } \inf_{x\in\X} w^\T f(x)-w^\T v \text{ subject to }  \norm{w}_\ast \leq 1, \ w \in C^+.
	\end{equation*}
	%\Fir{(Modele (PS($v$)) ve dual'ine (D-PS($v$)) demişiz önceden. PS ifadesini daha çok Pascoletti-Serafini scalarization için kullanıyorduk, burada nedensiz olmuş sanki :) Değiştirip P ve D yaptım. Label'ı değiştirmedim. Yine de eski halini manuel olarak kullanmış olabiliriz. Ben aradığımda bulamadım ama gözden kaçırmamışımdır umarım.)}
	
	Before explaining the details of the algorithm, we present some results regarding \eqref{PS(v)} and \eqref{D-PS(v)}; see \cite{umer_2022} for details and further results.
	
	%To find the distance between a given point $v$ and the upper image, a convex problem is solved. The convex problem and its dual in \cite{umer_2022} can be seen below:
	
	%\Fir{(Buradaki 3 Lemma yerine tek bir proposition yazdım ve ispatlarını yazmak yerine Muhammad'in makalesindeki sonuçlara referans verdim. Eski hali olduğu gibi commented out şekilde aşağıda duruyor. Simay tezinde ispatları da ekleyebilir gerekirse; ya da bu haliyle de yazılabilir diye düşünüyorum.)}
	
	%\Fir{(Bu proposition'ın ilk maddesi, önceki ilk iki Lemma'ya denk geliyor. Bunlara hiç referans verilmemiş. İkinci maddesine ise sadece Proposition 5.3'ün ispatında referans verilmiş, şmdilik gördüğüm kadarıyla.)}
	
	\begin{proposition} \label{prop:PS_oldresults}
		Let $v \in \R^q$. The following statements hold under \Cref{ass:1}:
		%	\begin{enumerate}[(a)]
		%\item 
		(a) \cite[Proposition 4.2]{umer_2022} There exist optimal solutions $(x^v,z^v)$ and $w^v$ to \eqref{PS(v)} and \eqref{D-PS(v)}, respectively, and the optimal values coincide.
		%\item
		(b) \cite[Proposition 4.6]{umer_2022} If $v\notin \Int\P$, then $x^v$ is a weak minimizer for \eqref{P}.
		%\item \cite[Proposition 4.7]{umer_2022} If $w^v\neq 0$, then $\{y\in \R^q \mid (w^v)^\T y \geq (w^v)^\T f(x^v)\}$ is a supporting halfspace of $\P$ that contains it.
		(c) \cite[Remark 4.4]{umer_2022} $x^v $ is an optimal solution to $\textnormal{(WS}(w^v))$, i.e., $\inf_{x\in\X}(w^v)^\T f(x) = (w^v)^\T f(x^v)$.
		%\end{enumerate}
	\end{proposition}

	The primal algorithm is initialized by solving the weighted sum scalarization problem \eqref{WS(w)} for each generating vector of the dual ordering cone $C^+$. Let $x^j \in \X$ denote the optimal solution of $(\textnormal{WS}(w^j))$ for each $j\in\{1,\ldots, J\}$. By \Cref{prop:WS}, each $x^j$ is a weak minimizer for \eqref{P}. Moreover, from \Cref{prop:WS_Kmax} and \Cref{def:solnDual}, it is known that each $w^j$ is a maximizer for \eqref{D}. Hence, we initialize the set to be returned as a weak $\epsilon$-solution for \eqref{P} as $\bar{\X} = \{x^1,\ldots,x^J\}$; and the set to be returned as an $\epsilon$-solution to \eqref{D} as $\bar{\W} = \{w^1,\ldots,w^J\}$. Note that by \Cref{prop:supp_halfsp}, for each $j\in\{1,\ldots, J\}$, the set $\H(\xi(w^j))$ is a supporting halfspace of $\P$ such that $\H(\xi(w^j))\supseteq \P$. Then, the initial outer approximation of $\P$ is defined as
	%\begin{equation*}
	$\P_0\coloneqq \bigcap_{j=1}^{J}\H(\xi(w^j))$. 
	%\end{equation*}
	As part of the initialization, we introduce a set $\Vkn$, which stores the set of vertices that have already been considered by the algorithm and initialize it as $\Vkn=\emptyset$, see lines 1-3 of \Cref{alg1}. We later introduce a set $\Vunkn$, which stores the set of vertices of the current approximation that are not yet considered, see line 7.
	
	In each iteration $k$, the first step is to compute the vertices $\V_k$ of the current outer approximation $\P_k$ by solving a vertex enumeration problem (line 6). Then, for each vertex in $\V_k$ which have not been considered before, optimal solutions for \eqref{PS(v)} and \eqref{D-PS(v)} are found; the respective solutions are added to sets $\bar{\X}$ and $\bar{\W}$ (see \Cref{prop:PS_oldresults}(b) and \Cref{prop:WS_Kmax}); and $\Vkn$ is updated (lines 7-10). Note that by \Cref{prop:PS_oldresults}(c) and \Cref{prop:supp_halfsp}, $\H(\xi(w^v))$ is a supporting halfspace of $\P$. If the distance of a vertex $v$ to the upper image, namely $\norm{z^v}$, is not sufficiently small, then $\H(\xi(w^v))=\{y\in \R^q\mid \varphi(y,w^v,(w^v)^\T f(x^v))\geq 0 \}$ is stored in order to be used in updating the current outer approximation. After each vertex in $\V_k$ is considered, then the current approximation is updated by intersecting it with those halfspaces (lines 11-16). The algorithm terminates when all the vertices of $\P_k$ are in $\epsilon$ distance to the upper image (lines 5 and 18).

	\begin{algorithm}[!]  
		\caption{Primal algorithm}
		\begin{algorithmic} [1]
			\STATE Compute an optimal solution $x^j$ of $(\textnormal{WS}(w^j))$ for each $j\in\{1,\ldots,J\}$;
			\STATE Let $\P_0=\bigcap_{j=1}^{J}\H(\xi(w^j))$;
			%		Store an H-representation $\mathcal{P}^H$ of $\mathcal{P}_0$, which is computed as in \eqref{eqn:8}.
			\STATE $k \gets 0, \bar{\X}\gets\{x^1, \dots, x^J\}, \bar{\W}\gets\{w^1, \dots, w^J\}, \Vkn=\emptyset$; %\Fir{\sout{$\Vunkn=\emptyset$; (Bunu aşağıda tanımlıyoruz zaten?)}} \\%$D\leftarrow \xi(\bar{\W})$ where $\xi_{q+1}(\bar{w})=\bar{w}^Tf(x^{\bar{w}})$ for $\bar{w} \in \bar{\mathcal{W}}$ ; 
			%\Fir{(D kümesi vardı kullanılmayan, comment out yaptım.)}
			\REPEAT
			\STATE $M\gets\R^q$;
			\STATE Compute the set $\V_k$ of vertices of $\P_k $;% from its H-representation $\mathcal{P}^H$;
			\STATE $\Vunkn\gets \V_k \setminus \Vkn$;
			\FOR {$v\in \Vunkn$}
			\STATE %\sout{Let $v$ be the $i^{th}$ element of $\mathcal{V}_{unknown}$ (i.e. the $i^{th}$ unknown vertex of $\mathcal{P}_k$); } \\ 
			Compute optimal solutions $(x^v,z^v)$ and $w^v$ to \eqref{PS(v)} and \eqref{D-PS(v)};
			\STATE $\bar{\X} \gets \bar{\X} \cup \{x^v\}, \bar{\W} \gets \bar{\W} \cup \{\frac{w^v}{\norm{w^v}_\ast}\}, \Vkn \gets \Vkn \cup \{v\}$;
			\IF {$\norm{z^v} >\epsilon$}
			
			\STATE $M \leftarrow M \cap \H(\xi(w^v))$;
			\ENDIF		
			\ENDFOR
			\IF{$M\neq\R^q$}
			\STATE %Store in $\mathcal{P}^H$ an H-representation of
			$\P_{k+1}=\P_k\cap M, \ k \gets k+1$;
			\ENDIF
			\UNTIL $M=\R^q$
			%\STATE Compute the set $\V$ of vertices of $\{y \in \R^q\mid \forall (w^\T,\alpha)^\T \in \xi(\bar{\W})\colon \varphi(y,w,\alpha)\geq 0\}$; \\
			%\Fir{(Bunu yazacak mıyız? Gerek yok diye düşünüyorum, eger kullanmıyorsak.)}
			\RETURN %\Fir{(Algorithmanın return kısmından bazı ifadeleri çıkardım.)}\\
			$\left\{\begin{array}{ll}\bar{\mathcal{X}}& 
			\colon\text{A finite weak }\epsilon \text{-solution to \eqref{P}}; \\
			\bar{\mathcal{W}}&\colon\text{A finite }\epsilon \text{-solution to \eqref{D}};
			%		\V & \text{: Vertices of an outer set of } \mathcal{P} \\
			%		f(\bar{\mathcal{X}}) & \text{: Vertices of an inner set of } \mathcal{P}
			\end{array}\right.	$
		\end{algorithmic}
		\label{alg1}
	\end{algorithm}
	
	\begin{remark}\label{rem:3}
		A `break' command can be placed between lines 12 and 13 in the algorithm. In the current version, the algorithm goes over all the vertices of the current outer approximation without updating it. With the `break' command, the algorithm updates the outer approximation as soon as it detects a vertex $v$ with $\norm{z^v} >\epsilon $.
	\end{remark}
	
	The next proposition states that \Cref{alg1} gives a finite weak $\epsilon$-solution %$\bar{\mathcal{X}}$ %and $\bar{\mathcal{W}}$ 
	to \eqref{P}. %and \eqref{D} respectively.
	
	\begin{proposition}\cite[Theorem 5.4]{umer_2022} \label{prop:11}
		If the primal algorithm terminates, then it returns a finite weak $\epsilon$-solution $\bar{\mathcal{X}} $ to \eqref{P}.
	\end{proposition}
	
	%\begin{proof}
	%	To prove the statement, one needs to show $\conv f(\bar{\mathcal{X}})+C+B(0,\epsilon) \supseteq \mathcal{P}$. Let the algorithm terminate at $J^{th}$ iteration and $\mathcal{P}_J$ denotes the resulting outer approximation of the upper image. By \Cref{prop:supp_halfsp} and \Cref{rem:2}, $\mathcal{P}_J \supseteq \mathcal{P} $. Note that $\mathcal{P}_J= \conv(v_J^1,\dots,v_J^m) + C$. Define $\tilde{\mathcal{X}} $ as the set of optimal solutions $x^{v_J^1},\dots,x^{v_J^m} $ to $\textnormal{PS}(v_J^1),\dots,\textnormal{PS}(v_J^m)$ respectively. Therefore, one can write the following inclusions. The first inclusion follows from the fact that the algorithm terminates only if all vertices of $\mathcal{P}_J $ satisfies $\|z^v\| \leq \epsilon $ and the second inclusion follows from $\tilde{\mathcal{X}} \subseteq \bar{\mathcal{X}} $.
	%	\begin{equation*}
	%	\begin{aligned}
	%	\conv(v_J^1,\dots,v_J^m) + C &\subseteq \conv f(\tilde{\mathcal{X}})+C+B(0,\epsilon) \\
	%	&\subseteq \conv f(\bar{\mathcal{X}})+C+B(0,\epsilon).
	%	\end{aligned}
	%	\end{equation*}
	%	Hence, $\conv f(\bar{\mathcal{X}})+C+B(0,\epsilon) \supseteq \mathcal{P}_J $ and $\conv f(\bar{\mathcal{X}})+C+B(0,\epsilon) \supseteq \mathcal{P}$. Since $\bar{\mathcal{X}}$ is nonempty and only consists of minimizers by \Cref{prop:WS}, one can conclude that $\bar{\mathcal{X}}$ is a finite $\epsilon$-solution to \eqref{P}.
	%\end{proof}
	
	Next, we show that the primal algorithm yields also a finite $\epsilon$-solution, $\bar{\W}$, for the dual problem \eqref{D}. To that end, we provide the following lemma which shows that inner and outer approximations of the lower image $\D$ can be obtained by using a finite $\epsilon $-solution $\bar{\X}$ of \eqref{P}. Then, this lemma will be used in order to prove the main result of this section.
	
	\begin{lemma}\label{prop:9}
		For $\epsilon>0$, let $\bar{\X}$ be a finite weak $\epsilon$-solution of \eqref{P}, and $\P_\epsilon\coloneqq\conv f(\bar{\X})+C+B(0,\epsilon)$. Then, $\D_\epsilon\coloneqq\D_{\P_\epsilon}
		%=\{(w^\T,\alpha)^\T \in \R^{q+1} \mid \forall y \in \P_\epsilon\colon \varphi(y,w,\alpha)\geq0\}
		$ is an inner approximation of $\D$ and %satisfies
		\begin{equation}\label{eq:prop9}
		\cone\left(\left(\D_\epsilon \cap (\S^{q-1}\times\R)\right)+ \epsilon \{e^{q+1}\} \right)- K \supseteq\D \supseteq \D_\epsilon.
		\end{equation}
		%	where $\S^{q-1}=\{r\in\R^q:\|r\|_*=1\} $.
	\end{lemma}
	
	\begin{proof}
		Since $\bar{\X}$ is a finite weak $\epsilon$-solution of \eqref{P}, $\P_\epsilon$ is an outer approximation of the upper image $\P$ by \Cref{def:solnPrimal}, that is, $\P_\epsilon \supseteq\P$. By \Cref{prop:DPequalsD}(a), $\D\supseteq \D_\epsilon$. 
		
		%Next, we show $\cone[(\D_\epsilon\cap (\S^{q-1}\times\R))+ \epsilon \{e^{q+1}\} ]- K \supseteq\D$ holds. Define
		In order to show the first inclusion in \eqref{eq:prop9}, we first note that $\P^\prime\coloneqq\conv f(\bar{\X})+C \subseteq \P$ and, %Since $\X \supseteq \bar{\X}$, $\P \supseteq \P^\prime$ holds. 
		by \Cref{prop:DPequalsD}(a), we have
		\begin{equation}\label{eq:inclzero}
		\D^\prime\coloneqq \D_{\P^\prime} =\{(w^\T,\alpha)^\T \in \R^{q+1} \mid \forall y \in \P^\prime \colon \varphi(y,w,\alpha)\geq0\} 
		\supseteq \D.
		\end{equation}
		We claim that
		\begin{equation} \label{eq:incl}
		(\D_\epsilon\cap(\S^{q-1}\times\R))+ \epsilon \{e^{q+1}\} \supseteq\D^\prime\cap(\S^{q-1}\times\R)
		\end{equation}
		holds. Indeed, observe that
		\begin{align}
		&(\D_\epsilon \cap (\S^{q-1}\times\R)) + \epsilon \{e^{q+1}\}\notag \\
		%&=\{(w^\T,\alpha+\epsilon)^\T \in \R^{q+1}  \mid \norm{w}_\ast=1,\ \forall y \in \P_\epsilon\colon \varphi(y,w,\alpha)\geq 0\} \notag \\
		&=\{(w^\T,\alpha)^\T \in \R^{q+1} \mid \norm{w}_\ast=1,\ \forall y \in \P_\epsilon\colon \varphi(y,w,\alpha-\epsilon)\geq 0\} \notag \\
		&=\{(w^\T,\alpha)^\T \in \R^{q+1} \mid \norm{w}_\ast=1, \ \forall y \in \P^\prime,\ \forall \gamma \in B(0,1)\colon \varphi(y+\gamma\epsilon,w,\alpha-\epsilon)\geq 0\}.\notag 
		\end{align}
		On the other hand, we have
		\begin{equation*}
		\D^\prime\cap(\S^{q-1}\times\R)=\{(w^\T,\alpha)^\T \in \R^{q+1} \mid \norm{w}_\ast=1,\ \forall y \in \P^\prime\colon \varphi(y,w,\alpha)\geq 0\}.
		\end{equation*}
		Let $(w^\T,\alpha)^\T\in\D^\prime\cap(\S^{q-1}\times\R)$ be arbitrary. Note that $\varphi(y,w,\alpha)=w^\T y-\alpha\geq0 $ for every $y\in \P^\prime$ and $\norm{w}_\ast=1$. Moreover, for every $\gamma \in B(0,1)$ we have $\abs{w^\T\gamma} \leq \norm{\gamma}\norm{w}_\ast \leq 1$, hence $w^\T\gamma \geq -1$.
		%	since $\norm{w}_\ast=1 $, we have $w^\T\gamma=\norm{w}_\ast\|\gamma\|cos\beta \geq -1 $ for every $\gamma \in B(0,1)$, where $\beta$ is the angle between $w$ and $\gamma$. 
		Then, $\varphi(y+\gamma\epsilon,w,\alpha-\epsilon) = w^\T y+\epsilon w^\T\gamma-\alpha+\epsilon\geq0 $ holds, that is, $(w^\T,\alpha)^\T\in (\D_\epsilon\cap(\S^{q-1}\times\R))+ \epsilon \{e^{q+1}\}$, which implies \eqref{eq:incl}. %Therefore the inclusion $(\D_\epsilon\cap(\S^{q-1}\times\R))+ \epsilon \{e^{q+1}\} \supseteq\D^\prime\cap(\S^{q-1}\times\R) $ holds.
		
		The next step is to show that
		\begin{equation}\label{eq:Dprime}
		\cone(\D^\prime\cap(\S^{q-1}\times\R))-K=\D^\prime.
		\end{equation} 
		The inclusion $\subseteq$ is straightforward. %Note that an arbitrary element of $\cone(\D^\prime\cap(\S^{q-1}\times\R))-K $ can be written as $(\lambda w^\T,\lambda\alpha-\gamma)^\T $ for some
		%	Let $(\lambda w^T,\lambda\alpha-\gamma)^T \in \cone(\D^\prime\cap(\S^{q-1}\times\R))-K $ be arbitrary, where
		%$\lambda,\gamma\geq 0 $ and  $(w^\T,\alpha)^\T\in \R^{q+1}$ satisfying $\norm{w}_\ast=1$ and $\varphi(y,w,\alpha)=w^\T y-\alpha\geq 0 $ for all $y\in\P^\prime$. Note that we have $\varphi(y,\lambda w,\lambda\alpha-\gamma)=\lambda w^\T y-\lambda\alpha+\gamma\geq0 $. Hence, $(\lambda w^\T,\lambda\alpha-\gamma)^\T \in\D^\prime$, which shows that $\cone(\D^\prime\cap(\S^{q-1}\times\R))-K\subseteq\D^\prime$. 
		For the reverse inclusion, let $( w^\T,\alpha)^\T \in \D^\prime$. First, suppose that $w\neq 0$. Let $\lambda\coloneqq\frac{1}{\|w\|_\ast}>0 $. Clearly, $\norm{\lambda w}_\ast=1 $ and $\varphi(y,\lambda w,\lambda \alpha)=\lambda w^\T y-\lambda\alpha\geq0 $ holds for each $y\in\P^\prime $. Hence, $(\lambda w^\T,\lambda\alpha)^\T\in \D^\prime\cap(\S^{q-1}\times\R) $ and $(w^\T,\alpha)^\T \in \cone(\D^\prime\cap(\S^{q-1}\times\R))-K $. Now, suppose that $w = 0$. By the definition of $\D^\prime$, we have $\alpha\leq 0$.
		%Now take $ (0,\dots,0,-\gamma)\in\D^\prime $ where $\gamma \in K $. 
		%Since 
		%$\cone(\D^\prime\cap(\S^{q-1}\times\R))$ is a \rev{closed cone} \Fir{not necessarily closed}, 
		\rev{Note that $0\in\cone(\D^\prime\cap(\S^{q-1}\times\R)) $ and $(w,\alpha)=(0,\alpha)\in\cone(\D^\prime\cap(\S^{q-1}\times\R))-K$.} Therefore, \eqref{eq:Dprime} holds. %we have $\cone(\D^\prime\cap(\S^{q-1}\times\R))-K=\D^\prime$.
		
		Finally, \eqref{eq:inclzero}, \eqref{eq:incl} and \eqref{eq:Dprime} imply $\cone((\D_\epsilon \cap (\S^{q-1}\times\R))+ \epsilon \{e^{q+1}\})- K \supseteq\D$.
		%As $(\D_\epsilon\cap(\S{q-1}\times\R))+ \epsilon \{e^{q+1}\} \supseteq\D^\prime\cap(\S^{q-1}\times\R) $, we have $\cone\big[(\D_\epsilon\cap(\S^{q-1}\times\R))+ \epsilon \{e^{q+1}\}\big]-K \supseteq\D^\prime$. As $\D^\prime\supseteq\D $, we have $\cone\big[(\D_\epsilon\cap(\S^{q-1}\times\R))+ \epsilon \{e^{q+1}\}\big]-K \supseteq\D$. Therefore, $\D_\epsilon $ is an inner $\epsilon$-approximation of the lower image.
	\end{proof}
	
	\begin{proposition}\label{prop:12}
		If the primal algorithm terminates, then it returns a finite $\epsilon$-solution $\bar{\W} $ to \eqref{D}.
	\end{proposition}
	
	\begin{proof}
		By the structure of the algorithm, $\bar{\W}\neq\emptyset$ and it consists of maximizers by \Cref{prop:WS_Kmax}. Also the inclusion $\bar{\W}\subseteq\W\cap \S^{q-1}$ holds, see line 10 of \Cref{alg1}. To prove the statement, it is sufficient to show that $\D \subseteq  \cone(\conv \xi(\bar{\W}) + \epsilon \{e^{q+1}\} ) - K$. 
		
		%	\Fir{(In \Cref{def:solnDual}, we have $\D \subseteq  \cone(\conv \xi(\bar{\W}) + \epsilon \{e^{q+1}\} ) - K$. Below, I prove the statement assuming that the one in the definition is correct.)}
		
		Let $\bar{\D}\coloneqq\cone \conv\xi(\bar{\W})-K $, $\bar{\P}\coloneqq\P_{\bar{\D}}$. %= \{y \in \R^q\mid \forall(w^\T,\alpha)^\T \in\bar{\D} \colon \varphi(y,w,\alpha)\geq0 \} $. 
		By \Cref{lem:6}, 
		$%\[
		\bar{\P}= \{y \in \R^q \mid \forall(w^\T,\alpha)^\T \in \xi(\bar{\W})\colon \varphi(y,w,\alpha)\geq0 \}.
		$ %\]
		Though not part of the original algorithm, we introduce an alternative for $\bar{\W}$ that is updated only when the current vertex $v$ is sufficiently far from the upper image, i.e., when $\norm{z^v}>\epsilon$. More precisely, let us introduce a set $\bar{\bar{\W}}$ that is initialized as $\bar{\bar{\W}}=\{w^1, \ldots, w^J\}$ in line 3 of \Cref{alg1} and updated as $\bar{\bar{\W}}\gets\bar{\bar{\W}} \cup \{\frac{w^v}{\|w^v\|_\ast}\}$ in line 12 throughout the algorithm. Observe that $\bar{\bar{\W}} \subseteq \bar{\W}$. %\Cag{$\bar{\bar{\W}}$ arbitrary bir küme mi, buna neden ihtiyacımız var? Algoritma durduğunda zaten $\bar{\W}$'ı döndürmüyor mu? Dolayısıyla $\P_K$'nin aşağıdaki ifadesinde $\bar{\W}$ kullanmamız gerekir.}
		%\Fir{Bir sonraki paragrafta $P_K$'yı $\bar{\bar{\W}}$ kullanarak yazdığımız ifadeyi $\P_K\subseteq \P_\epsilon$ olduğunu göstermek için kullanıyoruz diye anlıyorum..}
		
		Suppose that the algorithm terminates at the $\bar{k}^{\text{th}}$ iteration and let $\P_{\bar{k}}$ denote the resulting outer approximation of the upper image. By the construction of $\bar{\bar{\W}}$, 
		$%\[
		\P_{\bar{k}}= \{y \in \R^q \mid \forall(w^\T,\alpha)^\T \in\xi(\bar{\bar{\W}})\colon \varphi(y,w,\alpha)\geq0 \}.
		$ %\]
		Since $\bar{\bar{\W}} \subseteq \bar{\W}$, we get $\bar{\P}\subseteq \P_{\bar{k}}$.
		
		Let us define
		$%\[
		\P_\epsilon\coloneqq \conv f(\bar{\X})+C+B(0,\epsilon).
		$ %\]
		%and $\P_K$ being the outer approximation that the algorithm returns, 
		By the structure and the termination criterion of the algorithm, for every vertex $v\in\V_{\bar{k}}$ of $\P_{\bar{k}}$, the scalarization problem \eqref{PS(v)} is solved; $x^v$ is added to $\bar{\X}$; and we have $\norm{z^v}\leq\epsilon$. Moreover, $v+z^v \in \{f(x^v)\}+C \subseteq \conv f(\bar{\X}) +C$ holds for every $v\in \V_{\bar{k}}$. Thus, $\V_{\bar{k}} \subseteq \P_\epsilon$. From \cite[Lemma 5.2]{umer_2022}, the recession cone of $\P_{\bar{k}}$ is the ordering cone; hence, $\P_{\bar{k}} = \conv \V_{\bar{k}}+C$. Moreover, $\P_\epsilon$ is a convex upper set. As a result, we have $\P_{\bar{k}} \subseteq \P_\epsilon$. Together with $\bar{\P}\subseteq \P_{\bar{k}}$, the last inclusion implies that $\bar{\P} \subseteq \P_\epsilon$.
		
		%since $f(\bar{\X})+B(0,\epsilon) $ contains all of the vertices of $\P_{\bar{k}}$ and $\norm{z^v}\leq\epsilon $ holds for all vertices. \Fir{(I don't understand what do we mean by the last sentence. We need to show clearly why $\P_{\bar{k}}\subseteq \P_\epsilon$ holds.)} Therefore, the inclusions $\bar{\P}\subseteq \P_{\bar{k}} \subseteq \P_\epsilon$ hold. %\Cag{Bu paragraf ispatın devamında nasil kullanılıyor, göremedim.} \Fir{($\bar{\P} \subseteq \P_\epsilon$ kullanılıyor. Alttaki paragrafa ekledim daha açık olması için.)}
		
		Define $\D_\epsilon\coloneqq \D_{\P_\epsilon}=\{(w^\T,\alpha)^\T \in \R^{q+1} \mid \forall y \in \P_\epsilon\colon  \varphi(y,w,\alpha)\geq0\} $. By \rev{\Cref{cor:special}(a)} and since $\bar{\P} \subseteq \P_\epsilon$, we have $\bar{\D}\supseteq \D_\epsilon$. Moreover, using \Cref{prop:11} and \Cref{prop:9}, $\cone((\D_\epsilon\cap (\S^{q-1}\times\R))+ \epsilon \{e^{q+1}\} )- K \supseteq\D$. With $\bar{\D}\supseteq\D_\epsilon$, this implies 
		\begin{equation}\label{eq:prop12_0}
		\cone\left(\left(\bar{\D}\cap (\S^{q-1}\times\R)\right)+ \epsilon \{e^{q+1}\} \right)- K \supseteq\D.
		\end{equation}
		It is straightforward to check that %We claim that
		$%\begin{equation}\label{eq:prop12_1}
		(\bar{\D}\cap (\S^{q-1}\times\R))+ \epsilon \{e^{q+1}\} \subseteq\cone(\conv \xi(\bar{\W})+\epsilon \{e^{q+1}\} ) - K,
		$ %\end{equation}
		%Let $(\bar{w},\bar{\alpha})\in (\bar{\D}\cap (\S^{q-1}\times\R))+ \epsilon \{e^{q+1}\}$, which can be written as 
		%\[
		%(\bar{w}^\T,\bar{\alpha})^\T=\sum_{i=1}^n\lambda_i((w^i)^\T,p^i)^\T-a e^{q+1}+\epsilon e^{q+1}
		%\]
		%for some $a\geq 0$, $n\in\N$, $\lambda_i\geq 0$,  $w^i\in\bar{\W} $ for $i\in\{1,\ldots,n\}$, satisfying $\norm{\sum_{i=1}^n\lambda_i w^i}_\ast=1$, where $p^i\coloneqq\inf_{x\in\X}(w^i)^\T f(x)$ for $i\in\{1,\ldots,n\}$. Then, we can write
		%\begin{equation*}
		%\begin{aligned}
		%(\bar{w}^\T,\bar{\alpha})^\T &=  \bar{\lambda} \left(\sum_{i=1}^n \frac{\lambda_i}{\bar{\lambda}} ((w^i)^\T,p^i)^\T +\epsilon e^{q+1} \right)-a e^{q+1} -(\bar{\lambda}-1)\epsilon e^{q+1},
		%\end{aligned}
		%\end{equation*}
		%where $\bar{\lambda}\coloneqq\sum_{i=1}^n \lambda_i$. Noting that
		%$%\[
		%1 = \norm{\sum_{i=1}^n\lambda_i w^i}_\ast \leq \sum_{i=1}^n \lambda_i \norm{w^i}_\ast=\bar{\lambda},
		%$ %\]
		%we have $(\bar{w}^\T,\bar{\alpha})^\T \in \cone(\conv\xi(\bar{\W})+\epsilon \{e^{q+1}\}) - K$. Hence, \eqref{eq:prop12_1} holds.
		which implies that 
		\[
		\cone\left(\left(\bar{\D}\cap (\S^{q-1}\times\R)\right)+ \epsilon \{e^{q+1}\} \right)- K\subseteq\cone(\conv \xi(\bar{\W})+\epsilon \{e^{q+1}\} ) - K.
		\]
		By \eqref{eq:prop12_0}, $\cone(\conv \xi(\bar{\W})+\epsilon \{e^{q+1}\}) - K  \supseteq \D$; hence, $\bar{\W} $ is a finite $\epsilon$-solution to \eqref{D}.
		%\Fir{(Eskisini comment out yapmadım, bana doğru gelmediği için yukarıda yeniden ispatladım. Yukarıdaki doğruysa onu silebiliriz.)}
		%
		%	 
		%	 Let $\sum_{i=1}^n\lambda_i((w^i)^\T,p^i)^\T-a e^{q+1} \in(\bar{\D}\cap (\S^{q-1}\times\R)) $ where $a e^{q+1}\in K $, $\norm{\sum_{i=1}^n\lambda_i w^i}=1$, $\lambda_i\geq0 $, $w^i\in\bar{\W} $, $p^i=\inf_{x\in\X}(w^i)^\T f(x) $ for $i=1,\dots,n$. Hence, we have the following statements
		%	
		%	\begin{equation*}
		%	\begin{aligned}
		%	\sum_{i=1}^n\lambda_i ((w^i)^\T,p^i)^\T -ae^{q+1} & \in \bar{\D}\cap (\S^{q-1}\times\R) \\
		%	\sum_{i=1}^{n}\lambda_i ((w^i)^\T,p^i+\epsilon)^\T-ae^{q+1} &\in ( \bar{\D}\cap (\S^{q-1}\times\mathbb{R}))+\epsilon \{e^{q+1}\}\\
		%	\sum_{i=1}^{n}\lambda_i ((w^i)^\T,p^i+\epsilon)^\T-ae^{q+1} &\in\cone(\conv(\xi(\bar{\W})+\epsilon \{e^{q+1}\})) - K.
		%	\end{aligned}
		%	\end{equation*}
		%	Therefore, the inclusion $(\bar{\D}\cap (\S^{q-1}\times\R))+ \epsilon \{e^{q+1}\}\subseteq\cone(\conv(\xi(\bar{\W})+\epsilon \{e^{q+1}\})) - K$ holds and we have $\cone\left(\left(\bar{\D}\cap (\S^{q-1}\times\R)\right)+ \epsilon \{e^{q+1}\} \right)- K\subseteq\cone(\conv(\xi(\bar{\W})+\epsilon \{e^{q+1}\})) - K$. 
	\end{proof}

	\subsection{Dual algorithm} \label{subsect:dualalg}
	
	In this section, we describe a geometric dual algorithm for solving problems \eqref{P} and \eqref{D}. The main idea is to construct outer approximations of the lower image $\D$, iteratively. Recall that, by \Cref{prop:Dcone}, $\D$ is a closed convex cone; similarly, we will see that the outer approximations found through the iterations of the dual algorithm are polyhedral convex cones.
	
	The dual algorithm (\Cref{alg2}) is initialized by solving a weighted sum scalarization for some weight vector from $\Int C^+$. In particular, we solve ($\text{WS}(w^0)$) by taking %\sout{$\bar{w}=\sum_{j=1}^{J}w^j$, $w^0\coloneqq\frac{\bar{w}}{\norm{\bar{w}}_\ast}$} 
	%\[
	$w^0\coloneqq \sum_{j=1}^Jw^j / \|\sum_{j=1}^Jw^j\|_\ast$.
	%\]
	By Propositions \ref{prop:WS} and \ref{prop:WS_Kmax}, an optimal solution $x^0$ of ($\text{WS}(w^0)$) is a weak minimizer for \eqref{P} and $w^0$ is a maximizer for \eqref{D}. Hence, we set $\bar{\X} = \{x^0\}, \bar{\W} = \{w^0\}$. Moreover, using \Cref{prop:supp_halfsp} and the definition of the lower image $\D$, we define the initial outer approximation of $\D$ as
	%\[
	$\D_0\coloneqq \H^\ast(f(x^0))\cap (C^+ \times \R) \supseteq \D$; 
	%\]
	see lines 1-3 of \Cref{alg2}. Note that $\D_0$ satisfies $\D_0 = \D_0 - K$ and, under \Cref{ass:1}, it is a polyhedral convex cone.
	
	Throughout the algorithm, weighted sum scalarizations will be solved for some weight vectors from $\W$. In order to keep track of the already used $w\in\W$, we keep a list $\Wkn$ and initialize it as the empty set. 
	%finding $\bar{w} $ by the summation of the generating vectors of the dual cone $C^+ $ and by finding an optimal solution $x^{\bar{w}} $ and optimal value $p^{\bar{w}} $ to the problem \eqref{WS(w)} for $\bar{w} $. Then the optimal solution $x^{\bar{w}} $ is used to find the initial outer approximation $\mathcal{D}_0 $ of $\mathcal{D} $. The sets $\bar{\mathcal{X}} $, $\bar{\mathcal{W}} $ and $D$ are initialized by the found $x^{\bar{w}} $, $\bar{w} $ and $(\bar{w}^T,p^{\bar{w}})^T $ values. Therefore, initialization of the algorithm is completed. 
	
	In each iteration $k$, first, the set $\D^{\text{dir}}_k$ of extreme directions of the current outer approximation $\D_k$ is computed (line 6). The extreme directions in $\D^{\text{dir}}_k \setminus (\{0\}\times\R)$ are normalized such that $\norm{w}_\ast = 1$, and $\D^{\text{dir}}_k$ is updated such that it is a subset of $\S^{q-1} \times \R$ (line 7). It will be seen that $\D_k$ is constructed by intersecting $\D_0$ with a set of halfspaces of the form $\H^\ast (f(x))$, where $x$ is a weak minimizer. Then, by the definitions of $\D_0$ and $\H^\ast (\cdot)$, it is clear that $\D_k = \D_k -K$; moreover, $e^{q+1}$ is not a recession direction for $\D_k$. %\Fir{(Bu yeterince açık mı? Gerçekten $\H^\ast$ tanımından geliyor. İki potansiyel extreme direction elenmiş olabilir $\D^{\text{dir}}_k $ kümesinden, $\pm e^{q+1}$. Biri kesin kümenin içinde ve diğer de kesinlikle içinde değil. Çok ayrıntı olacak diye yazmadım.)}
	Hence, we have
	\begin{equation}\label{eq:Dk-K}
	\D_k = \cone\conv (\D^{\text{dir}}_k\cup\{-e^{q+1}\}) = \cone\conv \D^{\text{dir}}_k - K.
	\end{equation}
	
	For each unknown extreme direction $(w^\T,\alpha)^\T$, %after normalizing it such that $\norm{w}_\ast = 1$, 
	an optimal solution $x^w$ and the optimal value $p^w$ of \eqref{WS(w)} are computed (lines 8-10). Recall that $x^w$ is a weak minimizer by \Cref{prop:WS} and $\xi(w)=(w,p^w)$ is a $K$-maximal element of $\D$ by \Cref{prop:WS_Kmax}; we update $\bar{\X}$ and $\bar{\W}$ accordingly, and add $(w^\T,\alpha)^\T$ to the set of known extreme directions (line 11). Since $(w^\T,\alpha)^\T\in \D_k$ and $\D_k\supseteq\D$, we have $\alpha \geq p^w$. If the difference between $\alpha$ and $p^w$ is greater than the allowed error $\epsilon$, then the supporting halfspace $\H^\ast(f(x^w))$ of $\D$ is computed and stored. Once all unknown extreme directions are explored, the current outer approximation is updated by using the stored supporting halfspaces (lines 12-14 and 17). The algorithm terminates when every unknown extreme direction $(w^\T,\alpha)^\T$ is sufficiently close to the lower image in terms of the ``vertical distance" $\alpha-p^w$ (line 19).

	\setcounter{algo}{2}
	\begin{algorithm}[h]
		\caption{Dual algorithm}
		\begin{algorithmic} [1] 
			\STATE Compute an optimal solution $x^0$ to $(\text{WS}(w^0))$ for $w^0=\frac{\sum_{j=1}^Jw^j}{\norm{\sum_{j=1}^Jw^j}_\ast}$;
			\STATE %Store an H-representation $\mathcal{D}^H$ of
			Let $\D_0=\H^\ast(f(x^0))\cap (C^+ \times \R)$; %$\cap\{(w^\T,\alpha)^\T \in \R^{q+1} \mid \varphi(f(x^{0}),w,\alpha)\geq 0 \}$};
			\STATE $k \gets 0, \bar{\X}\gets\{x^0\}, \bar{\W}\gets\{w^0\}, \Dkn=\emptyset$; %\Cag{(or $\Dkn=\{(w^0,p^{w^0})\}$???)} 
			%\Fir{($\D_0$ tanımlayıp aşağıda 6. satırda $\D_k$ aldığımız için k=0 olarak değiştirdim.)} %\Cag{(Belki de $\alpha$'yı ne alacağımız belli olmadığı için eklememişizdir (line 11'i kontrol etmediğimiz için).)}
			%$\Wunkn=\emptyset $;
			%$D\leftarrow \xi(\bar{\mathcal{W}})$	where $\xi_{q+1}(\bar{w})=\bar{w}^Tf(x^{\bar{w}})$;
			\REPEAT
			\STATE $M\gets\R^{q+1}$;
			\STATE Compute the set $\D^{\text{dir}}_k$ of extreme directions of $\D_k$;
			\STATE % from its H-representation $\mathcal{D}^H$;
			$\D^{\text{dir}}_k\gets\{\frac{(w^\T,\alpha)^\T}{\norm{w}_\ast} \in\R^{q+1} \mid (w^\T,\alpha)^\T\in \D^{\text{dir}}_k \setminus (\{0\}\times \R)\}$; %\\\Fir{($\Dkn$ içindeki w'lar normalize edildiği için buraya aldım, küme farkını alırken bunun yapılmış olması lazım.)}\Cag{(Kodda bunu kontrol edelim.)} %\Fir{Burada sanırım $-e^{q+1}$ vektörü de extreme direction olacak, o yüzden yukarıdaki güncellemeyi $\D^{\text{dir}}_k\gets\{\frac{(w,\alpha)}{\norm{w}_\ast} \in\R^{q+1} \mid (w,\alpha)\in \D^{\text{dir}}_k\setminus \{0\}\times \R \}$ olarak yapmak gerekebilir. Şu anda $w=0$ durumunda kod ne yapıyor?.}
			\STATE $\Dunkn\gets\D^{\text{dir}}_k\setminus\Dkn$;
			\FOR {$(w^\T,\alpha)^\T\in\Dunkn$}
			\STATE %\sout{Let $(w^T,\alpha)^T$ be the $i^{th}$ element of $\mathcal{W}_{unknown}$ (i.e. the $i^{th}$ unknown extreme direction of $\mathcal{D}_k$);}\\
			%\Fir{\sout{Let $(w^\T,\alpha)\gets\frac{(w^\T,\alpha)}{\norm{w}_\ast}$, }} 
			Compute an optimal solution $x^w$ to \eqref{WS(w)} and let $p^w\coloneqq w^\T f(x^w)$;
			\STATE $\bar{\X} \gets \bar{\X} \cup \{x^w\}$, $\bar{\W} \gets \bar{\W} \cup \{{w}\}$, $\Dkn \gets \Dkn \cup \{(w^\T,\alpha)^\T\}$; 
			%		\IF {\sout{$\alpha-p^w \leq \epsilon$}}
			%		\STATE \sout{$\bar{\W} \gets \bar{\W} \cup \{{w}\}$;} 
			\IF {$\alpha-p^w > \epsilon$}
			\STATE $M \gets M \cap \H^\ast(f(x^w))$;
			%		\ELSE
			%		\STATE \Cag{\sout{$M \gets M \cap \H^\ast(f(x^w))$}}; \Cag{(Else de silinecek. En son text'teki satır numaralarını kontrol edelim.)}
			\ENDIF 
			\ENDFOR
			\IF {$M\neq\R^{q+1}$}
			\STATE $\D_{k+1}=\D_k\cap M, k \gets k+1$;
			\ENDIF
			\UNTIL $M=\R^{q+1}$
			%		\STATE Compute the vertices $\mathcal{V}$ of $\{y \in \mathbb{R}^q:\varphi(y,w,\alpha)\geq 0 , \forall (w^T,\alpha)^T \in D=\xi(\bar{\mathcal{W}}) \}$
			\RETURN 		
			$\left\{\begin{array}{ll}\bar{\X}&\text{: A finite weak }\tilde{\epsilon} \text{-solution to \eqref{P}};   \\
			\bar{\W}&\text{: A finite }\epsilon \text{-solution to \eqref{D}}; 
			%		\mathcal{V}&\text{: Vertices of an outer set of } \mathcal{P}\\
			%		f(\bar{\mathcal{X}}) & \text{: Vertices of an inner set of } \mathcal{P} \\
			\end{array}\right.	$
		\end{algorithmic}
		\label{alg2}
	\end{algorithm}
	
	\begin{remark}
		A `break' command can be placed between lines 13 and 14 in the algorithm. With the current version, the algorithm goes through all the extreme directions of the current outer approximation without updating it. With the `break' command, the algorithm would update the outer approximation as soon as it detects an extreme direction $w$ with $\alpha - p^w >\epsilon$.
	\end{remark}
	
	\rev{The following lemma will be useful in proving Propositions \ref{prop:5} and \ref{prop:6}.}
	\rev{\begin{lemma} \label{lem:D_inclusions}
			For $\epsilon>0$, suppose that the algorithm terminates at the $\bar{k}^{\text{th}}$ iteration, returns sets $\bar{\X}, \bar{\W}$. Let $%\[
			\bar{\P}\coloneqq \conv f(\bar{\X})+C,$
			%\qquad 
			$\bar{\D}\coloneqq\D_{\bar{\P}}
			$ %\]
			and
			$%\[
			\D_\epsilon\coloneqq\cone(\conv\xi(\bar{\W})+\epsilon \{e^{q+1}\})-K.
			$ Then, $\D_\epsilon \supseteq \D_{\bar{k}} \supseteq\bar{\D}\supseteq\D$, where $\D_{\bar{k}}$ denotes the resulting outer approximation of $\D$ at termination.
		\end{lemma}
		\begin{proof}
			First observe that $\bar{\P} \subseteq \P$ implies $\bar{\D} \supseteq \D$ from \Cref{prop:DPequalsD}. Moreover, $\bar{\D}\subseteq \{(w^\T,\alpha)^\T \in \R^{q+1}\mid \forall y \in f(\bar{\X})\colon \varphi(y,w,\alpha)\geq 0 \} $ holds since $\bar{\P} \supseteq f(\bar{\X})$. 
			%Suppose that the algorithm terminates at the $\bar{k}^{\text{th}}$ iteration. In particular, $\D_{\bar{k}}$ is the outer approximation of $\D$ returned by the algorithm. 
			Consider an artificial set $\bar{\bar{\X}}$ that is initialized in line 3 of \Cref{alg2} as the empty set, and updated in line 13 as $\bar{\bar{\X}} \gets \bar{\bar{\X}} \cup \{x^w\}$. By the definition of $\bar{\bar{\X}}$, we have
			$%\[
			\D_{\bar{k}}=\{(w^\T,\alpha)^\T \in \R^{q+1}\mid\forall y \in f(\bar{\bar{\X}})\colon \varphi(y,w,\alpha)\geq 0\}
			$ %\]
			and $\bar{\bar{\X}} \subseteq \bar{\X}$. Therefore, $\D_{\bar{k}}\supseteq \{(w^\T,\alpha)^\T \in \R^{q+1}\mid \forall y \in f(\bar{\X})\colon \varphi(y,w,\alpha)\geq 0 \}$. %We will show that $\D_{\bar{k}} \subseteq \D_\epsilon $, where $\D_\epsilon$ is defined by $\D_\epsilon\coloneqq \cone(\conv \xi(\bar{\W})+\epsilon \{e^{q+1}\}) - K $. 
			Note that for every $(w^\T,\alpha)^\T\in\D^{\textnormal{dir}}_{\bar{k}} \subseteq \S^{q-1}\times \R$, the scalarization problem \eqref{WS(w)} is solved; $w$ is added to $\bar{\W}$ and we have $\alpha - p^w\leq\epsilon$. Then, we have $\D^{\textnormal{dir}}_{\bar{k}}\subseteq  \xi(\bar{\W})+\epsilon\{e^{q+1}\}-K$. On the other hand, $\D_{\bar{k}} = \cone \conv\D^{\textnormal{dir}}_{\bar{k}} - K$, see \eqref{eq:Dk-K}. These imply
			$%\[
			\D_{\bar{k}} \subseteq \cone \conv (\xi(\bar{\W}) + \epsilon\{e^{q+1}\})-K = \D_{\epsilon}.
			$ %\]
			Therefore,
			%\begin{equation*}
			$\bar{\D} \subseteq \{(w^\T,\alpha)^\T \in \R^{q+1}\mid \forall y \in f(\bar{\X})\colon \varphi(y,w,\alpha)\geq 0 \} \subseteq \D_{\bar{k}} \subseteq \D_\epsilon.$
			%	\end{equation*}
		\end{proof}
	}

	Next, we prove that the dual algorithm returns a finite $\epsilon$-solution to problem \eqref{D}. %The following lemma will be useful to prove this. 

	%\Fir{(Makaleye Lemma'yı koymayabiliriz belki?)}\Cag{(Bence de koymayalım, çok basit bir result. Doğrudan kullanabiliriz. Lemmayı comment out ettim.)}
	
	%\begin{lemma}\label{lem:1}
	%	Let $A \subseteq \mathbb{R}^q$ and $c \in \mathbb{R}^q $, then $\conv A+\{c\} = \conv( A+ \{c\})$ holds.
	%\end{lemma}
	%
	%\begin{proof}
	%	Let $w \in \conv A+\{c\} $. Then, there exists $N \in \N$, $\lambda_i \geq 0$ with $\sum_{i=1}^{N}\lambda_i=1$, $a_i \in A $ for $i\in\{1,\ldots,N\}$ such that $w= \sum_{i=1}^{N}\lambda_ia_i +c$. Note that $w=\sum_{i=1}^{N}\lambda_i(a_i+c)$, therefore, $\ w \in  \conv(A+\{c\})$. Conversely, let $w \in \conv(A+\{c\})$. Similar to the previous implication, we have $w=\sum_{i=1}^{\bar{N}}\bar{\lambda}_i(\bar{a}_i+c)$ for some $\bar{N} \in \N$, $\bar{\lambda}_i \geq 0$ with $\sum_{i=1}^{\bar{N}}\bar{\lambda}_i=1$, $\bar{a}_i \in A$ for $i\in\{1,\dots,\bar{N}\}$. Then, $w= \sum_{i=1}^{\bar{N}}\bar{\lambda}_ia_i +c \in \conv A+\{c\}$. %Therefore, for all $w \in \conv(A+c), \ w \in  \conv(A)+c $. Hence, the equality holds.
	%\end{proof}
	
	%The following two propositions show that the dual approximation algorithm gives finite $\epsilon $-solutions $\bar{\mathcal{W}} $ and $\bar{\mathcal{X}} $ to problems \eqref{D} and \eqref{P} respectively.
	\begin{proposition}\label{prop:5}
		If \Cref{alg2} terminates, then it returns a finite $\epsilon$-solution $\bar{\W}$ to \eqref{D}.
	\end{proposition}
	
	\begin{proof}
		\rev{By the structure of the algorithm and \Cref{prop:WS_Kmax}, $\bar{\W} $ is nonempty and consists of maximizers. Also the inclusion $\bar{\W}\subseteq\W\cap \S^{q-1} $ holds. From \Cref{lem:D_inclusions}, we have 
			$\D \subseteq  \cone(\conv\xi(\bar{\W})+\epsilon \{e^{q+1}\}) - K$, which implies that $\bar{\W}$ is a finite $\epsilon$-solution to \eqref{D}.}
	\end{proof}

	The next step is to prove that the algorithm returns also a solution to the primal problem \eqref{P}. In order to prove this result, the following lemma and propositions will be useful. To that end, for each $n\in\{2,3,\ldots\}$, let us define 
	\[
	\Delta^{n-1}\coloneqq\bigg\{\lambda\in\R^n_+ \mid \sum_{j=1}^n \lambda_j = 1\bigg\},\qquad \Delta_+^{n-1}\coloneqq\bigg\{\lambda\in\R^n_+ \mid \sum_{j=1}^n \lambda_j \geq 1\bigg\}.
	\]

	\begin{lemma}\label{lem:5}
		Let $d^1,\dots,d^n $ be the generating vectors of a pointed convex \rev{polyhedral} cone $D\subseteq\R^q$, where $n\geq 2$.
		%\begin{itemize}
		(a) For every $\lambda\in\Delta^{n-1}$, $\sum_{j=1}^{n}\lambda_jd^j \neq 0$ holds.
		(b) It holds $\min_{\lambda\in\Delta^{n-1}} \|\sum_{j=1}^{n}\lambda_jd^j\|_\ast > 0$.
		(c) It holds
		%\[
		$\min_{\lambda\in\Delta^{n-1}} \|\sum_{j=1}^{n}\lambda_jd^j\|_\ast =\min_{\lambda\in\Delta_+^{n-1}} \|\sum_{j=1}^{n}\lambda_jd^j\|_\ast$. 
		%\]
	\end{lemma}
	
	\begin{proof}
		%	\begin{itemize}
		%	\item[(a)] 
		To prove (a), let $\lambda\in\Delta^{n-1}$. Assume to the contrary that $\sum_{j=1}^{n}\lambda_jd^j=0$. Since $\sum_{j=1}^{n}\lambda_j=1 $, we have $\lambda_j > 0$ for at least one $j\in\{1,\ldots,n\}$. Without loss of generality, we may assume that $j=1$. Hence, %$\sum_{j=2}^{n}\lambda_jd^j=-\lambda_1d^1$, that is,
		$%\[
		\frac{1}{\lambda_1}\sum_{j=2}^{n}\lambda_jd^j=-d^1.
		$ %\]
		Since $D $ is a convex cone, $\frac{1}{\lambda_1}\sum_{j=2}^{n}\lambda_jd^j \in D$. Therefore, $d^1,-d^1\in D $, contradicting the pointedness of $D$ \cite[Section 2.4.1]{boyd_2004}. Therefore, $\sum_{j=1}^{n}\lambda_jd^j \neq 0$.
		
		By (a), %we know $\sum_{j=1}^{J}\lambda_jd^j \neq 0 $, which means $\|{\sum_{j=1}^{J}\lambda_jd^j}\| \neq 0 $. We also know that $\|{\sum_{j=1}^{J}\lambda_jd^j}\|\geq 0 $. Therefore, we have
		$\|{\sum_{j=1}^{n}\lambda_jd^j}\|_\ast>0$ for every $\lambda\in\Delta^{n-1}$. Since the feasible set $\Delta^{n-1}$ is compact, the minimum of the continuous function $\lambda\mapsto \|\sum_{j=1}^n \lambda_j d^j\|_\ast$ is attained at some $\bar{\lambda}\in \Delta^{n-1}$. Hence, the minimum $ \|\sum_{j=1}^n \bar{\lambda}_j d^j\|_\ast$ is also strictly positive. This proves (b).
		%\item[(c)]
		
		Finally, we prove (c). Since $\Delta^{n-1}\subseteq\Delta^{n-1}_+$, we have
		$%\[
		\min_{\lambda\in\Delta^{n-1}} \|\sum_{j=1}^{n}\lambda_jd^j\|_\ast \geq \inf_{\lambda\in\Delta_+^{n-1}} \|\sum_{j=1}^{n}\lambda_jd^j\|_\ast.
		$ %\]
		To prove the reverse inequality, assume to the contrary that there exists $\bar{\lambda}\in\Delta^{n-1}_+$ such that
		%\begin{equation}\label{eq:lambda}
		$\|\sum_{j=1}^{n}\bar{\lambda}_jd^j\|_\ast<\min_{\lambda \in\Delta^{n-1}}\|\sum_{j=1}^n\lambda_j d^j\|_\ast$.
		%\end{equation}
		%		In particular, we  must have $\sum_{j=1}^{n}\bar{\lambda}_j>1$. \Cag{(Bu bilgiyi kullanmıyoruz sanki, öyleyse atabiliriz.)} \Fir{(Aşağıdaki ilk eşitsizliğin yazımında kullanıyoruz sanırım, inequality strict oluyor hatta.)} \Cag{(Ama contradiction için nonstrict gerekmiyor, yukarıdaki inequality strict çünkü.)} 
		Then, using $\bar{\lambda}\in\Delta^{n-1}_+$, we have
		\begin{equation*}
		\norm{\sum_{j=1}^{n}\bar{\lambda}_jd^j}_\ast \geq \frac{1}{\sum_{i=1}^{n}\bar{\lambda}_i} \norm{\sum_{j=1}^{n}\bar{\lambda}_jd^j}_\ast=\norm{\sum_{j=1}^{n}\frac{\bar{\lambda}_j}{\sum_{i=1}^{n}\bar{\lambda}_i}d^j}_\ast \geq \min_{\lambda\in\Delta^{n-1}}\norm{\sum_{j=1}^{n}\lambda_jd^j}_\ast,
		\end{equation*}
		%where the last inequality holds since $\sum_{j=1}^{n}\frac{\bar{\lambda}_j}{\sum_{i=1}^{n}\bar{\lambda}_i}=1$. 
		a contradiction. %to \eqref{eq:lambda}.
		Therefore, 
		$%\[
		\min_{\lambda\in\Delta^{n-1}} \|\sum_{j=1}^{n}\lambda_jd^j\|_\ast = \inf_{\lambda\in\Delta_+^{n-1}} \|\sum_{j=1}^{n}\lambda_jd^j\|_\ast.
		$ %\]
		Since the minimum is attained on the left and $\Delta^{n-1}\subseteq\Delta_+^{n-1}$, the infimum on the right is also a minimum.
		%	\end{itemize}
	\end{proof}
	
	Next, we show that an inner $\tilde{\epsilon}$-approximation of the upper image $\mathcal{P} $ can be obtained by using a finite $\epsilon $-solution $\mathcal{\bar{W}} $ of \eqref{D}.
	
	\begin{proposition}\label{prop:eps_tilde_1}
		For $\epsilon>0$, let $\bar{\mathcal{W}}$ be a finite $\epsilon$-solution of \eqref{D}, and define
		$%\[
		\D_\epsilon:= \cone(\conv \xi(\bar{\W})+\epsilon \{e^{q+1}\}) - K.
		$ %\]
		Then, $\P_\epsilon\coloneqq\P_{\D_\epsilon}$ is an inner approximation of $\P$ and $\P_\epsilon + B(0,\tilde{\epsilon}) \supseteq \P \supseteq \P_\epsilon $, where
		%\begin{equation*}
		$\tilde{\epsilon}=\epsilon/\min_{\lambda\in\Delta^{J-1}}\|\sum_{j=1}^{J}\lambda_j w^j\|_\ast$.
		%\end{equation*}
		%	and $w^1,\dots w^J $ are the generating vectors of $C^+ $.
	\end{proposition}
	%\Fir{Asagıdaki ispat tamamlanmamış gibi duruyor, ya da ben anlamadım. Sonunda bir contradiction çıkması lazım sanki. }
	\begin{proof}
		Since $\bar{\W} $ is a finite $\epsilon $-solution of \eqref{D}, $\D_\epsilon $ is an outer approximation of the lower image $\D$ by Definition \ref{def:solnDual}, that is, $\D_\epsilon\supseteq\D$. Therefore, by \Cref{prop:DPequalsD}(b), the inclusion $\P\supseteq\P_\epsilon $ holds.
		
		Next, we show that $\P_\epsilon + B(0,\tilde{\epsilon}) \supseteq \P$. Assume that there exists $\bar{y}\in \P\setminus (\P_\epsilon+B(0,\tilde{\epsilon}))$. Hence, there exists $\bar{w}\in \R^q\setminus\{0\} $ such that 
		%\begin{equation}\label{eq:separation}
		$\bar{w}^{\T}\bar{y}< \inf_{y\in\P_\epsilon}\bar{w}^{\T}y + \inf_{\gamma\in B(0,1)}\tilde{\epsilon}\bar{w}^\T\gamma$. 
		%\end{equation}
		Without loss of generality, we may assume that $\|\bar{w}\|_\ast=1$ so that
		$%\[
		\inf_{\gamma\in B(0,1)}\tilde{\epsilon}\bar{w}^\T\gamma=-\tilde{\epsilon}\norm{\bar{w}}_\ast=-\tilde{\epsilon}.
		$ %\]
		Hence, we have %\eqref{eq:separation} becomes
		\begin{equation}\label{eq:separation2}
		\bar{w}^\T\bar{y}+\tilde{\epsilon}< \inf_{y\in\P_\epsilon}\bar{w}^{\T}y=:\bar{\alpha}.
		\end{equation}
		The definition of $\bar{\alpha}$ ensures that $\varphi(y,\bar{w},\bar{\alpha})=\bar{w}^{\T}y-\bar{\alpha}\geq 0$ for each $y\in\P_\epsilon$, that is, $(\bar{w}^\T,\bar{\alpha})^{\T}\in \D_{\P_\epsilon}$. By \rev{\Cref{cor:special}(b)}, we have $\D_{\P_\epsilon} =\D_{\P_{\D_\epsilon}}= \D_{\epsilon}$. Therefore, $(\bar{w}^{\T},\bar{\alpha})^\T \in \D_{\epsilon}$. Hence, there exist $\delta\geq 0$, $k\geq0 $, $n\in \N$ and $\mu\in\Delta^{n-1}$, $((\bar{w}^i)^\T,\alpha_i )^{\T}\in \xi(\bar{\W}) $ for each $i\in\{1,\ldots,n\}$ such that %\Fir{($\cone\conv$ olmasını direkt kullanmayı denedim, ispat çok kısaldı, gözden kaçırdığım bir şey yoktur umarım. Buhali Lemma 5.9 (c)'yi kullanmıyor. Tekrar gözden geçireceğim...)}
		\begin{equation}\label{eqn:2}
		(\bar{w}^\T,\bar{\alpha})^\T= \delta\of{\sum_{i=1}^{n}\mu_i((\bar{w}^i)^\T,\alpha_i)^\T+\epsilon e^{q+1} } - k e^{q+1}.
		\end{equation}
		Using \eqref{eqn:2}, we have $\bar{w}=\delta\sum_{i=1}^{n}\mu_i\bar{w}^i $ and $\bar{\alpha}=\delta(\sum_{i=1}^{n}\mu_i\alpha_i+\epsilon )- k$. %and consequently, \[	\norm{\bar{w}}_\ast = \norm{ \sum_{i=1}^{m}\mu_iw_i}_{\ast}=1.\]
		In particular, having $\|\bar{w}\|_\ast=1$ implies that
		$%\begin{equation*}
		\delta=\frac{1}{\norm{\sum_{i=1}^{n}\mu_i\bar{w}^i}}_\ast.
		$ %\end{equation*}
		
		Next, we claim that $(\bar{w}^\T,\bar{\alpha}-\tilde{\epsilon})\in\D$. Since $\W=C^+$ is a convex cone, we have $\bar{w}\in\W$. Therefore, using the definition of $\D$, proving the inequality 
		\begin{equation}\label{eq:enough}
		\bar{\alpha}-\tilde{\epsilon}\leq \inf_{x\in\X}\bar{w}^\T f(x)
		\end{equation}
		is enough to conclude that $(\bar{w}^\T,\bar{\alpha}-\tilde{\epsilon})\in\D$.	
		%\Fir{$(\bar{w}^\T,\bar{\alpha}-\tilde{\epsilon})\in\D$ olduğunu göstermek neden yeterli? $\bar{y}\in \mathcal{P} $ olmasına bir contradiction mı var buradan? Ya da $\bar{y} \in \mathcal{P}_\epsilon + B(0,\tilde{\epsilon})$ mu oluyor? }
		Since $((\bar{w}^i)^\T,\alpha_i)\in\xi(\bar{\W})$, $\alpha_i = \inf_{x\in\X}(\bar{w}^i)^\T f(x)$ for each $i\in\{1,\ldots,n\}$. Hence, using \eqref{eqn:2} and $k\geq 0$, we have
		\begin{align*}
		\bar{\alpha}&=\delta\of{\sum_{i=1}^{n}\mu_i\alpha_i+\epsilon } - k 
		\leq\delta\of{\sum_{i=1}^{n}\mu_i \inf_{x\in\X}(\bar{w}^i)^{\T}f(x)+\epsilon }
		%&\leq\delta\of{\inf_{x\in\X}\of{\sum_{i=1}^{n}\mu_i\bar{w}^i}^\T f(x)}+\delta\epsilon  
		=\inf_{x\in\X}\bar{w}^\T f(x) +\delta\epsilon.
		\end{align*}
		Hence, we have $\bar{\alpha}-\delta\epsilon\leq\inf_{x\in\mathcal{X}}\bar{w}^\T f(x) $.
		
		Let us show that $\delta\epsilon \leq \tilde{\epsilon} $ so that \eqref{eq:enough} follows. For each $i\in\{1,\ldots,n\}$, since $\bar{w}^i\in\bar{\W}\subseteq\W\cap(\S^{q-1}\times\R)$, there exists $(\gamma_{i1},\ldots,\gamma_{iJ})^\T\in\Delta^{J-1}$ such that
		$%\[
		\bar{w}^i=\sum_{j=1}^{J}\frac{\gamma_{ij}}{\|\sum_{j^\prime=1}^{J}\gamma_{ij^\prime}w^{j^\prime}\|_\ast}w^j,
		$ %\]
		where $w^1,\ldots, w^J $ are the generating vectors of $C^+ $. It follows that 
		$\sum_{i=1}^n \mu_i \bar{w}^i =\sum_{j=1}^J \lambda_j w^j$,
		%\[
		%\sum_{i=1}^n \mu_i \bar{w}^i = \sum_{i=1}^n \mu_i\sum_{j=1}^{J}\frac{\gamma_{ij}}{\norm{\sum_{j^\prime=1}^{J}\gamma_{ij^\prime}w^{j^\prime}}_\ast}w^j =\sum_{j=1}^J \lambda_j w^j,
		%\]
		where
		$%\[
		\lambda_j \coloneqq \sum_{i=1}^{n}\mu_i\frac{\gamma_{ij}}{\|\sum_{j^\prime=1}^{J}\gamma_{ij^\prime}w^{j^\prime}\|_\ast}\geq 0
		$ %\]
		for $j\in\{1,\ldots,J\}$. Note that $\sum_{j=1}^J \lambda_j \geq 1$ since $\|\sum_{j=1}^{J}\gamma_{ij}w^j \|_\ast\leq \sum_{j=1}^{J}\|\gamma_{ij}w^j \|_\ast=\sum_{j=1}^{J}\gamma_{ij}\|w^j \|_\ast=1$, $\sum_{j=1}^J \gamma_{ij} = 1$ for each $i\in\{1,\ldots,n\}$ and $\sum_{i=1}^{n}\mu_i=1 $. Therefore, $\lambda=(\lambda_1,\ldots,\lambda_J)^\T\in\Delta_+^{J-1}$. Since the dual cone $C^{+}$ is convex and pointed, we may use \Cref{lem:5} to write
		\[
		\delta\epsilon%= \frac{\epsilon}{\norm{\sum_{i=1}^{n}\mu_i\bar{w}^i}_\ast}
		=\frac{\epsilon}{\norm{\sum_{j=1}^J \lambda_j w^j}_\ast}
		\leq \frac{\epsilon}{\underset{\lambda^\prime\in\Delta_+^{J-1}}{\min}\norm{\sum_{j=1}^{J}\lambda^\prime_jw^j}_\ast} 
		=\frac{\epsilon}{\underset{\lambda^\prime\in\Delta^{J-1}}{\min}\norm{\sum_{j=1}^{J}\lambda^\prime_jw^j}_\ast} 
		=\tilde{\epsilon}.
		\]
		Therefore, \eqref{eq:enough} follows and we have $(\bar{w}^\T,\bar{\alpha}-\tilde{\epsilon})^\T\in \D$. However, by \eqref{eq:separation2}, we have $\varphi(\bar{y},\bar{w},\bar{\alpha}-\tilde{\epsilon})=\bar{w}^\T\bar{y}-\bar{\alpha}+\tilde{\epsilon}<0 $ for $\bar{y}\in\mathcal{P} $ and $(\bar{w}^\T,\bar{\alpha}-\tilde{\epsilon})^\T\in\mathcal{D} $, a contradiction to \Cref{prop:DPequalsD}. Hence, $\P_\epsilon+B(0,\tilde{\epsilon})\supseteq \P$.
	\end{proof}
	
	\rev{The next proposition provides a better bound on the realized approximation error compared to \Cref{prop:eps_tilde_1}; however, it requires post-processing the finite $\epsilon $-solution of \eqref{D} provided by the algorithm. We omit its proof for brevity.}
	
	\begin{proposition}\label{prop:eps_tilde_2}
		For $\epsilon>0 $, let $\bar{\W} $ be a finite $\epsilon $-solution of \eqref{D}. Define 
		$%\[
		\D_\epsilon\coloneqq\cone(\conv\xi(\bar{\W})+\epsilon \{e^{q+1}\})-K.
		$ %\]
		Let $\mathcal{F}=\{F_1,\ldots,F_T\}$ be the set of \rev{$K$-maximal} facets of $\D_\epsilon$. For each $i\in\{1,\ldots,T\}$, let
		$%\[
		\{ ((w^{i1})^\T,\alpha_{i1})^\T, \ldots, ((w^{iJ_i})^\T,\alpha_{iJ_i})^\T\}
		$ %\]
		be the set of extreme directions of $F_i$ and define
		$%\[
		f^i_{\min} \coloneqq \min_{\lambda\in\Delta^{J_i-1}} \|\sum_{j=1}^{J_i}\lambda_j w^{ij}\|_\ast.
		$ %\]
		Then, $\P_\epsilon\coloneqq\P_{\D_\epsilon}$ is an inner approximation of $\P$ and $\P_\epsilon+B(0,\tilde{\epsilon})\supseteq\P\supseteq \P_\epsilon$, where $\tilde{\epsilon}=\epsilon/\min\{f^1_{\min},\ldots,f^T_{\min}\}$.
	\end{proposition}

	\begin{proposition}\label{prop:6}
		If the algorithm terminates, then it returns a finite weak $\tilde{\epsilon}$-solution $\bar{\mathcal{X}}$ to \eqref{P}, where $\tilde{\epsilon}$ is either as in \Cref{prop:eps_tilde_1} or as in \Cref{prop:eps_tilde_2}.
	\end{proposition}
	
	\begin{proof}
		\rev{Note that every element of $\bar{\X}$ is of the form $x^w$ which is an optimal solution to \eqref{WS(w)} for some $w \in C^+\setminus \{0\} $; by \Cref{prop:WS}, $x^w$ is a weak minimizer of \eqref{P}. To prove the statement, we need to show that $\conv f(\bar{\X})+C+B(0,\tilde{\epsilon}) \supseteq \P$ holds. Let us define
			$
			\bar{\P}\coloneqq \conv f(\bar{\X})+C,$
			%\qquad 
			$\bar{\D}\coloneqq\D_{\bar{\P}}.$ From \Cref{lem:D_inclusions}, we have $\bar{\D} \subseteq \D_\epsilon = \cone(\conv\xi(\bar{\W})+\epsilon \{e^{q+1}\})-K$. Then, \Cref{prop:7} implies $\bar{\P}=\P_{\bar{\D}} \supseteq\P_{\D_\epsilon}$. %i.e., $\bar{\P} \supseteq \P_\epsilon$ 
			By \Cref{prop:5} and \Cref{prop:eps_tilde_1} (or \Cref{prop:eps_tilde_2}), we have $\P_{\D_\epsilon}+B(0,\tilde{\epsilon}) \supseteq \P$. As $\bar{\P} \supseteq \P_{\D_\epsilon}$, we get $\bar{\P} +B(0,\tilde{\epsilon}) =\conv f(\bar{\X})+C+B(0,\tilde{\epsilon}) \supseteq \P$, i.e., $\bar{\X}$ is a finite $\tilde{\epsilon} $-solution to \eqref{P}.
		}
	\end{proof}

	\section{\rev{Relationships to similar approaches from literature}}\label{sect:relation}	
	\rev{In this section, we compare our approach with the approaches in two closely related works.}
	
	\subsection{\rev{Connection to the geometric dual problem by Heyde \cite{heyde_2013}}}
	%	we consider the geometric dual problem developed by Heyde \cite{heyde_2013} and the geometric dual algorithm proposed in \cite{ulus_2014}, in comparison to those proposed in this paper.}
	\rev{In \cite{heyde_2013}, a geometric dual image, which corresponds to the upper image $\mathcal{P}$ of problem \eqref{P}, is constructed in $\R^q$ as follows. For a fixed $c\in \Int C$, a matrix $E\in\R^{q\times(q-1)}$ is taken such that $T= [E \ \  c] \in \R^{q\times q}$ is orthogonal. Then, the dual image is defined as \[\D^H \coloneqq \{(t^\T, s)^\T \in \R^{q-1}\times \R \mid c^*(t)\in C^+, s \leq \inf_{x\in \X}(c^*(t))^\T f(x) \},\] where $c^*:\R^{q-1}\to \R^q$ is given by $c^\ast(t) \coloneqq T^{-\T} (t^\T, 1)^\T.$ Throughout, we denote the identity matrix in $\R^{n\times n}$ by $I_n \in \R^{n\times n}$. Observe from $T^\T T^{-\T} = I_q$ that $c^\T T^{-\T} = e_q$ and $E^\T T^{-\T} = [I_{q-1}\: 0] \in \R^{(q-1)\times q}$ hold true. }
	
	\rev{We will show that $\D^H$ and the dual image $\D$ constructed in \Cref{sec:problems} are related. For this purpose, we define $\hat{T}\in \R^{(q+1)\times(q+1)}, S \in \R^{(q+1)\times(q+1)}$ and $P \in \R^{q\times(q+1)}$ as \begin{equation}\label{eqq:T_hat_S_P} 
		\hat{T} \coloneqq \begin{bmatrix}
		T & 0\\
		0 & 1 
		\end{bmatrix}, \: S \coloneqq \begin{bmatrix}
		I_{q-1} & 0 & 0\\
		0 & 0 & 1 \\
		0 & 1 & 0 \end{bmatrix}, \: P \coloneqq \begin{bmatrix} I_q & 0 \end{bmatrix}, \end{equation} where $0$ is the matrix or vector of zeros with respective sizes. The following lemma will be useful for proving the next proposition.}%observations will be used to prove the subsequent results: For $w\in \R^q, \alpha\in \R$ we have \[S\hat{T}^\T(w^\T \: \alpha)^\T = (w^\T E, \alpha, w^\T c)^\T %= \begin{bmatrix}
	%E^\T w\\s\\k^\T w
	%\end{bmatrix} 
	%	\text{~~and~~ }\: PS\hat{T}^\T(w^\T, \: \alpha)^\T = (w^\T E, \alpha)^\T.%\begin{bmatrix} E^\T w\\s \end{bmatrix} 	\] 
	%
	\rev{\begin{lemma}\label{lem:DHvsD}
			Let $H_y := \{ (t^\T,s)^\T \in \R^{q-1}\times \R\mid c^\ast(t)\in C^+, s \leq (c^*(t))^\T y\}$, where $y\in \R^q$ is fixed.  %, where $c^\ast(t)=T^\T(t^\T,1)^{-\T}$. 
			Then, 
			\begin{align*}\bar{H}_y:=\{(w^\T, \alpha)^\T \in C^+ \times \R \mid \alpha \leq w^\T y \} = \hat{T}^{-\T}S \left[\cone ( H_y \times \{1\}) \right]. \end{align*}	
		\end{lemma}
		\begin{proof}
			Note that %$\hat{T}^{-\T} = \begin{bmatrix} T^{-\T} & 0\\ 0 & 1 \end{bmatrix}. $ 
			for any $(t^\T, s)^\T \in H_y$, we have $\hat{T}^{-\T}S (t^\T, s, 1)^\T = (c^\ast(t)^\T, s)^\T$, $c^\ast(t)\in C^+$ and $s \leq (c^*(t))^\T y$. Hence, $\hat{T}^{-\T}S[H_y \times \{1\}] \subseteq \bar{H}_y$ holds. As $\bar{H}_y$ is a cone, we obtain $$\hat{T}^{-\T}S[\cone(H_y \times \{1\})] = \cone\hat{T}^{-\T}S[H_y \times \{1\}] \subseteq \bar{H}_y.$$
			For the reverse inclusion, let $(w^\T, \alpha)^\T\in \bar{H}_y$. Consider $t = \frac{E^\T w}{c^\T w}, s = \frac{\alpha}{c^\T w}$. Then, $$c^\ast(t) = \frac{1}{k^\T w}T^{-\T} (w^\T E, w^\T c)^\T 
			%\begin{bmatrix} E^{\T}w \\ k^\T w \end{bmatrix} 
			= \frac{1}{c^\T w}T^{-\T}T^\T w = \frac{w}{c^\T w} \in C^+.$$ Moreover, $s \leq (c^*(t))^\T y$ holds. Hence, $(t^\T, s)^\T \in H_y$. Similar to the previous case, we have $$\hat{T}^{-\T}S (t^\T, s, 1)^\T = (c^\ast(t)^\T, s)^\T 
			%\begin{bmatrix} c^\ast(t) \\s \end{bmatrix}
			= \frac{1}{c^\T w} (w^\T, \alpha)^\T \in \hat{T}^{-\T}S[\D^H\times\{1\}],$$ which implies $(w^\T, \alpha)^\T \in \cone \hat{T}^{-\T}S[\D^H\times\{1\}].$
		\end{proof}
		The next proposition shows that $\D$ and $\D^H$ can be recovered from each other.
		\begin{proposition}\label{prop:DHvsD} The following relations hold true for $\D$ and $\D^H$:
			%	\begin{align} \label{eq:DHequals}
			\[(a) \: \D^H = P\left[ S\hat{T}^\T [\D] \cap \{z \in\R^{q+1} \mid z_{q+1} = 1 \} \right],\: (b) \: 
			\D = \hat{T}^{-\T}S \left[\cone (\D^H \times \{1\}) \right].\]
			%	\end{align}
		\end{proposition}
		\begin{proof}
			To see (a), let $(t^\T,s)^\T \in P[ S\hat{T}^\T [\D] \cap \{z \in\R^{q+1} \mid z_{q+1} = 1 \} ]$ be arbitrary. There exists $(w^\T,\alpha)^\T\in\D$ such that $(t^\T, s)^\T = PS\hat{T}^\T (w^\T, \alpha)^\T = (w^\T E, \alpha)^\T$ together with $c^\T w=1$. % as $(S\hat{T}^\T (w^\T,\: \alpha)^\T)_{q+1} = k^\T w$. 
			In particular, $t=E^\T w, s = \alpha$. We have $w^\T T = (
			w^\T E, w^\T c) = (t^\T, 1)$ and $c^\ast (t) = T^{-\T}( t^\T, 1)^\T = T^{-\T} T^\T w = w \in C^+$ as $(w^\T, \alpha)^\T \in \D$. Moreover, $s \leq \inf_{x\in \X}(c^*(t))^\T f(x)$ holds. Hence, $(t^\T,s)^\T \in \D^H$. For the reverse inclusion, let $(t^\T,s)^\T \in \D^H$. Then, $(c^\ast(t)^\T, s)^\T \in \D$ and $ S\hat{T}^\T(c^\ast(t)^\T, s)^\T = ((c^\ast(t))^\T E , s, c^\T c^\ast(t))^\T$ $= ( t^\T, s, 1)^\T.$ Hence, $(t^\T,s)^\T \in P[ S\hat{T}^\T [\D] \cap \{z \in\R^{q+1} \mid z_{q+1} = 1 \}]$.\\
			To see (b), let $H_y,\bar{H}_y$ be defined as in \Cref{lem:DHvsD}. Noting that $\D^H = \bigcap_{x \in \X} H_{f(x)}$ and $\D = \bigcap_{x \in \X} \bar{H}_{f(x)}$, \Cref{lem:DHvsD} implies
			\begin{align*}
			\hat{T}^{-\T}S \left[\cone ( \D^H \times \{1\}) \right] &= 		\hat{T}^{-\T}S \bigg[\cone \big( \bigcap_{x \in \X} (H_{f(x)} \times \{1\}) \big) \bigg] \\&=  \bigcap_{x \in \X} \hat{T}^{-\T}S \left[\cone  \left(H_{f(x)} \times \{1\}\right) \right]= \bigcap_{x \in \X}\bar{H}_{f(x)} = \D.
			\end{align*}
			%		To see (b), first note that %$\hat{T}^{-\T} = \begin{bmatrix} T^{-\T} & 0\\ 0 & 1 \end{bmatrix}. $ 
			%		for any $(t^\T, s)^\T \in \D^H$ we have $\hat{T}^{-\T}S (t^\T, s, 1)^\T = (c^\ast(t)^\T, s)^\T$, where $c^\ast(t) = T^{-\T}(t^\T, 1)^\T \in C^+$ and $s \leq \inf_{x\in \X}(c^*(t))^\T f(x)$. Then, it holds true that $\hat{T}^{-\T}S[\D^H \times \{1\}] \subseteq \D$. As $\D$ is a cone, we obtain $$\hat{T}^{-\T}S[\cone(\D^H \times \{1\})] = \cone\hat{T}^{-\T}S[\D^H \times \{1\}] \subseteq \D.$$
			%		For the reverse inclusion, let $(w^\T, \alpha)^\T\in \D$. Consider $t = \frac{E^\T w}{c^\T w}, s = \frac{\alpha}{c^\T w}$. Then, $$c^\ast(t) = \frac{1}{k^\T w}T^{-\T} (w^\T E, w^\T c)^\T 
			%		%\begin{bmatrix} E^{\T}w \\ k^\T w \end{bmatrix} 
			%		= \frac{1}{c^\T w}T^{-\T}T^\T w = \frac{w}{c^\T w} \in C^+.$$ Moreover, $s \leq \inf_{x\in \X}(c^*(t))^\T f(x)$ holds. Hence, $(t^\T, s)^\T \in \D^H$. Similar to the previous case, we have $$\hat{T}^{-\T}S (t^\T, s, 1)^\T = (c^\ast(t)^\T, s)^\T 
			%		%\begin{bmatrix} c^\ast(t) \\s \end{bmatrix}
			%		= \frac{1}{c^\T w} (w^\T, \alpha)^\T \in \hat{T}^{-\T}S[\D^H\times\{1\}],$$ which implies $(w^\T, \alpha)^\T \in \cone \hat{T}^{-\T}S[\D^H\times\{1\}].$
		\end{proof}
		%\Fir{\textbf{To see \eqref{eq:DATUequals}}, observe that for any $(t^\T, s)^\T \in \D^H$ we have $\hat{T}S (t^\T, s, 1)^\T = \begin{bmatrix}		Et+k\\s	\end{bmatrix}$, $c^\ast(t) = T^{-\T}(t^\T, 1)^\T \in C^+$ and $s \leq \inf_{x\in \X}(c^*(t))^\T f(x)$. We need to show $Et+k = T(t^\T, 1)^\T \in C^+$ and ... \rev{bunu tamamlayamadim. $\hat{T}$ yerine $\hat{T}^{-\T}$ alinca oluyor, onu yazdim asagiya. }} 
		%
		% \begin{align} \label{eq:DATUequals}
		%	\D^{ATU} &= \hat{T}S \left[\cl\cone (\D^H \times \{1\}) \right]  \\  
		%	&= \hat{T}S \left[\cone (\D^H \times \{1\}) \cup \left(\cone \{(0,\ldots,0,-1)^\T \} \times \{0\}\right) \right] \\ \notag
		%	&= \hat{T}S \left[\cone (\D^H \times \{1\}) \right] \cup \cone \{(0,\ldots,0,0,-1)^\T \}, \notag
		%\end{align}
		%where the second equality follows from [34, Theorem 8.2], since $\rec \D^H = \cone \{ (0, ..., 0, -1)^\T \}$ due to the fact that $C$ is solid.
		\begin{remark}
			For $\D^H$ and $\D$, one can check that the following results hold:
			(i) A subset $F \subseteq \D^H$ is an exposed face of $\D^H$ if and only if $\hat{T}^{-\T}S [\cl \cone (F \times \{1\})]$ is an exposed face of $\D$.
			(ii) For every exposed face $G$ of $\D$, not lying in the hyperplane $\hat{T}^{-\T}S [\R^q\times \{0\}]$, there is an exposed face $F$ of $\D^H$ such that $\hat{T}^{-\T}S [\cl\cone (F\times\{-1\})]=G$.
			The proof relies on a generalization of \cite[Lemma 3.1]{luc-duality} for the exposed faces of (possibly nonpolyhedral) closed convex sets.
	\end{remark}}
	
	\subsection{\rev{Connection to the geometric dual algorithm by L\"{o}hne, Rudloff and Ulus \cite{ulus_2014}}}
	
	\rev{The geometric dual algorithm proposed in \cite{ulus_2014} is based on the geometric duality results from \cite{heyde_2013}. We will now show that by considering a special norm $\norm{\cdot}$ used in \Cref{alg2}, we recover the dual algorithm from \cite{ulus_2014}. In particular, assume that the dual norm satisfies $\norm{w}_* = c^\T w$ for all $w \in C^+$, \Cag{where $c\in\Int C$ is fixed}.\footnote{For instance, $\norm{w}_* \coloneqq \inf \{ c^\T (v + u) \mid v, u \in C^+, v - u = w \}$ would satisfy this. %For any $w\in \R^q$, $A_w:=\{(u,v)\in C^+\times C^+ \mid u-v = w\}$ is a nonempty closed set. Indeed, let $v \in \Int C^+$, $\inf_{c \in C\cap B(0,1)} v^\T c = \alpha > 0$ and $\inf_{c \in C\cap B(0,1)} w^\T c = \beta \in \R$. Then for $\lambda \geq \frac{-\beta}{\alpha}$, $u=w + \lambda v \in C^+$. It is not difficult to see that $\norm{w}_\ast \geq 0$ for all $w$ and it is zero only if $w = 0$. Also, easy to see that $\norm{\lambda w}_\ast = \abs{\lambda}\norm{w}_\ast$, by a change of variables. Remains to see the triangular inequality.
	}} 
	
%	\rev{The geometric dual algorithm proposed in \cite{ulus_2014} is based on the geometric duality results from \cite{heyde_2013}. We will now show that by considering a special norm $\norm{\cdot}$ used in \Cref{alg2}, we recover the dual algorithm from \cite{ulus_2014}. In particular, assume that the dual norm satisfies $\norm{w}_* = c^\T w$ for all $w \in C^+$.\footnote{For instance, $\norm{w}_* \coloneqq \inf \{ c^\T (v + u) \mid v, u \in C^+, v - u = w \}$ would satisfy this. %For any $w\in \R^q$, $A_w:=\{(u,v)\in C^+\times C^+ \mid u-v = w\}$ is a nonempty closed set. Indeed, let $v \in \Int C^+$, $\inf_{c \in C\cap B(0,1)} v^\T c = \alpha > 0$ and $\inf_{c \in C\cap B(0,1)} w^\T c = \beta \in \R$. Then for $\lambda \geq \frac{-\beta}{\alpha}$, $u=w + \lambda v \in C^+$. It is not difficult to see that $\norm{w}_\ast \geq 0$ for all $w$ and it is zero only if $w = 0$. Also, easy to see that $\norm{\lambda w}_\ast = \abs{\lambda}\norm{w}_\ast$, by a change of variables. Remains to see the triangular inequality.}} 
	
	\rev{We denote by $\D^{H}_k, k\geq 0$ the outer approximations of $\D^H$ that is obtained in the $k^{th}$ iteration of \cite[Algorithm 2]{ulus_2014}. Let $w^0, x^0$ be as defined in Section \ref{subsect:dualalg}. Noting that $\norm{w^0}_\ast = c^\T w^0 =1$ and $x^0\in \argmin_{x\in\X}(w^0)^\T f(x)$, in \cite[Algorithm 2]{ulus_2014}, the initial outer approximation for $\D^H$ is $$\D_0^H = \{(t^\T,s)^\T \in \R^{q-1}\times \R \mid c^\ast(t)\in C^+,  s \leq c^\ast(t)^\T f({x}^0)\} = H_{f(x^0)}.$$ By \Cref{lem:DHvsD}, we obtain $\D_0 = \hat{T}^{-\T}S \left[\cone (\D^H_0 \times \{1\}) \right].$ Now, assume for some $k \geq 1$, $\D_k = \hat{T}^{-\T}S \left[\cone (\D^H_k \times \{1\}) \right]$ holds. Then, for any vertex $(t^\T,s)^\T$ of $\D_k^H$, $\hat{T}^{-\T}S(t^\T,s,1)^\T$ is an extreme direction of $\D_k$.} %\textcolor{gray}{(This can be seen by applying the definitions of vertex and extreme direction as follows: As $v = (t^\T,s)^\T$ be a vertex of $\D^H_k$. If $u_1,u_2 \in \D^H_k$ and $\lambda\in(0,1)$ satisfying $\lambda u_1 +(1-\lambda)u_2 = v$, then we have $u_1=u_2 = v$. Consider $\hat{T}^{-\T}S(v^\T,1)^\T \in \D_k$. Let $a,b\in \D_k, \lambda_a,\lambda_b >0 $ such that $\lambda_a a + \lambda_b b = \hat{T}^{-\T}S(v^\T,1)^\T \in \D_k$. We know that there exist $u,w\in\D^H_k, \alpha_u,\alpha_w >0$ such that $a = \alpha_u \hat{T}^{-\T}S(u^\T,1)^\T$ and $b = \alpha_w \hat{T}^{-\T}S(w^\T,1)^\T$. Then, $$\lambda_a \alpha_u \hat{T}^{-\T}S(u^\T,1)^\T + \lambda_b \alpha_w \hat{T}^{-\T}S(w^\T,1)^\T = \hat{T}^{-\T}S(v^\T,1)^\T$$ implies $$\lambda_a \alpha_u u+ \lambda_b \alpha_w w = v, \quad \lambda_a \alpha_u + \lambda_b \alpha_w  = 1,$$ hence $u = w =v$ and $a,b$ yield the same direction as $\hat{T}^{-\T}S(v^\T,1)^\T $.)} 
	\rev{Moreover, for an extreme direction $(w^\T, \alpha)$ of $\D_k$, there exists a vertex $v = (t^\T,s)^\T$ of $\D^H_k$ such that $\frac{(w^\T,\alpha)^\T}{\norm{w}_\ast} = \frac{\hat{T}^{-\T}S(t^\T,s,1)^\T}{\norm{P\hat{T}^{-\T}S(t^\T,s,1)^\T}_\ast}$.} 
	%\textcolor{gray}{(Wlog, let $(w^\T,\alpha) = \hat{T}^{-\T}S(t^\T,s,1)^\T$ be an extreme direction of $\D_k$ for $v =(t^\T,s)^\T \in \D^H_k$. If for any $a,b\in\D_k, \lambda_a,\lambda_b >0$ we have $\lambda_a a+ \lambda_b b = \hat{T}^{-\T}S(t^\T,s,1)^\T$, then  $\frac{\hat{T}^{-\T}S(t^\T,s,1)^\T}{\norm{P\hat{T}^{-\T}S(t^\T,s,1)^\T}} = \frac{a}{\norm{Pa}} = \frac{b}{\norm{Pb}}$. We claim that $v$ is a vertex of $\D^H_k$. Let $u,w\in \D^H_k, \lambda\in(0,1)$ and $\lambda u + (1-\lambda)w = v$. Then, $\lambda(u^\T,1)^\T + (1-\lambda)(w^\T,1)^\T= (v^\T,1)^\T $. This implies $$\lambda \hat{T}^{-\T}S(u^\T,1)^\T + (1-\lambda) \hat{T}^{-\T}S(w^\T,1)^\T = \hat{T}^{-\T}S(v^\T,1)^\T.$$ Since $\hat{T}^{-\T}S(u^\T,1)^\T,\hat{T}^{-\T}S(w^\T,1)^\T \in \D^H_k$ we have $\frac{\hat{T}^{-\T}S(u^\T,1)^\T}{\norm{P\hat{T}^{-\T}S(u^\T,1)^\T}} = \frac{\hat{T}^{-\T}S(w^\T,1)^\T}{\norm{P\hat{T}^{-\T}S(w^\T,1)^\T}} = \frac{\hat{T}^{-\T}S(v^\T,1)^\T}{\norm{P\hat{T}^{-\T}S(v^\T,1)^\T}}$. By checking the last components we obtain $$\norm{P\hat{T}^{-\T}S(u^\T,1)^\T} = \norm{P\hat{T}^{-\T}S(w^\T,1)^\T} = \norm{P\hat{T}^{-\T}S(v^\T,1)^\T},$$ which also implies $u=w=v$.)}
	
	\rev{In \Cref{alg2}, we consider an extreme direction of $\D_k$, say $(w^\T, \alpha)^\T$ with $\norm{w}_\ast = 1$ and solve \eqref{WS(w)}. Assume that $v=(t^\T,s)^\T\in\D^H_k$ is a vertex  satisfying $(w^\T,\alpha)^\T = \frac{\hat{T}^{-\T}S(v^\T,1)^\T}{\norm{P\hat{T}^{-\T}S(v^\T,1)^\T}_\ast}$. Note that with the specified norm, we have $\norm{P\hat{T}^{-\T}S(t^\T,s,1)^\T}_\ast = c^\T P\hat{T}^{-\T}S(t^\T,s,1)^\T = %c^\T P (c^\ast(t)^\T,s) = c^\T c^\ast(t) = 
		c^\T T^{-\T}(t^\T,1)^\T = e_q^\T (t^\T,1)^\T = 1$. Then, we have $\hat{T}^{-\T}S(t^\T,s,1)^\T = (w^\T,\alpha)^\T $, hence
		%	\begin{align*}
		$(t^\T,1,s)^\T = S(t^\T,s,1)^\T = \hat{T}^{\T}(w^\T,\alpha)^\T = (w^\T T, \alpha)^\T.$
		%	\end{align*}
		In particular, $(t^\T,1)^\T = T^\T w$ holds. This implies $c^\ast(t) = T^{-\T}(t^\T,1)^\T =  w$ and $s = \alpha$. In particular, both algorithms solve the same weighted sum scalarization problem \eqref{WS(w)},
		%\textcolor{gray}{(since in \cite{ulus_2014}, $w(\cdot)$ is defined as $w(v) := c^\ast(v_1,\ldots,v_{q-1})$)}, 
		hence may find the same minimizer $x^w$. Moreover, both update the current outer approximation if $\alpha-p^w > \epsilon$ holds. In that case, the updated outer approximations are $\D^H_{k+1} = \D^H_k \cap H_{f(x^w)}$%\{(t^\T,s)^\T \in \R^q \mid c^\ast(t)^\T f(x^w)\geq s\}$ 
		and $\D_{k+1} = \D_k \cap \bar{H}_{f(x^w)}$, %\{(w^\T,\alpha)\in \R^{q+1} \mid w^\T f(x^w) \geq \alpha\}$
		respectively. Then, by \Cref{lem:DHvsD}, we have %$\D_{k+1} = \hat{T}^{-\T}S \left[\cone (\D^H_{k+1} \times \{1\}) \right] $ since}  	 
		\begin{align*}
		\hat{T}^{-\T}S \left[\cone (\D^H_{k+1} \times \{1\}) \right] &= \hat{T}^{-\T}S \left[\cone \right( (\D^H_{k}\cap H_{f(x^w)}) \times \{1\}\left) \right] 
		\\&= \hat{T}^{-\T}S \left[\cone \right( (\D^H_{k} \times \{1\}) \cap(H_{f(x^w)} \times \{1\}) \left) \right] \\
		&= \hat{T}^{-\T}S \left[\cone \left( \D^H_{k} \times \{1\}\right) \cap \cone \left(H_{f(x^w)} \times \{1\}\right) \right] \\
		&= \D_k \cap \bar{H}_{f(x^w)}  = \D_{k+1}.
		\end{align*}
	}
	\rev{We have shown by induction that $\D_k = \hat{T}^{-\T}S \left[ \cone ( \D^H_k \times \{1\} ) \right] $ hold for all $k\geq 0$ as long as the algorithms consider the extreme directions, respectively the corresponding vertices, in the same order. In that case, $\bar{\X}$ returned by both algorithms would also be the same. We conclude that by considering a norm satisfying $\norm{w}_\ast = c^\T w$, the proposed geometric dual algorithm recovers \cite[Algorithm 2]{ulus_2014}.}	
	
	\section{Numerical examples}\label{sec:numres}
	
	We implement the primal and dual algorithms given in \Cref{sec:algorithms} using MATLAB.\footnote{We run \rev{the experiments in \Cref{subsect:results}} on a computer with i5-8265U CPU and 8 GB RAM\rev{, and the ones in \Cref{subsect:compar_lit} on a computer with i5-10310U CPU and 16 GB RAM.}} For problems \eqref{PS(v)} and \eqref{WS(w)}, we employ CVX, a framework for specifying and solving convex programs \cite{gb08,cvx}. For solving vertex enumeration problems within the algorithms, we use \emph{bensolve tools} \cite{ciripoi_2018,weissing_2016,bensolve_notes}.

	\subsection{Proximity measures}\label{sec:prox}
	
	%\Fir{(Bu paragrafın yerine aşağıya bunu farklı şekilde yazdım. Buradaki discussion'ı olduğu gibi bıraktım şimdilik. Buna gerek duymazsak tamamen kaldırabiliriz.)}
	%We measure the performance of the primal and dual algorithms using two indicators as we explain next. The \emph{primal error indicator} gives the distance between $\P$ and the furthest outer approximation point to $\P$. \Cag{(Bunun kullanıldığı/tanımlandığı yere referans versek iyi olur.)} \Fir{(Böyle bir tanım literatürde var mı, emin değilim. Bu outer approximation ile upper image arasındaki Hausdorff distance'a karşılık gelmiyor mu? Direkt o şekilde söylesek?) } This indicator can be used to compare the worst approximation points \Fir{(bununla ne demek istediğimizi anlamadım.)} returned by the algorithms. On the other hand, the \emph{hypervolume indicator} gives us the difference between the hypervolumes covered by the inner and outer approximation points. It is a useful indicator to assess whether the points returned by the algorithms are well-spread. \Cag{(Bunun için de bir referans verebiliriz.)} \Fir{(Bunu da literatürdeki haliyle kullanmadığımız için kendi tanımımız olduğunu vurgulamakta fayda var bence.)}
	
	We measure the performance of the primal and dual algorithms using two indicators. The distance between $\P$ and the furthest point from the outer approximation $\P_o$ returned by the algorithm to $\P$ is referred to as the \emph{primal error indicator}. If the recession cone of $\P_o$ is $C$, then the primal error indicator is nothing but the Hausdorff distance between $\P_o$ and $\P$, see \cite[Lemma 5.3]{umer_2022}. The primal error indicator of an algorithm is calculated by solving \eqref{PS(v)} for the set $\V_o$ of the vertices of $\P_o$. For each $v\in \V_o$, we find an optimal solution $z^v$ and the corresponding optimal value $\norm{z^v}$ by solving \eqref{PS(v)}. Then, the primal error indicator is calculated by
	%\begin{equation*}
	%\begin{aligned}
	$\text{PE}\coloneqq \max\{\norm{z^v}\mid v\in \V_o\}$.
	%\end{aligned}
	%\end{equation*}
	
	%\Cag{(Aşağıdaki paragrafta anlatımı biraz değiştirdim. Simay, tam olarak doğru anlatmış oluyor muyuz, kontrol eder misin?)} \Fir{(Ben de önce sadece başına biraz açıklama eklemeye çalışırken sonra epey değiştirdim :))\\}
	
	Note that the norm used in the definition of PE is the same (primal) norm that is used in the algorithm. %\textbf{(Böyle mi Simay?)} 
	\rev{Hence, this indicator depends on the choice of the norm. Motivated by this, we define a \emph{hypervolume indicator for CVOPs} which is free of norm-biasedness.} %(Bu ikisi sanırım öncü makaleler 99'da isim vermeden ilk kez kullanılmış, bir tane 2003 makalesi var aşağıda cite edilmiş, orada indicator değil hypervolume hesabı var.. diğerlerini çıkarabiliriz isterseniz, çok fazla cite etmeye gerek olmayabilir.)} \cite{auger_2012} \Fir{\sout{\cite{cardona_2014},\cite{lima_2017}} (bunlar spesifik uygulamalar sanki).} 
	Recall that, for a multiobjective optimization problem, the hypervolume of a set $S\subseteq \R^q$ of points with respect to a reference point $r\in\R^q$ %satisfying $\{r\}-\R^p_+ \supseteq S$ 
	is computed as $\Lambda(\{y\in\R^q \mid s \leq_{\R^q_+} y\leq_{\R^q_+} r,\ s\in S \})$,
	%\[
	%\Lambda\bigg(\bigcup_{s\in S, s\leq_{\R^q_+} r}\{y\in\R^q \mid s \leq_{\R^q_+} y\leq_{\R^q_+} r \}\bigg),
	%\]
	where $\Lambda$ is the Lebesgue measure on $\R^q$ % \Cag{and $\leq=\leq_{\R^q_+}$}. 
	\cite{zitzler2003performance}.
	%(gerek görmezsek çıkarabiliriz)}. %\textbf{(Bu tanımın introduce edildiği ya da kullanıldığı daha eski bir makaleyi cite etmek lazım,} bu şuradan aldım şimdilik: https://arxiv.org/pdf/2005.00515.pdf)} 
	Different from this hypervolume measure, which can be used for convex as well as nonconvex multiobjective optimization problems, we use the fact that \eqref{P} is a convex problem and the underlying order relation is induced by cone $C$. To that end, we define the hypervolume of a set $S\subseteq \R^q$ \emph{with respect to a bounding polytope} $\mathcal{Q}\subseteq \R^q$ as 
	%\begin{equation*}
	$\text{HV}(S,\mathcal{Q})\coloneqq \Lambda((\conv S+C)\cap\mathcal{Q})$.
	%\end{equation*}
	Let $\P_o, \P_i$ be, respectively, the outer and inner approximations of $\P$ returned by an algorithm and $\V_o,\V_i$ be the set of vertices of them. Then, we compute the \emph{hypervolume indicator} by
	%\begin{equation*}
	$\text{HV}\coloneqq(\frac{\text{HV}(\V_o,\mathcal{Q})-\text{HV}(\V_i,\mathcal{Q})}{\text{HV}(\V_o,\mathcal{Q})})\times 100$, 
	%\end{equation*} 
	where $\mathcal{Q}\subseteq \R^q$ is a polytope satisfying $\V_o\cup\V_i \subseteq \mathcal{Q}$. Suppose that the problem is solved by finitely many algorithms and let $A\subseteq\R^q$ be the set of all vertices of the outer and inner approximations returned by \emph{all} algorithms. For our computational tests, in order to have a fair comparison, we fix the polytope $\mathcal{Q}$ such that $A \subseteq \mathcal{Q}$. For this, we set
	%\begin{equation*}
	%\begin{aligned}
	$\mathcal{Q}\coloneqq \bigcap_{j=1}^J\{y\in\R^q \mid (w^j)^{\T}y\leq \max_{a\in A} (w^j)^{\T}a \}$,
	%\end{aligned}
	%\end{equation*} 
	where $w^1, \ldots, w^J $ are the generating vectors of $C^+ $. Note that a smaller hypervolume indicator is more desirable in terms of an algorithm's performance since it means less of a difference between the inner and outer sets.

	\subsection{Computational results \rev{on Algorithms \ref{alg1} and \ref{alg2}}} \label{subsect:results}
	
	We assess the performance of the primal and dual algorithms by solving randomly generated problem instances. A problem structure that is simple and versatile is required for scaling purposes, in terms of both the decision space and the image space. To this end, we work with a linear objective function and a quadratic constraint as described in \Cref{example:3}.
	
	\begin{example}\label{example:3}
		Consider the problem
		\begin{equation*}
		\begin{aligned}
		\textnormal{minimize } f(x)=A^{\T}x \textnormal{ with respect to } \leq_{\R_+^q}
		\textnormal{ subject to } x^{\T}Px-1\leq0,
		\end{aligned}
		\end{equation*}
		where $P \in \R^{n\times n}$ is a symmetric positive definite matrix, and $A\in\R_+^{n\times q}$. In the computational experiments, we generate $A$ and $P$ as instances of random matrices.\footnote{More precisely, we generate $A$ as the instance of a random matrix with independent entries having the uniform distribution over $[0,50]$. To construct $P$, we first create a matrix $U\in\R^{n\times n}$ following the same procedure as for $A$. We assume that $U$ has at least one nonzero entry, which occurs with probability one. Then, $S\coloneqq (U+U^\T)/2$ is a symmetric matrix and we calculate its diagonalization as $S=Q\bar{D}Q^{\T}$, where $\bar{D}\in\R^{n\times n}$ is the diagonal matrix of the eigenvalues of $S$, and $Q$ is the orthogonal matrix of the corresponding eigenvectors. Denoting by $D\in\R^{n\times n}$ the diagonal matrix whose entries are the absolute values of the entries of $\bar{D}$, we let $P\coloneqq QDQ^{\T}$, which is guaranteed to be symmetric and positive definite.} %\Cag{(Makaleye doğrudan kodu koyamayız, pseudocode olarak koysak bile açıklamak gerekir. Tek başına bu açıklama da yeterli olabilir.)} \Fir{(Tek başına açıklama makale için bence de yeterli olur.)}
	\end{example}
	%\Cag{\sout{The matrices $P$ and $A$ were randomly generated, using the parameters $q$ and $n$. The MATLAB codes that were used to create $P$ and $A$ can be found below.}}
	
	%\begin{lstlisting}[style=Matlab-editor]
	% 	 G = rand(n)*50; 
	%  G = (G+G')/2;
	%  [Q,D] = eig(G);
	%  P = Q*abs(D)*Q';
	%  A = rand(n,q)*50;
	%\end{lstlisting}
	
	\subsubsection{Results for multiobjective problem instances} \label{subsubsect:tests1}
	
	%The rest of this section provides the results obtained by solving randomly generated problem instances.
	This subsection provides the results obtained by solving randomly generated instances of \Cref{example:3}. The computational results are presented in \Cref{table:2}, which shows the stopping criteria (Stop), %\Sim{\sout{cardinality of the returned solution ($\abs{\bar{\X}} $),}} 
	number of optimization problems solved (Opt), time taken to solve optimization problems ($\textnormal{T}_{\textnormal{opt}}$), number of vertex enumeration problems solved (En), time taken to solve vertex enumeration problems ($\textnormal{T}_{\textnormal{en}}$), and total runtime of the algorithm (T) in terms of seconds. As a measure of efficiency for the algorithms, we also provide the average runtime spent per optimization problem ($\textnormal{T}_{\textnormal{opt}}  / \textnormal{Opt} $) and total runtime per weak minimizer in $\bar{\X} $ ($\textnormal{T}/\abs{\bar{\X}} $).
	
	The performance indicators for the tests are primal error (PE) and hypervolume (HV). For the computations of these, we take the following: 
	%\Fir{\textbf{(Bunları düzenledim, tekrar kontrol edelim.)}} 
	For \Cref{alg1}, $\P_o=\P_{\bar{k}}$, where $\bar{k}$ is the iteration number at which the algorithm terminates; for \Cref{alg2}, $\P_o = \bigcap_{w\in \bar{\W}}\H(w,p^w)$ where $\bar{\W}$ is the solution to \eqref{D} that \Cref{alg2} returns. For both algorithms, $\P_i = \conv f(\bar{\X}) +C$, where $\bar{\X}$ is the solution to \eqref{P} that the corresponding algorithm returns.
	
	We take $q=3$. The results for various choices of the dimension of the decision space ($n$) can be found in \Cref{table:2}. For each value of $n$, $20$ random problem instances are generated using the structure of \Cref{example:3} and solved by both Algorithms \ref{alg1} and \ref{alg2} twice with different stopping criteria. %\Fir{(Alg 2 için eps belirlemiyoruz son satırda diye bu şekilde yazdım)}. %\Cag{(Burada neden twice olduğu havada kalıyordu.)} 
	The averages of the results and performance indicators of the $20$ instances are presented in the table in corresponding cells. %\Cag{\sout{In the table, there are $4$ rows of results for each different $n$ value.}}
	
	%\Cag{(Buraya kadar okudum.)}

	\begin{table}[h!]
		\centering
		\resizebox{0.7\textwidth}{!}{% <------ Don't forget this %
			\begin{tabular}{| c| c| c| c | c | c | c | c | c | c | c |c| }
				\hline
				$n$ & Alg & Stop  & Opt & $ \textnormal{T}_{\textnormal{opt}} $  & $\textnormal{T}_{\textnormal{opt}}  / \textnormal{Opt} $ & En  & $ \textnormal{T}_{\textnormal{en}} $ & T & $\textnormal{T}/ \abs{\bar{\X}} $ & PE & HV\\\hline
				
				\multirow{4}{*}{10}
				&\ref{alg1}	&$\epsilon_1$ = 0.5000	&52.35	&19.92	&0.3801	&5.05	&0.09	&46.41	&0.91	&0.4209	&3.5805\\
				&\ref{alg2} &$\epsilon_2$ = 0.2887	&86.25	&29.03	&0.3363	&4.40	&0.20	&71.36	&70.84	&0.1106	&1.1973\\
				&\ref{alg1} &$\epsilon_3$ = 0.1106	&232.80	&85.32	&0.3655	&6.90	&0.16	&200.13	&0.92	&0.1052	&1.2846\\
				&\ref{alg2}           &T = 46.41	&60.70	&20.32	&0.3336	&4.05	&0.24	&48.43	&0.80	&0.2244	&1.8735\\ \hline
				\multirow{4}{*}{15}
				&\ref{alg1}	&$\epsilon_1$ = 0.5000	&70.85	&26.98	&0.3777	&5.60	&0.09	&63.54	&0.93	&0.4694	&2.6458\\
				&\ref{alg2} &$\epsilon_2$ = 0.2887	&105.15	&34.87	&0.3311	&4.70	&0.26	&84.58	&0.80	&0.1033	&0.7146\\
				&\ref{alg1} &$\epsilon_3$ = 0.1033	&295.30	&108.74	&0.3677	&7.45	&0.21	&251.56	&0.91	&0.0991	&0.5690\\
				&\ref{alg2}            &T = 63.54	&81.75	&27.01	&0.3307	&4.40	&0.36	&65.51	&0.81	&0.2022	&1.0703\\ \hline
				\multirow{4}{*}{20}
				&\ref{alg1}	&$\epsilon_1$ = 0.5000	&65.45	&24.62	&0.3797	&5.50	&0.09	&57.12	&0.92	&0.4569	&3.4845\\
				&\ref{alg2} &$\epsilon_2$ = 0.2887	&101.95	&33.66	&0.3302	&4.65	&0.24	&81.31	&0.80	&0.1052	&1.0783\\
				&\ref{alg1} &$\epsilon_3$ = 0.1052	&284.25	&104.18	&0.3672	&7.35	&0.19	&238.73	&0.91	&0.1020	&0.7881\\
				&\ref{alg2}            &T = 57.12	&74.40	&24.63	&0.3324	&4.30	&0.29	&59.06	&0.80	&0.1996	&1.5307\\ \hline
				\multirow{4}{*}{25}
				&\ref{alg1}	&$\epsilon_1$ = 0.5000	&85.95	&32.63	&0.3786	&5.95	&0.12	&82.51	&0.98	&0.4650	&2.3063\\
				&\ref{alg2} &$\epsilon_2$ = 0.2887	&139.25	&46.13	&0.3306	&4.95	&0.34	&112.65	&0.81	&0.1071	&0.6154\\
				&\ref{alg1} &$\epsilon_3$ = 0.1071	&368.60	&137.91	&0.3736	&7.70	&0.26	&340.28	&1.01	&0.1034	&0.5063\\
				&\ref{alg2}            &T = 82.51	&106.10	&35.02	&0.3297	&4.80	&0.56	&84.55	&0.80	&0.1878	&0.8167\\ \hline
				\multirow{4}{*}{30}
				&\ref{alg1}	&$\epsilon_1$ = 0.5000	&95.70	&36.59	&0.3818	&6.00	&0.12	&91.27	&0.99	&0.4690	&2.2715\\
				&\ref{alg2} &$\epsilon_2$ = 0.2887	&150.70	&51.13	&0.3386	&5.15	&0.43	&131.15	&0.87	&0.1066	&0.6110\\
				&\ref{alg1} &$\epsilon_3$ = 0.1066	&452.70	&170.22	&0.3751	&8.05	&0.32	&419.02	&1.02	&0.1046	&0.5420\\
				&\ref{alg2}            &T = 91.27	&109.00	&36.89	&0.3381	&4.60	&0.48	&93.41	&0.86	&0.2272	&0.8947 \\ 
				\hline
			\end{tabular}
		}
		\caption{Results of randomly generated problems.}
		\label{table:2}
	\end{table}
	
	In \Cref{table:2}, the first two rows for each value of $n$ show the results of the primal and dual algorithms when the given $\epsilon_i$ value is fed to the algorithms as stopping criterion. It can be seen that the $\epsilon_i$ values that are used in the algorithms are different. We run \Cref{alg1} and obtain a weak $\epsilon_1$-solution to problem \eqref{P}. When working with \Cref{alg2}, to obtain also a weak $\epsilon_1$-solution to problem \eqref{P}, we take
	%\[
	$\epsilon_2 = \epsilon_1\min_{\lambda\in\Delta^{J-1}}\|\sum_{j=1}^J\lambda_jw^j\|_\ast
	% \epsilon_1\underset{\lambda\in\Delta^{J-1}}{\min} \norm{\sum_{j=1}^J\lambda_jw^j}_\ast
	$
	%\]
	based on \Cref{prop:eps_tilde_1}. As a result, we take $\epsilon_1=0.5$ for \Cref{alg1} and $\epsilon_2=0.2887$ for \Cref{alg2} as stopping criteria. We observe from the first two rows for each $n$ in \Cref{table:2} that \Cref{alg1} stops in shorter runtime; however, \Cref{alg2} returns smaller primal error and hypervolume indicators.
	
	For further comparison of the algorithms, we also run each algorithm with different stopping criteria. In the third row, we aim to observe the runtime it takes for \Cref{alg1} to reach similar primal error that \Cref{alg2} returns. Therefore, the PE value in the second row is fed to \Cref{alg1} as stopping criterion. %\Sim{\sout{More specifically, for each instance, we round up the PE value of the second row up to three decimals in order to set $\epsilon_3$. The values in the tables show the averages over 50 instances. Because of the rounding, $\epsilon_3$ is slightly higher than the PE value from the second row.}} %\Cag{(Tablodaki mismatch'leri nasıl açıklayacağımızı düşünmemiz gerekir.)} 
	Finally, in the fourth row, we aim to observe the primal error and hypervolume indicators of \Cref{alg2} when it is terminated after a similar runtime as of \Cref{alg1} from the first row. Therefore, T from the first row is fed to \Cref{alg2} as stopping criterion.%\Sim{\sout{Indeed, we round down the T value from the first row.}}
	\footnote{Note that the actual termination time is slightly higher than the predetermined time limit as we check the time only at the beginning of the loop in the implementation and it takes a couple of more seconds to terminate the algorithm afterwards.}

	From \Cref{table:2}, one can observe that Algorithms \ref{alg1} and \ref{alg2} in the first two rows are not comparable. Indeed, while \Cref{alg1} in the first row has shorter runtime for each $n$ value, it also gives larger hypervolume results, which indicates worse performance in comparison with \Cref{alg2} in the second row. The main reason is that, when \Cref{alg2} is run with a stopping criterion that guarantees obtaining a weak $\epsilon_1$-solution to \eqref{P}, the solution that it returns has much higher proximity, for instance, compare $\epsilon_2=0.2887$ and PE$=0.1106$ in the second row of \Cref{table:2}.\footnote{After running \Cref{alg2} with the predetermined $\epsilon_2$ value (to get a weak $\epsilon_1$-solution for \eqref{P}), in order to compute the realized PE for \eqref{P}, we solve \eqref{PS(v)} problems as explained in \Cref{sec:prox}. Instead, in order to compute an upper bound for the PE value (which would be a better bound than $\epsilon_1$), one can also use \Cref{prop:eps_tilde_2}.} On the other hand, when we set $\epsilon_3$ in a way that Algorithms \ref{alg1} and \ref{alg2} return similar PE values, \Cref{alg1} may return slightly better HV results but it requires higher runtime.
	
	In order to have further insights, we analyze %\sout{are going to make comparisons between}} 
	the results of \Cref{alg1} from the first row and \Cref{alg2} from the fourth row in more detail, since they have similar runtimes. %, see \Cref{fig:7} (first two figures). %\sout{A plot of the primal error and hypervolume indicators of these rows can be seen in Figure \ref{fig:4}.} 
	Similarly, as they yield similar PE values, we analyze the results of \Cref{alg1} from the third row and \Cref{alg2} from the second row. %in \Cref{fig:7} (last two figures).
	%In Figure \ref{fig:7}, we plot the PE and HV values from rows one and four of \Cref{table:2} for each $n$, together with the underlying CPU times. We observe that, when the two algorithms work for similar runtimes, \Cref{alg2} returns smaller primal error. Moreover, the algorithms perform similarly  in terms of the hypervolume indicator, especially for higher values of $n$. When we consider the total runtimes, we see approximately $2$ seconds of difference between the algorithms. Since the runtime is used as a stopping criterion for \Cref{alg2} in the fourth row, it takes approximately $2$ more seconds to terminate the algorithm, as explained before. %\Sim{\sout{Although \mbox{\Cref{alg2}} seems to have advantage with running for $2$ more seconds, it is also stopped without considering completing the iteration and this might cause disadvantage in terms of primal error and hypervolume indicators.}} %\Fir{(stopping condition ile ilgili bu avantaj/dezavantaj durumu pek açık değil. Açıklamaya gerek var mı ondan da emin olamadım.)}
	In \Cref{fig:7}, the plots of the PE and HV values for rows one and four of \Cref{table:2} are shown (first two figures). %\Sim{\sout{One can see that the primal error difference between the algorithms is larger in comparison with \mbox{\Cref{fig:4}}.}} 
	One can see that \Cref{alg2} has consistently better performance for each value of $n$ in terms of both primal error and hypervolume indicator. Moreover, the plots of total runtimes and HV values corresponding to rows two and three of \Cref{table:2} can be seen together with the PE values on the right vertical axis (last two figures). We observe that the difference between the primal error indicators of both algorithms are very similar to each other, although, \Cref{alg1} has smaller primal error as expected. With very similar primal error indicators, we observe that \Cref{alg2} has around half of the runtime of \Cref{alg1}. In line with the primal error results, \Cref{alg1} has a better HV value than \Cref{alg2}.
	
	\begin{figure}[h!]
		\centering
		\includegraphics[scale=0.12]{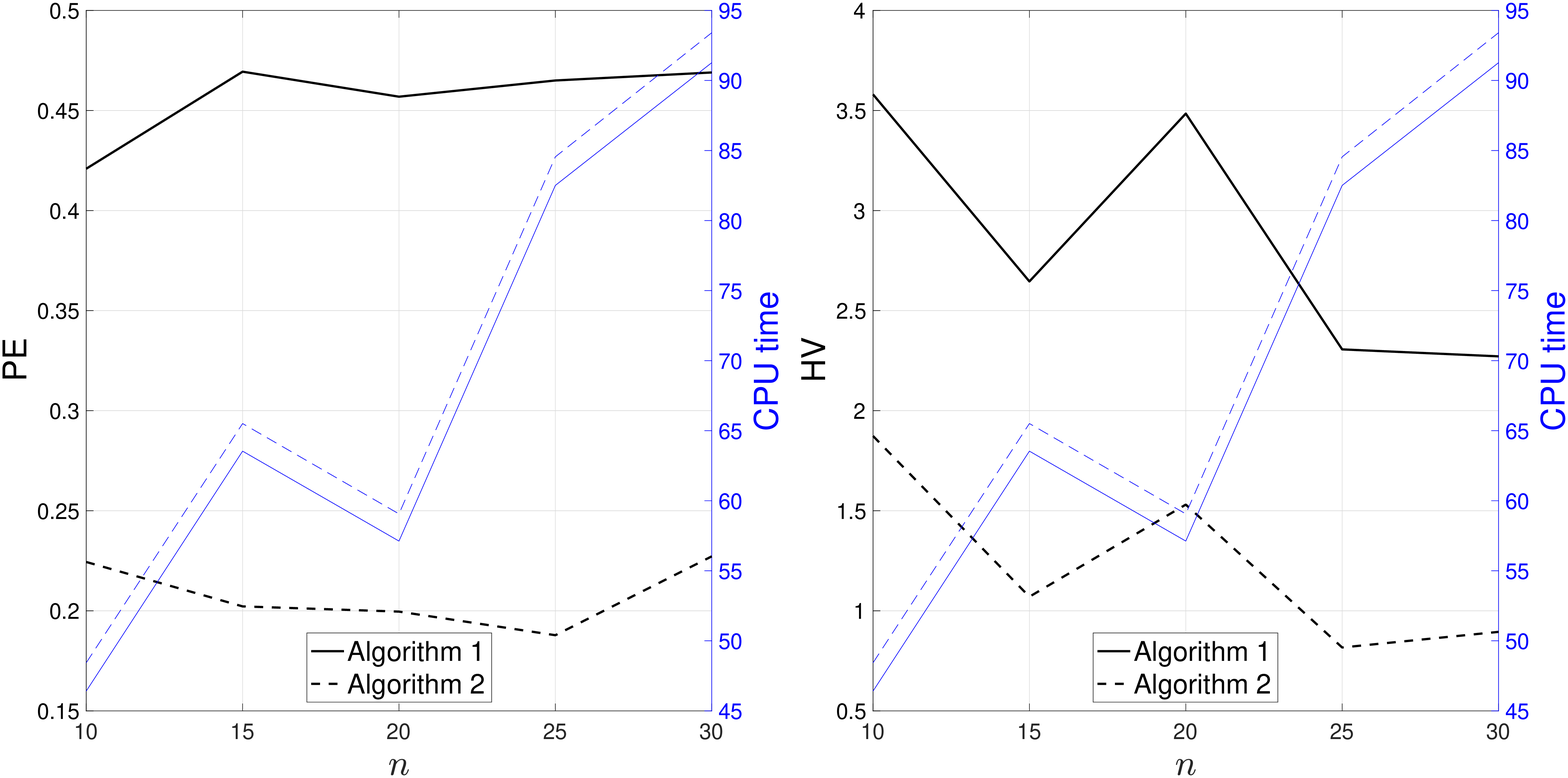}
		\includegraphics[scale=0.12]{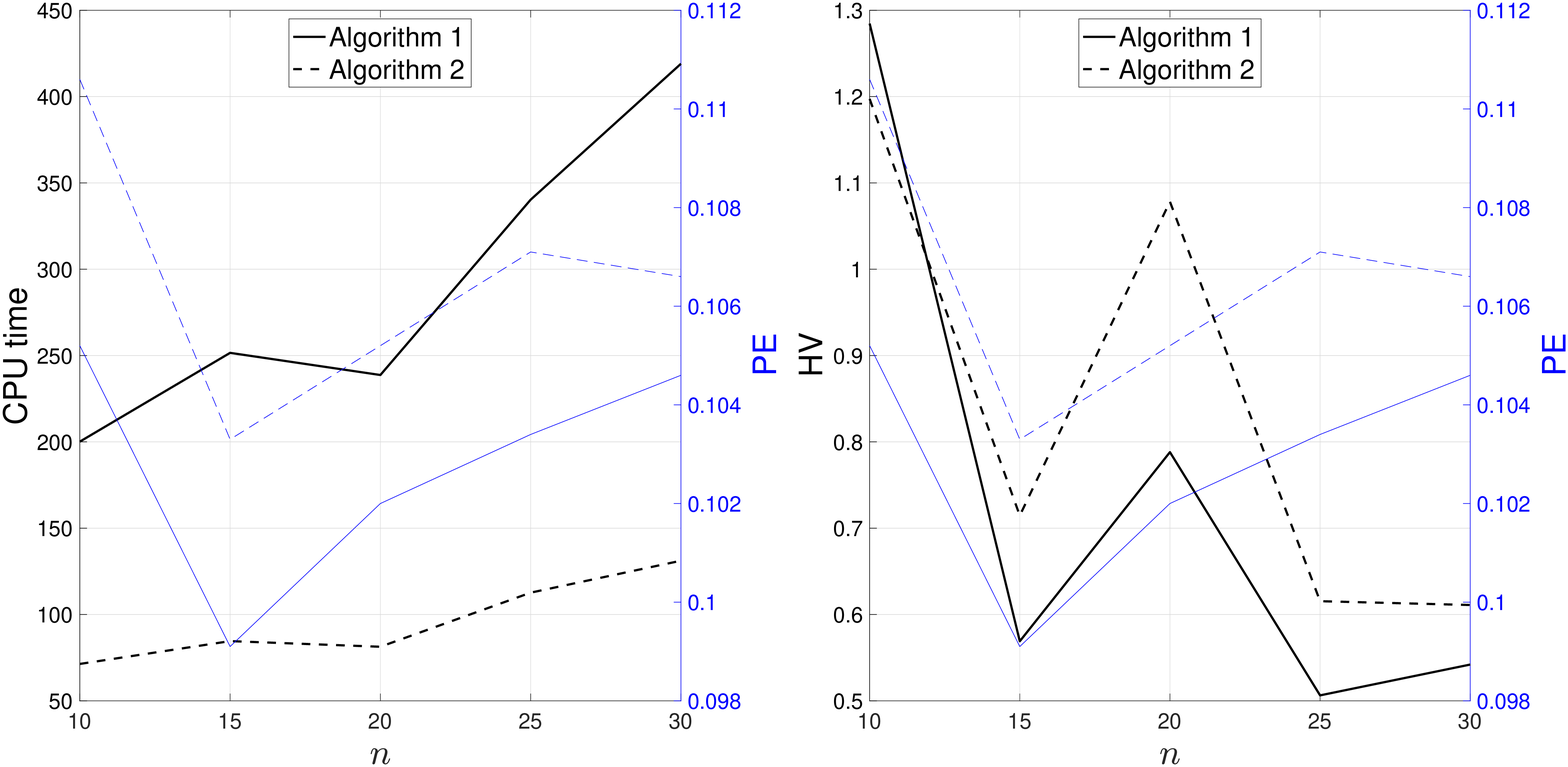}
		\caption{\Cref{example:3}: Average PE (first) and HV (second) values of $20$ random instances under nearly equal runtime (rows one and four of \Cref{table:2});
			average CPU time (third) and HV (fourth) values of these instances when the dual algorithm uses the PE of the primal algorithm as approximation error (rows two and three of \Cref{table:2}).}
		\label{fig:7}
	\end{figure}
	\subsubsection{Results for different ordering cones and norms} \label{subsubsect:tests_2}
	
	In order to test the performance of the algorithms on problems with different ordering cones and with different norms employed in \eqref{PS(v)} and \eqref{D-PS(v)}, we design some additional experiments. Although the original stopping criteria in Algorithms \ref{alg1} and \ref{alg2} are based on $\epsilon$-closeness, we implement an alternative stopping criterion for the rest of our analysis. In particular, motivated by our analysis in \Cref{subsubsect:tests1}, we set a predefined time limit for the algorithms. Note that since $\epsilon$-closeness depends on the choice of the norm used in the scalarization models, using runtime as the stopping criterion results in a fair comparison of the algorithms when considering different norms.
	
	%First, we consider the following example for $q \in\{2,3,4\}$.
	%\begin{example}\label{example:1}
	%	Consider the problem
	%	\begin{equation*}
	%		\begin{aligned}
	%			&\textnormal{minimize } f(x)= x \textnormal{ with respect to } \leq_{\R_+^q}
	%			\textnormal{ subject to } \norm{x-e}_2 \leq 1, \quad x\in \R^q_+, % (x_1-1)^2+(x_2-1)^2\leq1, \ x_1,x_2\geq0 .
	%			\end{aligned}
	%		\end{equation*}
	%	where $e\in \R^q$ is the vector of ones. 	
	%\end{example}
	%\Cref{example:1} from \Cref{sec:illustrate} for $q \in\{2,3,4\}$. 
	We denote the nonnegative cone by $C_1=\R^3_+$. The non-standard ordering cones that will be used throughout are\footnote{These cones are taken from \cite{umer_2022}.}:\\ %\Fir{In each case, we have $C_3 = C_2^+ $. \textbf{(Önceki yanlıştı bu doğru mu acaba q=4 için?)}} 
	%\Fir{For each we have $C_2 \subseteq \R^q_+ \subseteq C_3$. \textbf{($q=4$ için de doğru mu?)}}\Sim{q=4 icin $C_3 \subseteq C_1 \subseteq C_2 $ oluyor hocam.} \Fir{Bu da olmuyor galiba Simay, her ikisinin de generating vektörlerinde pozitif ve negatif bileşenler varmış. Bu ifadeyi kaldırıyorum o yüzden.}
	%\begin{itemize}\item
	%	 For $q=2$, we take 
	%	\begin{enumerate}[]
	%	\item 
	%	$C_{2} = \conv \cone \{(1,2)^\T,(2,1)^\T\}$, $C_{3}=\conv \cone\{(2,-1)^\T,(-1,2)^\T\}$.
	%	\end{enumerate}
	%\item 
	%	For $q=3$, 
	%We take
	%\begin{enumerate}[]
	%\item 
	$C_{2} = \conv \cone \{ (4,2,2)^\T, (2,4,2)^\T, (4,0,2)^\T, (1,0,2)^\T,(0,1,2)^\T,(0,4,2)^\T\}$,\\
	%\item 
	$C_{3}=\conv \cone\{(-1,-1,3)^\T,(2,2,-1)^\T,(1,0,0)^\T,(0,-1,2)^\T,(-1,0,2)^\T,(0,1,0)^\T\}$.
	%\end{enumerate}
	%\item 
	%	For $q=4$, we take 
	%	\begin{enumerate}[]
	%	\item 
	%\begin{align*}
	%& C_{2}=\conv \cone \left\{
	%		\begin{array}{ll}
	%		(3,4,-4,4)^\T,(1,-4,-2,0)^\T,(5,5,-3,5)^\T, (5,0,3,-4)^\T, \\
	%		(-1,4,3,5)^\T,(2,-5,3,4)^\T,(2,3,1,-1)^\T,(2,-3,2,-5)^\T
	%		\end{array}
	%		\right\},\\
	%		%\item 
	%&		C_{3}=\conv \cone \left\{
	%		\begin{array}{ll}
	%		(1,-1,0,0)^\T,(0,-1,1,-1)^\T,(1,0,0,0)^\T,(1,0,1,-1)^\T\\
	%		(1,-1,1,0)^\T,(1,0,1,1)^\T,(1,-1,1,1)^\T
	%		\end{array}
	%		\right\}.
	%		\end{align*}
	%	\end{enumerate}
	%\end{itemize}
	
	We consider the $20$ instances of \Cref{example:3} that were generated randomly for our analysis in \Cref{subsubsect:tests1}. We solve all instances under the runtime limit of $50$ seconds. We consider ordering cones $C_1,C_2,C_3$ and $\ell_p$ norms for $p\in \{1,2,\infty\}$. While solving some of the instances with different ordering cones or different norms, we encountered some errors in the solvers that we have employed in the algorithms. Hence, we consider a subset of instances which can be solved in all settings and we indicate the size of this subset in a separate column in \Cref{table:3} (Size).
	
	%\Cref{table:3} shows the results and performance indicators obtained by solving \Cref{example:1} for $q=2$, $q=3 $ and $q=4$. 
	%The problems are solved with respect to ordering cones $C_1$, $C_2 $ and $C_3$, where the norm in \eqref{PS(v)} is taken as the $\ell_2$ norm; moreover, they are solved with respect to the standard ordering cone $C_1$, where the norm in \eqref{PS(v)} is set as $\ell_1 $ and $\ell_\infty $. 
	%For different $q$ values, different runtime limits are implemented. For $q=2$, $q=3$, $q=4$, the problems are terminated after $25, 50, 75$ seconds, respectively.
	%for $q=3$, the problems are terminated after $50$ seconds and for $q=4$, the problems are terminated after $75$ seconds, to acommodate for the difference in problem sizes. 
	
	First, we fix the $\ell_2$ norm and solve the problem instances with respect to the ordering cones $C_1,C_2$ and $C_3$; second, we fix the ordering cone as $C_1$ and solve the problem instances where we take the norms in \eqref{PS(v)} as $\ell_1,\ell_2$ and $\ell_\infty$, see the left and the right columns of \Cref{table:3}, respectively. %(Tablo sonuçları için hiç yorum yoktu, ben de şimdilik bu şekilde bıraktım. Tabloları Appendix'e atıp burada sadece Figürleri verebiliriz diye düşünüyorum.) 
	In order to summarize the results, we plot the average PE and HV values obtained by the algorithms, see Figures \ref{fig:9} and \ref{fig:11}. From these figures, we observe that for all considered ordering cones and norms, \mbox{\Cref{alg2}} has better performance in terms of both PE and HV under time limit.

	\begin{table}[h]
		\centering
		\resizebox{\textwidth}{!}{% <------ Don't forget this %
			\begin{tabular}{|c|c| c| c | c | c | c | c | c | c | c |c||c| c| c | c | c | c | c | c | c | c |c| }
				\hline
				$n$ & Size & Cone  &Alg & Opt & $ \textnormal{T}_{\textnormal{opt}} $  & $\textnormal{T}_{\textnormal{opt}}  / \textnormal{Opt} $ & En  & $ \textnormal{T}_{\textnormal{en}} $  & $\textnormal{T}/ \abs{\bar{\X}} $ & PE & HV & Size & $p$  &Alg & Opt & $ \textnormal{T}_{\textnormal{opt}} $  & $\textnormal{T}_{\textnormal{opt}}  / \textnormal{Opt} $ & En  & $ \textnormal{T}_{\textnormal{en}} $  & $\textnormal{T}/ \abs{\bar{\X}} $ & PE & HV \\\hline
				
				\multirow{6}{*}{10} &\multirow{6}{*}{15}
				&\multirow{2}{*}{$C_{1}$} 
				&\ref{alg1}			&57.47	&21.60	&0.3763	&5.00	&0.09	&0.93	&0.4024	&1.9112 & \multirow{6}{*}{17}
				&\multirow{2}{*}{1} 
				&\ref{alg1} 		&53.71	&21.09	&0.3931	&5.00	&0.08	&1.00	&0.5665	&2.1873\\
				&&&\ref{alg2} 		&60.73	&20.42	&0.3363	&4.00	&0.16	&0.86	&0.2527	&1.0413 			&&&\ref{alg2} 		&57.35	&19.79	&0.3450	&4.00	&0.17	&0.91	&0.3473	&1.2154\\ \cline{3-12} \cline{14-23}
				&&\multirow{2}{*}{$C_{2}$} 
				&\ref{alg1}			&60.60	&22.65	&0.3743	&3.80	&0.07	&0.92	&0.8190	&1.5472 &&\multirow{2}{*}{2} 
				&\ref{alg1} 		&56.00	&21.14	&0.3777	&5.00	&0.07	&0.95	&0.3973	&2.0040\\
				&&&\ref{alg2} 		&60.73	&20.49	&0.3376	&3.80	&0.35	&0.86	&0.4035	&0.8266 &&&\ref{alg2} 		&57.88	&19.98	&0.3453	&4.00	&0.16	&0.90	&0.2574	&1.1113\\\cline{3-12} \cline{14-23}
				& &\multirow{2}{*}{$C_{3}$} 
				&\ref{alg1}			&61.20	&22.73	&0.3719	&3.93	&0.07	&0.91	&0.1337	&2.5455 &&\multirow{2}{*}{$\infty$} &\ref{alg1} 		&53.41	&21.23	&0.3976	&5.06	&0.07	&0.99	&0.2826	&1.7489\\
				&&&\ref{alg2} 		&60.20	&20.40	&0.3392	&3.93	&0.39	&0.87	&0.0802	&1.7527 &&&\ref{alg2} 		&58.18	&20.07	&0.3450	&4.00	&0.17	&0.90	&0.1634	&1.0354\\\hline		
				
				\multirow{6}{*}{15} & \multirow{6}{*}{12}
				&\multirow{2}{*}{$C_{1}$}
				&\ref{alg1} 		&54.92	&21.38	&0.3908	&5.00	&0.10	&0.97	&0.5544	&2.3338 & \multirow{6}{*}{17}
				&\multirow{2}{*}{1} 
				&\ref{alg1} 		&57.65	&22.04	&0.3831	&5.00	&0.08	&0.92	&0.7394	&2.1807\\
				&&&\ref{alg2} 		&60.17	&20.44	&0.3401	&4.00	&0.17	&0.87	&0.2843	&1.3300 	&&&\ref{alg2} 		&63.24	&20.72	&0.3278	&4.00	&0.16	&0.82	&0.3808	&0.8593\\\cline{3-12} \cline{14-23}
				&&\multirow{2}{*}{$C_{2}$}
				&\ref{alg1} 		&58.58	&22.44	&0.3839	&3.67	&0.07	&0.95	&1.1556	&2.2875 &&\multirow{2}{*}{2} 
				&\ref{alg1} 		&61.82	&22.47	&0.3637	&5.00	&0.07	&0.86	&0.6205	&2.2184\\
				&&&\ref{alg2} 		&60.25	&20.49	&0.3405	&3.92	&0.41	&0.87	&0.4582	&0.9734 &&&\ref{alg2} 		&63.88	&20.83	&0.3261	&4.00	&0.16	&0.82	&0.2966	&1.0616\\\cline{3-12} \cline{14-23}
				&&\multirow{2}{*}{$C_{3}$}
				&\ref{alg1} 		&60.33	&22.81	&0.3785	&4.00	&0.07	&0.92	&0.1754	&2.2426 &&\multirow{2}{*}{$\infty$} 
				&\ref{alg1} 		&60.94	&22.54	&0.3700	&5.18	&0.07	&0.87	&0.3284	&2.0562\\
				&&&\ref{alg2} 	 	&61.33	&20.72	&0.3380	&4.00	&0.42	&0.85	&0.0958	&1.6479&&&\ref{alg2} 		&64.06	&20.90	&0.3264	&4.00	&0.16	&0.81	&0.1807	&1.2320\\\hline
				
				\multirow{6}{*}{20} & \multirow{6}{*}{15}
				&\multirow{2}{*}{$C_{1}$}
				&\ref{alg1} 		&58.40	&22.36	&0.3837	&5.00	&0.09	&0.91	&0.4571	&2.8328 & \multirow{6}{*}{18}
				&\multirow{2}{*}{1} 
				&\ref{alg1} 		&61.17	&22.51	&0.3686	&5.00	&0.08	&0.87	&0.7294	&3.0080\\
				&&&\ref{alg2} 	 	&62.67	&21.24	&0.3390	&4.00	&0.16	&0.83	&0.2631	&1.5824&&&\ref{alg2} 		&63.83	&20.90	&0.3274	&4.00	&0.15	&0.82	&0.3499	&1.2083\\\cline{3-12} \cline{14-23}
				&&\multirow{2}{*}{$C_{2}$}
				&\ref{alg1} 		&60.73	&23.10	&0.3809	&3.80	&0.07	&0.91	&0.9700	&2.0552 &&\multirow{2}{*}{2} 
				&\ref{alg1} 		&60.11	&22.12	&0.3684	&5.00	&0.07	&0.88	&0.4588	&2.6591\\
				&&&\ref{alg2} 	 	&61.80	&20.88	&0.3382	&4.00	&0.41	&0.85	&0.4777	&1.3532 &&&\ref{alg2} 		&64.00	&20.98	&0.3279	&4.00	&0.16	&0.81	&0.2757	&1.4626\\\cline{3-12} \cline{14-23}
				&&\multirow{2}{*}{$C_{3}$}
				&\ref{alg1} 		&62.33	&23.39	&0.3756	&4.00	&0.07	&0.89	&0.1390	&1.9080 &&\multirow{2}{*}{$\infty$} 
				&\ref{alg1} 		&59.28	&22.21	&0.3753	&5.00	&0.07	&0.90	&0.2694	&2.6331\\
				&&&\ref{alg2} 	 	&62.27	&21.06	&0.3384	&4.00	&0.41	&0.84	&0.0953	&1.3192 &&&\ref{alg2} 		&64.00	&21.10	&0.3300	&4.00	&0.15	&0.81	&0.1729	&1.5567\\\hline 
				
				\multirow{6}{*}{25} & \multirow{6}{*}{12}
				&\multirow{2}{*}{$C_{1}$}
				&\ref{alg1} 	 	&54.58	&21.19	&0.3895	&5.00	&0.08	&0.98	&0.5988	&2.2838 & \multirow{6}{*}{14}
				&\multirow{2}{*}{1} 
				&\ref{alg1} 		&58.79	&21.74	&0.3703	&5.00	&0.07	&0.90	&0.9830	&2.3330\\
				&&&\ref{alg2} 	 	&61.08	&20.54	&0.3366	&4.00	&0.16	&0.85	&0.3339	&1.3099 &&&\ref{alg2} 		&64.64	&21.23	&0.3286	&4.00	&0.16	&0.81	&0.4813	&0.8697\\\cline{3-12} \cline{14-23}
				&&\multirow{2}{*}{$C_{2}$}
				&\ref{alg1} 		&56.67	&21.77	&0.3845	&3.58	&0.06	&0.98	&1.7656	&2.6061 &&\multirow{2}{*}{2} 
				&\ref{alg1} 		&57.93	&21.49	&0.3714	&5.00	&0.07	&0.91	&0.6971	&2.3461\\
				&&&\ref{alg2} 	 	&61.75	&20.79	&0.3369	&4.00	&0.41	&0.85	&0.5857	&1.4343 &&&\ref{alg2} 		&64.79	&21.25	&0.3282	&4.00	&0.15	&0.81	&0.3578	&1.2700\\\cline{3-12} \cline{14-23}
				&&\multirow{2}{*}{$C_{3}$}
				&\ref{alg1} 		&57.67	&21.89	&0.3799	&3.83	&0.07	&0.96	&0.2698	&2.4895 &&\multirow{2}{*}{$\infty$} 
				&\ref{alg1} 		&57.36	&21.60	&0.3767	&5.00	&0.06	&0.92	&0.4166	&2.6032\\
				&&&\ref{alg2} 	 	&62.17	&20.88	&0.3360	&4.00	&0.41	&0.84	&0.1346	&1.5884&&&\ref{alg2} 		&64.71	&21.26	&0.3286	&4.00	&0.16	&0.81	&0.2241	&1.4588\\\hline
				
				\multirow{6}{*}{30} & \multirow{6}{*}{11}
				&\multirow{2}{*}{$C_{1}$}
				&\ref{alg1} 		&55.45	&21.11	&0.3829	&5.00	&0.08	&0.98	&0.6844	&2.8334 & \multirow{6}{*}{16}
				&\multirow{2}{*}{1} 
				&\ref{alg1} 		&57.44	&21.51	&0.3751	&5.00	&0.08	&0.93	&1.0316	&2.1772\\
				&&&\ref{alg2} 	 	&59.55	&20.08	&0.3375	&4.00	&0.16	&0.88	&0.4049	&1.5810 &&&\ref{alg2} 		&59.75	&20.08	&0.3363	&4.00	&0.16	&0.88	&0.5999	&0.8015\\\cline{3-12} \cline{14-23}
				&&\multirow{2}{*}{$C_{2}$}
				&\ref{alg1} 		&58.91	&22.26	&0.3783	&3.73	&0.07	&0.95	&1.9284	&2.1519 &&\multirow{2}{*}{2} 
				&\ref{alg1} 		&55.63	&20.98	&0.3777	&5.00	&0.07	&0.96	&0.7395	&2.5195\\
				&&&\ref{alg2} 	 	&59.36	&20.03	&0.3378	&3.91	&0.43	&0.88	&0.7094	&1.1776 &&&\ref{alg2} 		&60.00	&20.11	&0.3355	&4.00	&0.16	&0.87	&0.4591	&1.3388\\\cline{3-12} \cline{14-23}
				&&\multirow{2}{*}{$C_{3}$}
				&\ref{alg1} 		&60.09	&22.50	&0.3747	&3.82	&0.07	&0.93	&0.2666	&2.6265 &&\multirow{2}{*}{$\infty$} 
				&\ref{alg1} 		&56.00	&21.35	&0.3815	&5.00	&0.07	&0.95	&0.4387	&2.4787\\
				&&&\ref{alg2} 	 	&59.73	&20.18	&0.3382	&3.91	&0.39	&0.87	&0.1357	&1.6780 &&&\ref{alg2} 		&60.69	&20.29	&0.3345	&4.00	&0.16	&0.86	&0.2922	&1.4639\\\hline			
			\end{tabular}
		}
		\caption{Results for randomly generated instances of \Cref{example:3} with different ordering cones (left) and with different norms used in \eqref{PS(v)} (right), when the algorithms are run for T=$50$ seconds.}
		\label{table:3}
	\end{table}

	%In Figure \ref{fig:9}, the graph of primal error results of Algorithm \ref{alg1} and Algoritm \ref{alg2} are presented for ordering cones $C_1 $, $C_2$ and $C_3 $. From these graphs, we can observe that when the runtime constraint is applied to the algorithms, Algorithm \ref{alg2} has a better performance in terms of primal error. This means, the algorithm returns better and closer outer approximation points in comparison with Algorithm \ref{alg1}.
	\begin{figure}[h!]
		\centering
		\includegraphics[scale=0.12]{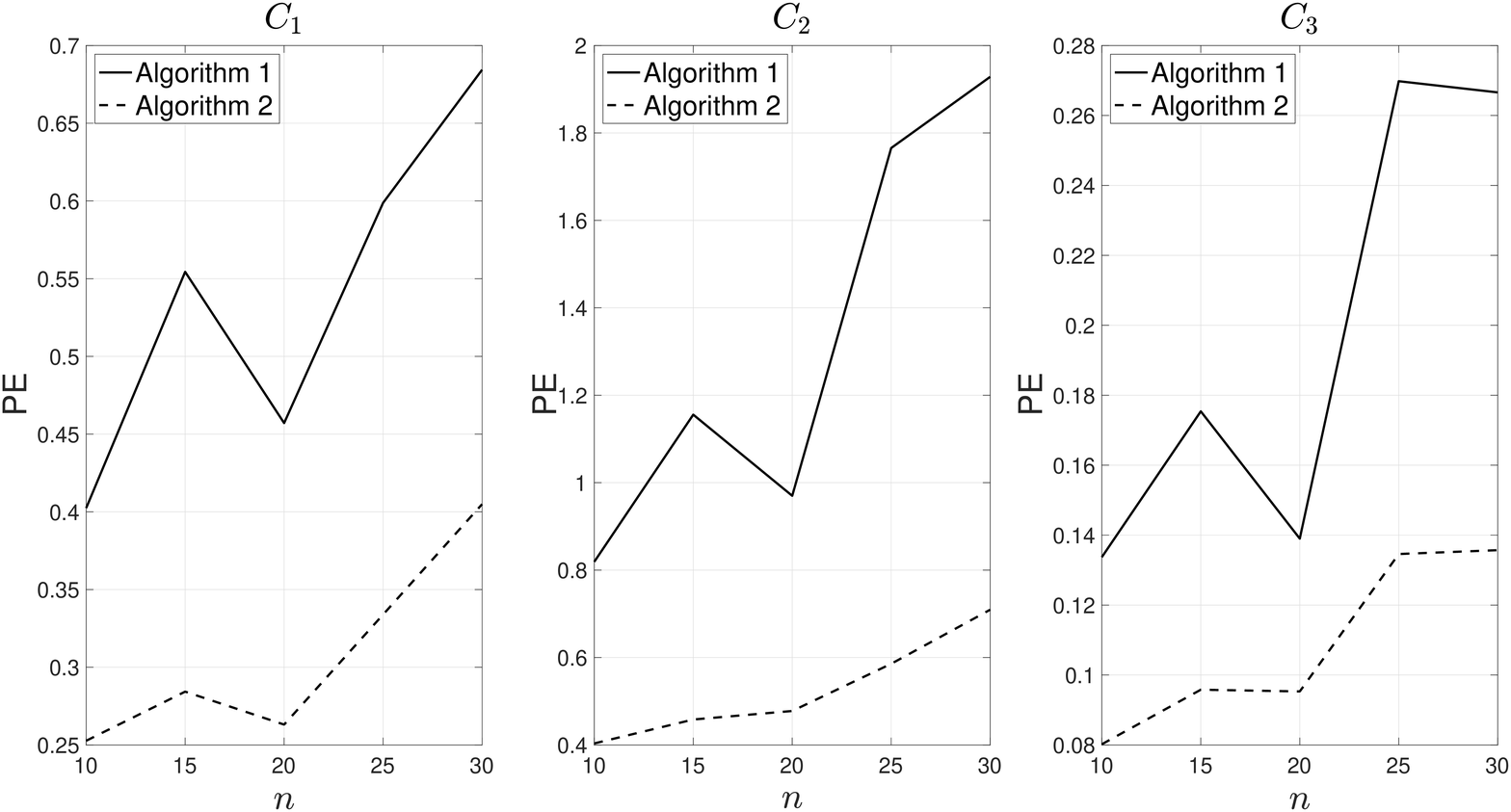}
		\includegraphics[scale=0.12]{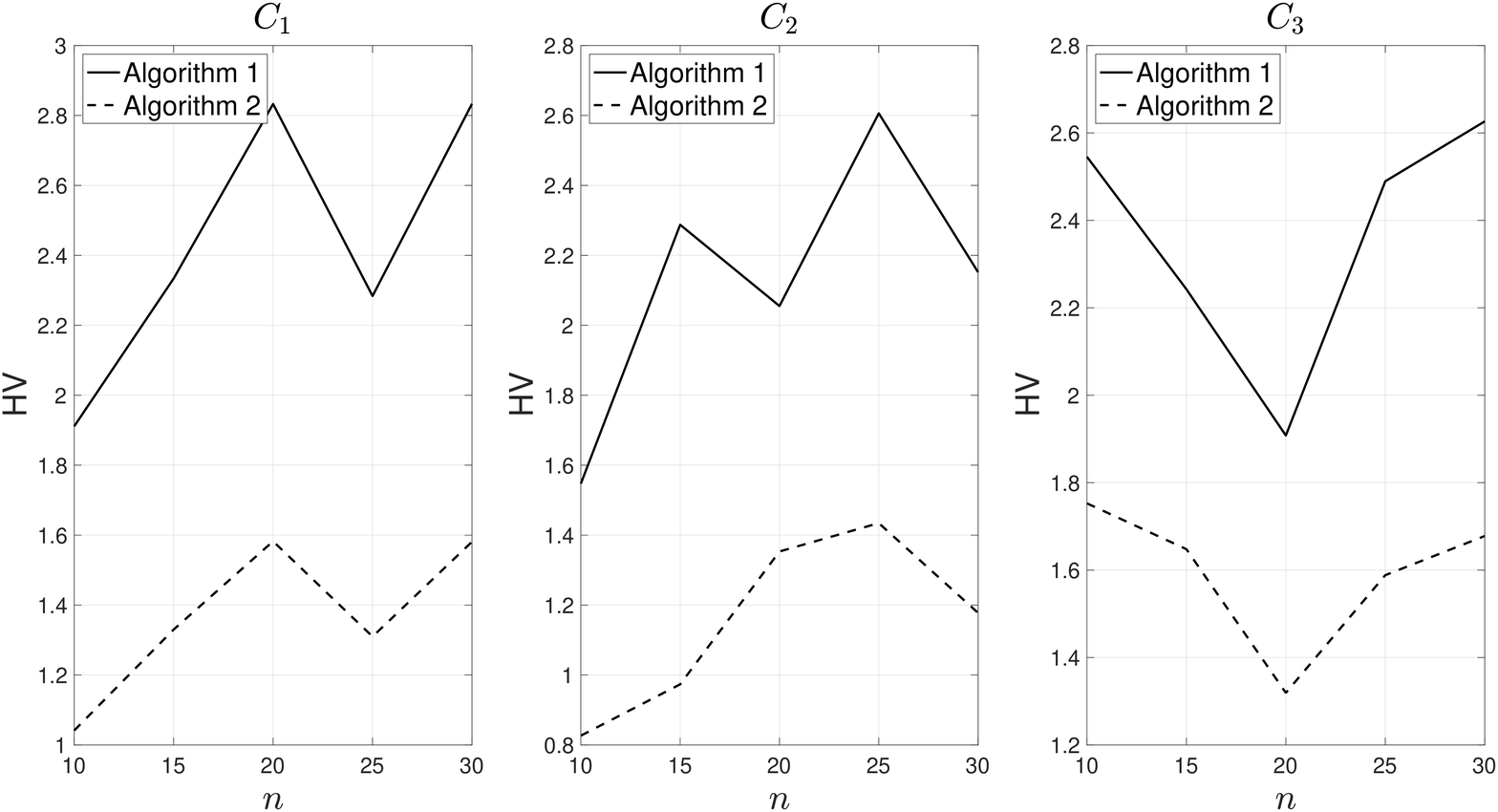}
		\caption{Average primal error (first \rev{group}) and HV (second \rev{group}) values of random instances of \Cref{example:3} for ordering cones $C_1 $ (left), $C_2 $ (middle) and $C_3 $ (right) when the algorithms are run under time limit of 50 seconds.}
		\label{fig:9}
	\end{figure}

	\begin{figure}[h!]
		\centering
		\includegraphics[scale=0.12]{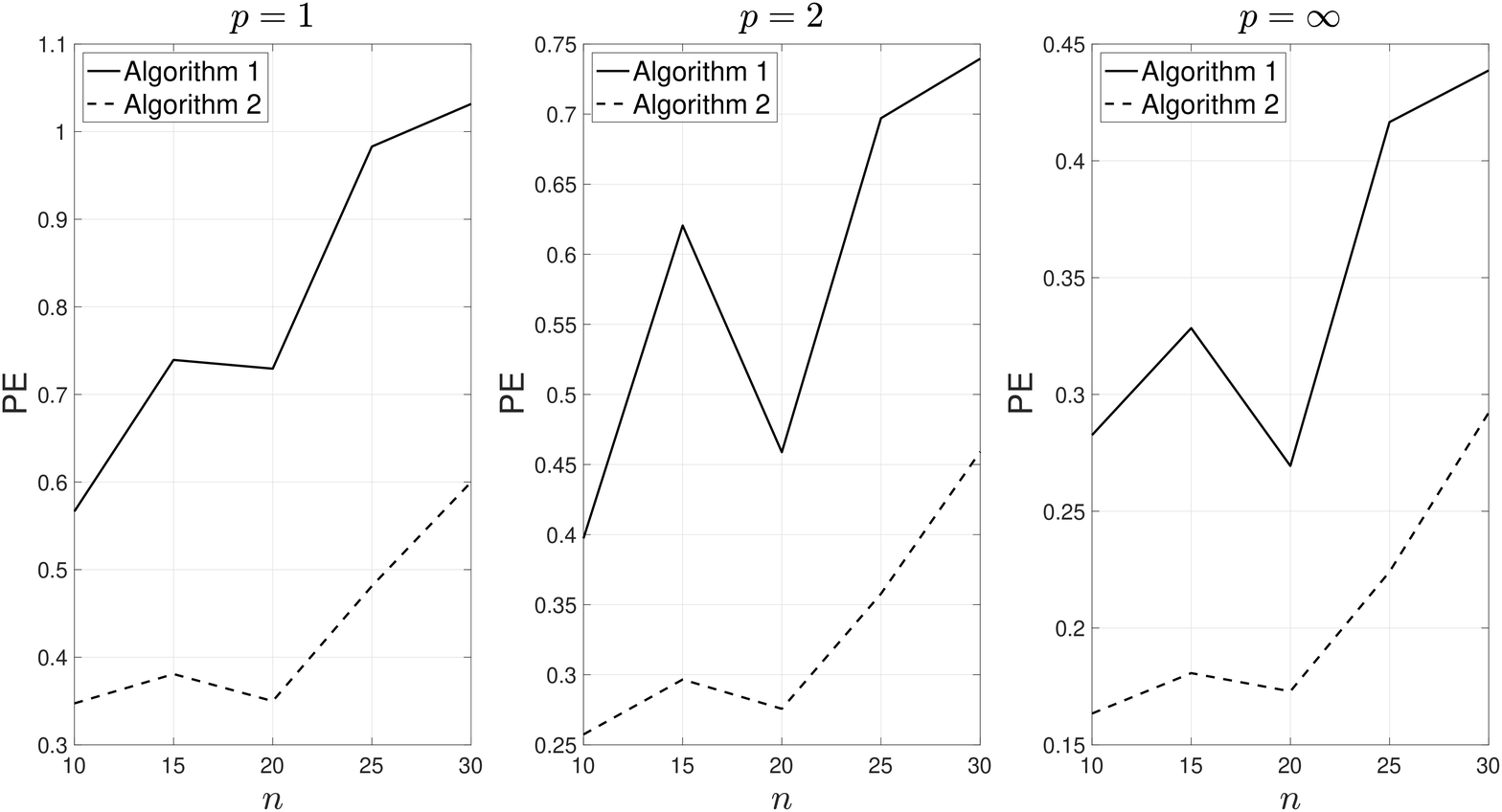}
		\includegraphics[scale=0.12]{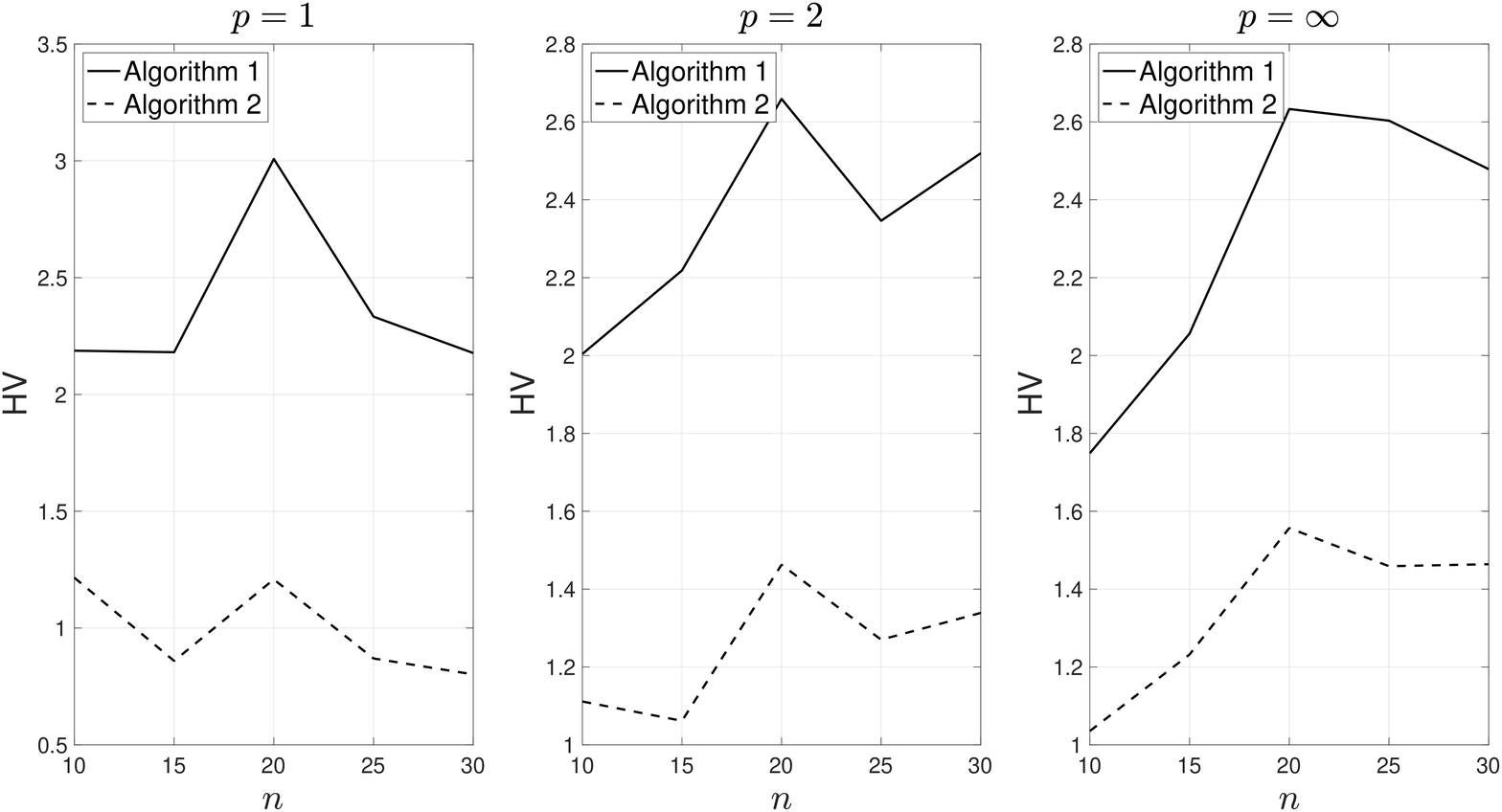}
		\caption{Average primal error (first \rev{group}) and HV (second \rev{group}) values of random instances of \Cref{example:3} with $\ell_p$ norms for $p=1$ (left), $p=2$ (middle) and $p=\infty$ (right), when the algorithms are run for 50 seconds.}
		\label{fig:11}
	\end{figure}
	%%%%%%%%%%%%%%%%%%%%%%%%%%%%%%%%%%%%%%%%%%%%%%%%%%%%%%%%%%%%%%%%%%%%%%%%%%%%%%%%%%%%%%%%%
	%%%%%%%%%PERFORMANCE PROFILES
	%%%%%%%%%%%%%%%%%%%%%%%%%%%%%%%%%%%%%%%%%%%%%%%%%%%%%%%%%%%%%%%%%%%%%%%%%%%%%%%%%%%%%%%%%
	\subsection{\rev{Comparison with algorithms from the literature}} \label{subsect:compar_lit}
	
	\rev{We compare the performance of Algorithms \ref{alg1} and \ref{alg2} with similar ones from the literature; in particular with the following algorithms which guarantee returning polyhedral inner and outer approximations to the upper image: the primal (LRU-P) and dual (LRU-D) algorithms from \cite{ulus_2014} and the (primal) algorithm (DLSW) from \cite{dorfler2021benson}.\footnote{\rev{We use MATLAB implementations of these algorithms that were also used in \cite{ulus_2014} and \cite{irem}, respectively.}}} 
	
	\rev{For the first set of experiments we use the hundred randomly generated problem instances from \Cref{subsubsect:tests1}. We solve each instance under runtime limits of 50 seconds and 100 seconds and compare the performances of Algorithms \ref{alg1}, \ref{alg2}, LRU-P, LRU-D and DLSW via empirical cumulative distribution functions of the proximity measures, see \Cref{fig:performance}. In \cite{bechmarking}, when comparing different solvers with respect to CPU times, they scale the data points by the minimum over different solvers and plot the corresponding empirical cumulative distributions, which are called \emph{performance profiles}.  } 
	
	\begin{figure}[h!]
		\centering
		\includegraphics[scale=0.21]{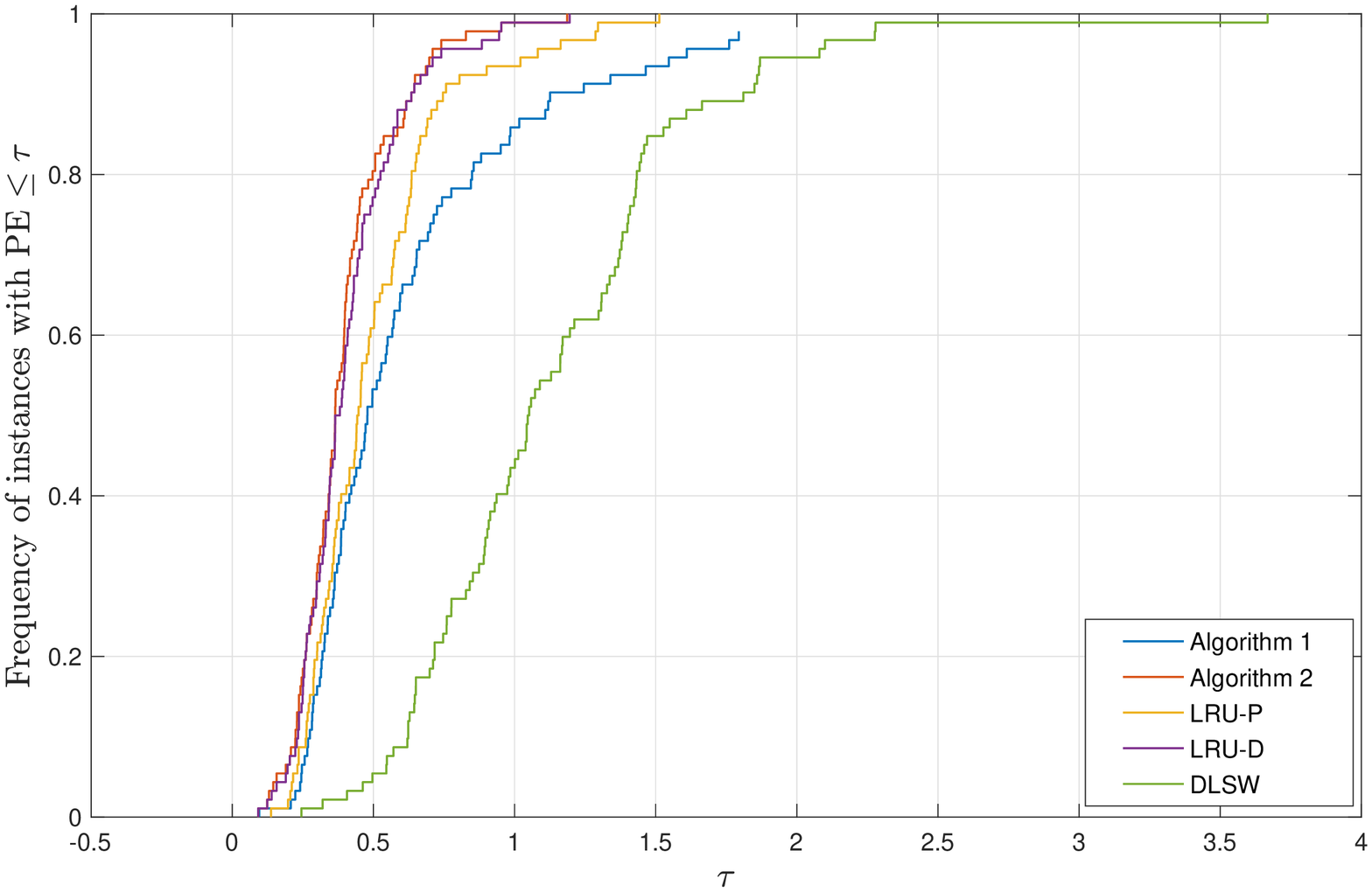}
		\includegraphics[scale=0.21]{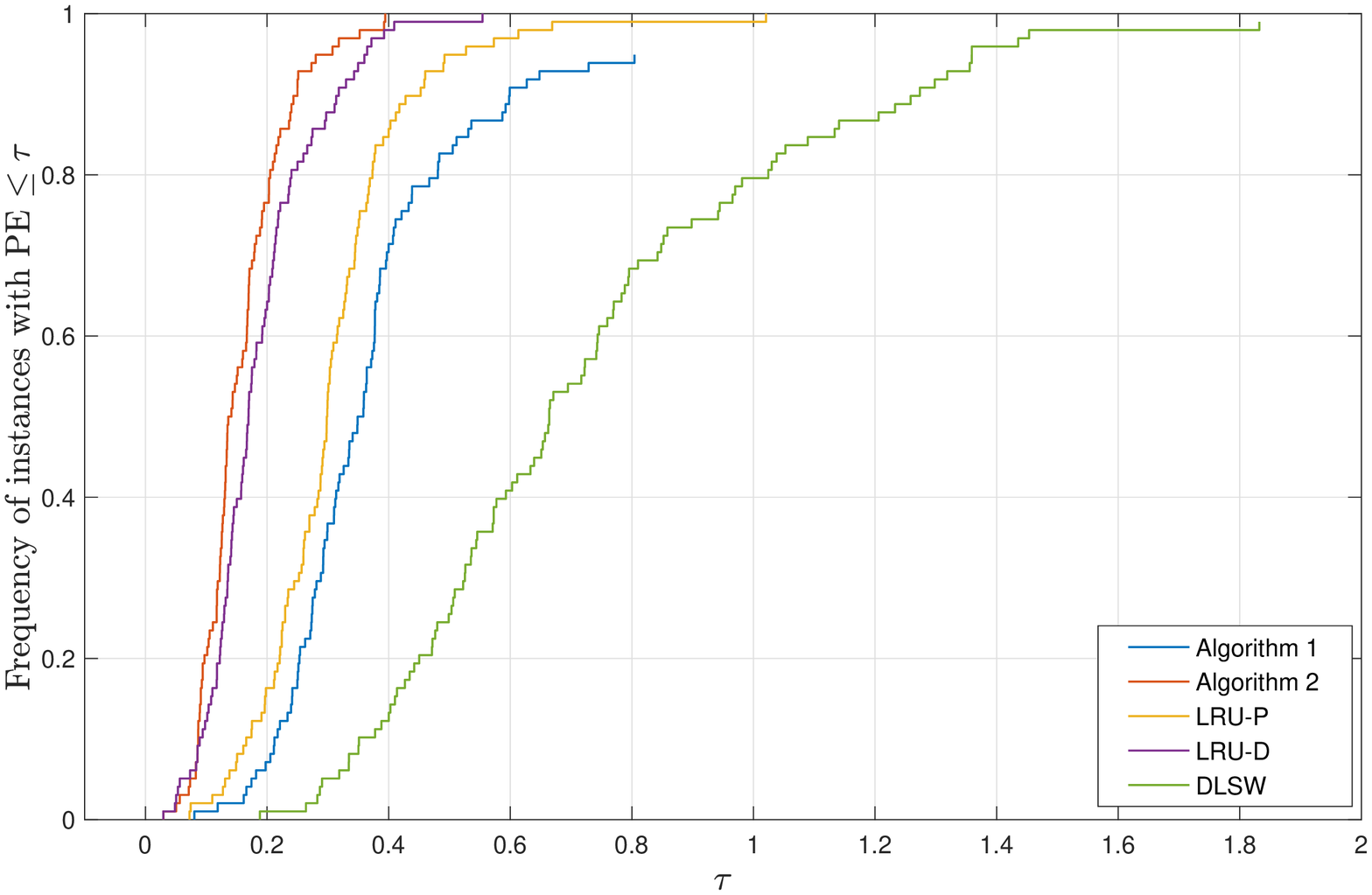}\\
		\includegraphics[scale=0.21]{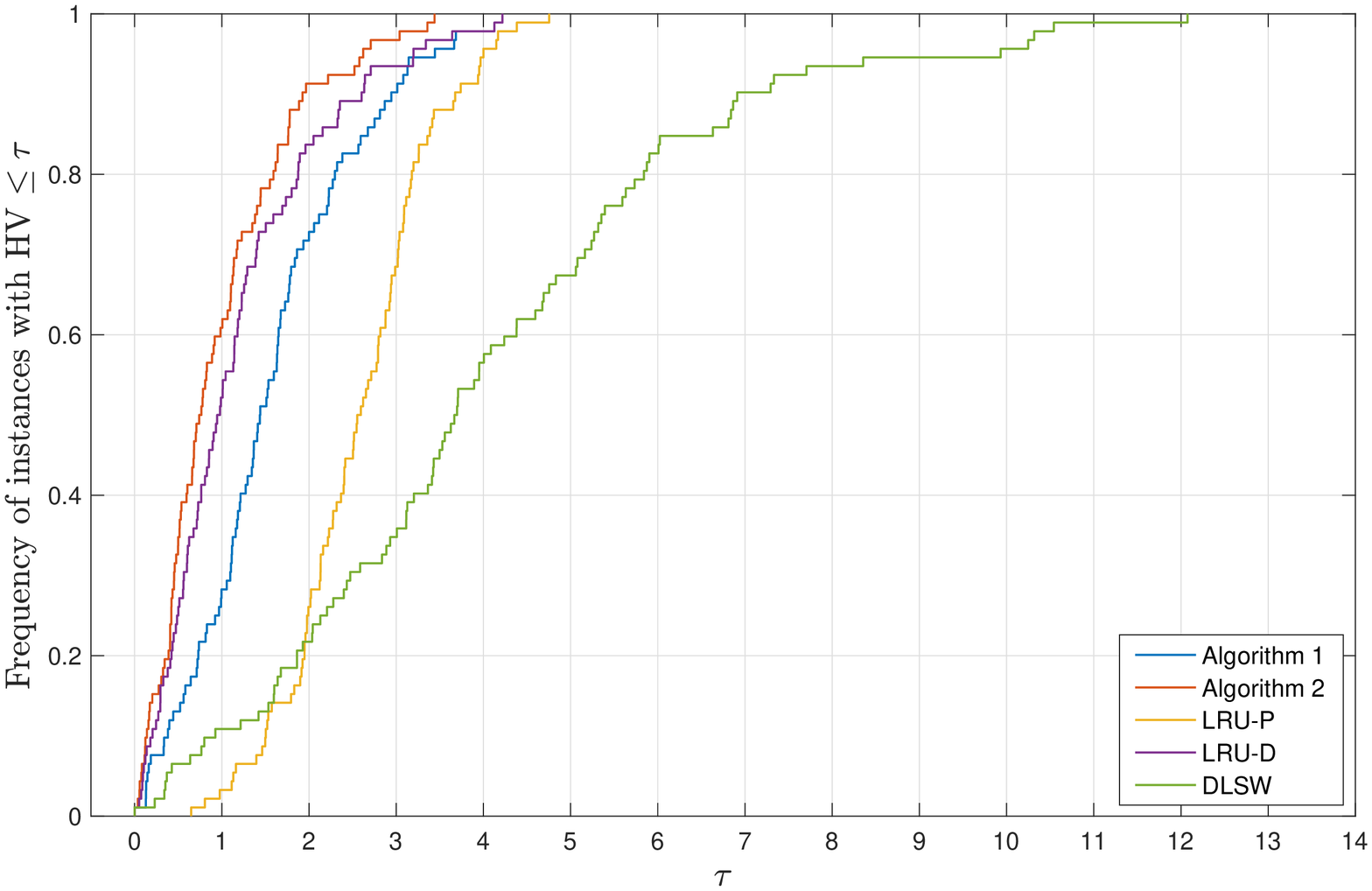}
		\includegraphics[scale=0.21]{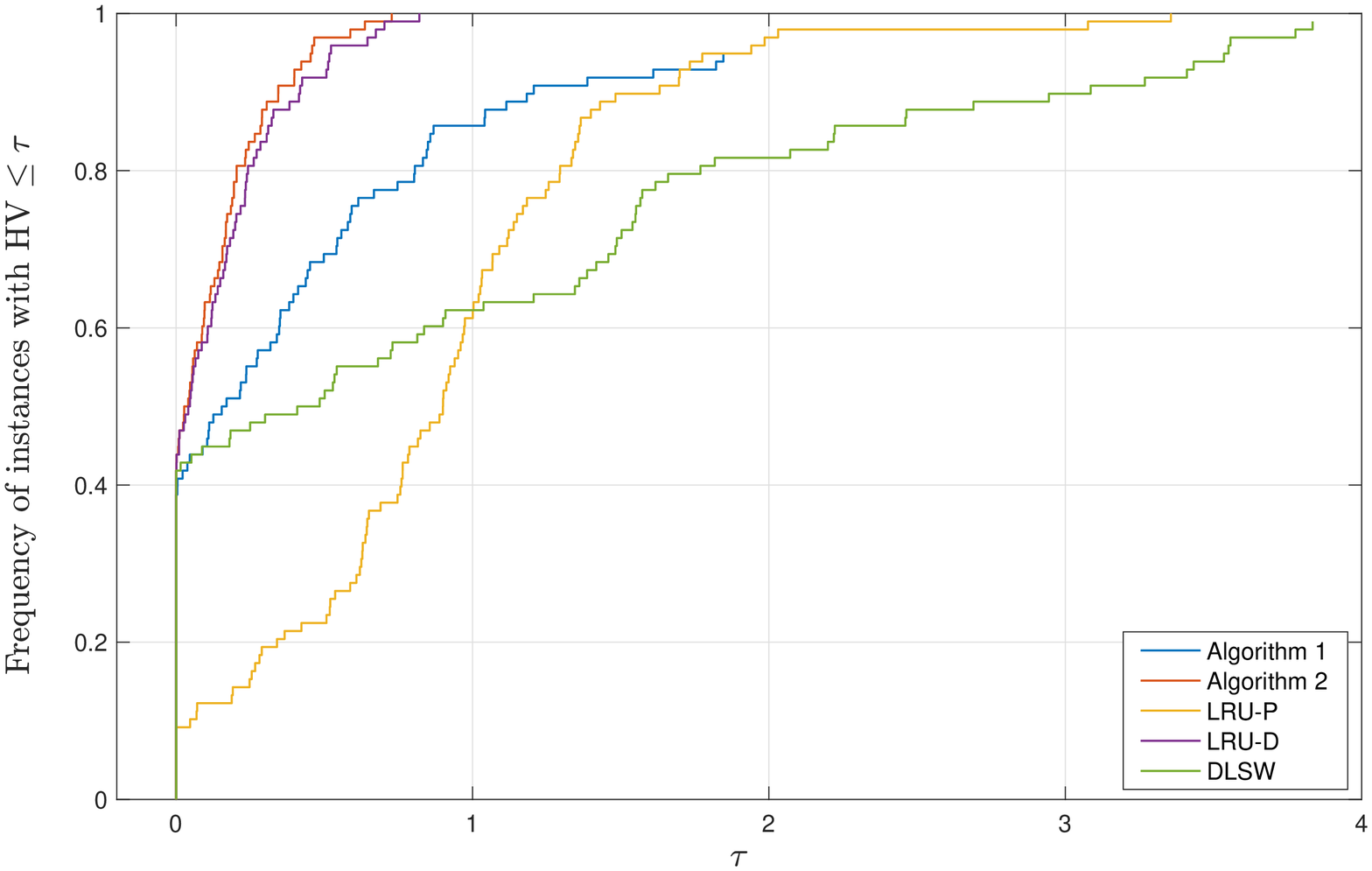} 	
		\caption{\rev{Empirical cumulative distribution functions for PE (first row) and HV (second row) values of random instances from \Cref{subsubsect:tests1} when the algorithms are run under time limit of 50 seconds (left) and 100 seconds (right).}}
		\label{fig:performance}
	\end{figure}
	
	\rev{From \Cref{fig:performance}, we see that the dual algorithms perform better than the primal ones in HV and PE under fixed run time. Moreover, \Cref{alg2} performs slightly better than LRU-D.}
	
	%%%%%%%%%%%%%%%%%%%%%%%%%%%%%%%%%%%%%%%%%%%%%%%%%%%%%%%%%%%%%%%%%%%%%%%%%%%%%%%%%%%%%%%%%%
	%%%%%%%%%% Example 8.1, 8.2,8.3 Results
	%%%%%%%%%%%%%%%%%%%%%%%%%%%%%%%%%%%%%%%%%%%%%%%%%%%%%%%%%%%%%%%%%%%%%%%%%%%%%%%%%%%%%%%%%%
	\rev{Next, we compare the algorithms over different examples from the literature. \Cref{example:4} is a special case of \Cref{example:3}, which can be also seen in \cite{ess_2011,ulus_2014}. In \Cref{example:5}, the objective functions are nonlinear while the constraints are linear, see \cite[Examples 5.8]{ess_2011}, \cite{miettinen}.%and in \Cref{example:6}, nonlinear terms appear both in the objective function and constraints, see \cite[Examples 5.8, 5.10]{ess_2011}, \cite{miettinen}. 
		We solve these examples for different norms and ordering cones as in \Cref{subsubsect:tests_2} under fixed runtime of 50 seconds. Since the objective function of \Cref{example:5} is quadratic, it is not $C$-convex for $C = C_2$, hence we solve it only with ordering cones $C_1$ and $C_3$.} 
	\rev{
		\begin{example}\label{example:4}
			We consider the following problem
			\begin{equation*}
			\begin{aligned}
			&\textnormal{minimize } f(x)=x \textnormal{ with respect to } \leq_{C} \\
			&\textnormal{ subject to } \norm{x-e}_2 \leq 1,\ x\in\R^3.
			\end{aligned}
			\end{equation*}
		\end{example}
		\begin{example}\label{example:5}
			Let $a^1=(1,1)^\T$, $a^2=(2,3)^\T$, $a^3=(4,2)^\T$. Consider
			\begin{equation*}
			\begin{aligned}
			&\textnormal{minimize } f(x)=(\norm{x-a^1}^2_2,\norm{x-a^2}^2_2,\norm{x-a^3}^2_2) \textnormal{ with respect to } \leq_{C} \\
			&\textnormal{ subject to } x_1+2x_2\leq10,\ 0\leq x_1\leq10,\ 0\leq x_2\leq4,\ x\in\R^2.
			\end{aligned}
			\end{equation*}
		\end{example}
		%\begin{example}\label{example:6}
		%    Let $b^1=(0,10,120)^\T$, $b^2=(80,-448,80)^\T$, $b^3=(-448,80,80)^\T$ and $b^1,b^2,b^3\in\R^n$. Consider
		%    \begin{equation*}
		%	\begin{aligned}
		%	&\textnormal{minimize } f(x)=(\norm{x}^2_2+b^1x,\norm{x}^2_2+b^2x,\norm{x}^2_2+b^3x)^\T \textnormal{ with respect to } \leq_{C} \\
		%	&\textnormal{ subject to } \norm{x}^2_2\leq100,\ 0\leq x_i\leq 10 \textnormal{ for }i\in\cb{1,2,3},\ x\in\R^3.
		%	\end{aligned}
		%	\end{equation*}
		%\end{example}
	}
	
	\rev{\Cref{tab:other_ex} shows the Opt, En, $\abs{\bar{\X}}$, PE and HV values that the algorithms return when run for 50 seconds.} \rev{For both examples the minimum HV values are attained by \Cref{alg2} in each setting. The same holds true also for PE values for \Cref{example:5}. However, the PE values returned by the algorithms are comparable for \Cref{example:4}.}
	
	% Table generated by Excel2LaTeX from sheet 'Sheet2'

	% Table generated by Excel2LaTeX from sheet 'Sheet2'
	\begin{table}[htbp]
		\centering
		\resizebox{0.58\textwidth}{!}{	\begin{tabular}{|c|c|c|ccccc|ccccc}
				\cline{4-13}    \multicolumn{1}{r}{} & \multicolumn{1}{r}{} &       & \multicolumn{5}{c|}{\Cref{example:4}}        & \multicolumn{5}{c|}{\Cref{example:5}} \\
				\hline
				Cone  & $p$     & Alg   & Opt   & En  & $\abs{\X}$    & PE    & HV    & \multicolumn{1}{c}{Opt} & \multicolumn{1}{c}{En} & \multicolumn{1}{c}{$\abs{\X}$} & \multicolumn{1}{c}{PE} & \multicolumn{1}{c|}{HV} \\
				\hline
				\multirow{15}[6]{*}{$C_1$} & \multirow{5}[2]{*}{1} & 1     & 38    & 5     & 38    & 0,0339 & 1,1788 & \multicolumn{1}{c}{31} & \multicolumn{1}{c}{4} & \multicolumn{1}{c}{31} & \multicolumn{1}{c}{0,4289} & \multicolumn{1}{c|}{0,00507} \\
				&       & 2     & 38    & 4     & 38    & 0,0354 & 0,9296 & \multicolumn{1}{c}{36} & \multicolumn{1}{c}{4} & \multicolumn{1}{c}{36} & \multicolumn{1}{c}{0,2046} & \multicolumn{1}{c|}{0,00130} \\
				&       & LRU-P & 41    & 5     & 41    & 0,0319 & 2,5017 & \multicolumn{1}{c}{33} & \multicolumn{1}{c}{4} & \multicolumn{1}{c}{33} & \multicolumn{1}{c}{0,4288} & \multicolumn{1}{c|}{0,35351\footnotemark} \\
				&       & LRU-D & 40    & 5     & 34    & 0,0354 & 1,1046 & \multicolumn{1}{c}{34} & \multicolumn{1}{c}{4} & \multicolumn{1}{c}{32} & \multicolumn{1}{c}{0,2790} & \multicolumn{1}{c|}{0,00174} \\
				&       & DLSW  & 46    & 11    & 13    & 0,0898 & 3,7211 & \multicolumn{1}{c}{-} & \multicolumn{1}{c}{-} & \multicolumn{1}{c}{-} & \multicolumn{1}{c}{-} & \multicolumn{1}{c|}{-} \\
				\cline{2-13}          & \multirow{5}[2]{*}{2} & 1     & 39    & 5     & 39    & 0,0267 & 3,9998 & \multicolumn{1}{c}{31} & \multicolumn{1}{c}{4} & \multicolumn{1}{c}{31} & \multicolumn{1}{c}{0,3990} & \multicolumn{1}{c|}{0,00451} \\
				&       & 2     & 41    & 4     & 41    & 0,0259 & 3,9608 & \multicolumn{1}{c}{36} & \multicolumn{1}{c}{4} & \multicolumn{1}{c}{36} & \multicolumn{1}{c}{0,1762} & \multicolumn{1}{c|}{0,00129} \\
				&       & LRU-P & 44    & 5     & 44    & 0,0266 & 4,0791 & \multicolumn{1}{c}{34} & \multicolumn{1}{c}{4} & \multicolumn{1}{c}{34} & \multicolumn{1}{c}{0,3989} & \multicolumn{1}{c|}{0,35351} \\
				&       & LRU-D & 40    & 5     & 34    & 0,0259 & 4,7979 & \multicolumn{1}{c}{33} & \multicolumn{1}{c}{4} & \multicolumn{1}{c}{31} & \multicolumn{1}{c}{0,2672} & \multicolumn{1}{c|}{0,00174} \\
				&       & DLSW  & 44    & 13    & 15    & 0,0824 & 10,3820 & \multicolumn{1}{c}{44} & \multicolumn{1}{c}{8} & \multicolumn{1}{c}{10} & \multicolumn{1}{c}{0,8495} & \multicolumn{1}{c|}{0,01205} \\
				\cline{2-13}          & \multirow{5}[2]{*}{$\infty$} & 1     & 34    & 5     & 34    & 0,0183 & 3,9642 & \multicolumn{1}{c}{27} & \multicolumn{1}{c}{4} & \multicolumn{1}{c}{27} & \multicolumn{1}{c}{0,3126} & \multicolumn{1}{c|}{0,00003} \\
				&       & 2     & 37    & 4     & 37    & 0,0189 & 3,6921 & \multicolumn{1}{c}{36} & \multicolumn{1}{c}{4} & \multicolumn{1}{c}{36} & \multicolumn{1}{c}{0,1206} & \multicolumn{1}{c|}{0,00001} \\
				&       & LRU-P & 37    & 5     & 37    & 0,0183 & 4,1461 & \multicolumn{1}{c}{34} & \multicolumn{1}{c}{4} & \multicolumn{1}{c}{34} & \multicolumn{1}{c}{0,3124} & \multicolumn{1}{c|}{0,04624} \\
				&       & LRU-D & 37    & 5     & 31    & 0,0189 & 4,3730 & \multicolumn{1}{c}{37} & \multicolumn{1}{c}{5} & \multicolumn{1}{c}{34} & \multicolumn{1}{c}{0,1206} & \multicolumn{1}{c|}{0,00001} \\
				&       & DLSW  & 42    & 9     & 11    & 0,0635 & 12,5425 & \multicolumn{1}{c}{45} & \multicolumn{1}{c}{10} & \multicolumn{1}{c}{12} & \multicolumn{1}{c}{0,3155} & \multicolumn{1}{c|}{0,00006} \\
				\hline
				\multirow{5}[2]{*}{$C_3$} & \multirow{5}[2]{*}{2} & 1     & 36    & 3     & 36    & 0,0166 & 2,3381 & \multicolumn{1}{c}{31} & \multicolumn{1}{c}{3} & \multicolumn{1}{c}{31} & \multicolumn{1}{c}{0,1964} & \multicolumn{1}{c|}{2,08492} \\
				&       & 2     & 41    & 3     & 41    & 0,0148 & 2,1469 & \multicolumn{1}{c}{36} & \multicolumn{1}{c}{3} & \multicolumn{1}{c}{36} & \multicolumn{1}{c}{0,1346} & \multicolumn{1}{c|}{1,66544} \\
				&       & LRU-P & 40    & 3     & 40    & 0,0541 & 4,1121 & \multicolumn{1}{c}{30} & \multicolumn{1}{c}{3} & \multicolumn{1}{c}{30} & \multicolumn{1}{c}{0,1964} & \multicolumn{1}{c|}{4,38963} \\
				&       & LRU-D & 35    & 4     & 33    & 0,0274 & 3,0475 & \multicolumn{1}{c}{30} & \multicolumn{1}{c}{4} & \multicolumn{1}{c}{24} & \multicolumn{1}{c}{0,3065} & \multicolumn{1}{c|}{3,34918} \\
				&       & DLSW  & 45    & 9     & 14    & 0,0541 & 4,1913 & \multicolumn{1}{c}{42} & \multicolumn{1}{c}{9} & \multicolumn{1}{c}{14} & \multicolumn{1}{c}{0,1964} & \multicolumn{1}{c|}{3,58336} \\
				\hline
				\multirow{5}[2]{*}{$C_2$} & \multirow{5}[2]{*}{2} & 1     & 36    & 3     & 36    & 0,1926 & 6,0955 &       &       &       &       &  \\
				&       & 2     & 41    & 3     & 41    & 0,1150 & 4,1466 &       &       &       &       &  \\
				&       & LRU-P & 44    & 3     & 44    & 0,0498 & 4,5592 &       &       &       &       &  \\
				&       & LRU-D & 36    & 4     & 33    & 0,1393 & 8,7210 &       &       &       &       &  \\
				&       & DLSW  & 43    & 9     & 14    & 0,1064 & 9,0697 &       &       &       &       &  \\
				\cline{1-8}    \end{tabular}}%
		\caption{\rev{Results for Examples \ref{example:4} and \ref{example:5} with different ordering cones and with different norms used in \eqref{PS(v)}, when the algorithms are run for T=$50$ seconds.}}
		\label{tab:other_ex}%
	\end{table}%
	\footnotetext{\rev{In \Cref{example:4} with cone $C_1$, the outer approximations returned by LRU-P contain outlier vertices in the sense that their distances to $f(\X)$ are quite large while distances to $\P$ are small. This explains high HV values compared to small PE values.}}

	\section*{Acknowledgments}
	
	This work is supported by the 3501 program of T\"{U}B{\.I}TAK (Scientific \& Technological Research Council of Turkey), Project No. 118M479. The authors thank three anonymous reviewers and an associate editor, as well as Andreas L\"{o}hne and \"{O}zlem Karsu, for their valuable feedback and suggestions for improvement on earlier versions of the paper. They also thank \.Irem Nur Keskin for sharing MATLAB implementation of DLSW algorithm which was also used in \cite{irem}.
	
	\appendix
	\section{(Alternative) Proofs of some results}
%		\begin{proposition}\label{prop:Dcone}
%		The lower image $\D$ is a closed convex cone.
%	\end{proposition}
	\begin{proof}[\textbf{Proof of \Cref{prop:Dcone}}]
		Let $((w^i)^{\T},\alpha_i)^{\T}\in\D$ and $\lambda_i\geq 0$ for each $i\in\{1,\dots,n\} $, where $n\in\mathbb{N}$. Since $\W=C^+ $ is a convex cone and we have $w^i\in C^+ $ for each $i\in\{1,\dots,n\}$, we have $\sum_{i=1}^{n}\lambda_iw^i \in \W $. Moreover, we have $\alpha_i\leq\inf_{x\in\X} (w^i)^{\T}f(x)$ for each $i\in\{1,\ldots,n\}$, which implies that
		\begin{equation*}
		\sum_{i=1}^{n}\lambda_i\alpha_i \leq \sum_{i=1}^{n}\lambda_i \inf_{x\in\X}(w^i)^{\T}f(x) \leq \inf_{x\in\X}\of{\sum_{i=1}^{n}\lambda_iw^i}^{\T}f(x).
		\end{equation*}
		Hence, $\sum_{i=1}^{n}\lambda_i((w^i)^{\T},\alpha_i)^{\T}\in \D$. It follows that $\D$ is a convex cone.	
		
	Note that the function $w\mapsto p^w$ is continuous as a concave function on $\mathcal{W}$ with finite values. Moreover, $\D \subseteq \W\times\R$ coincides with the hypograph of this function. Hence, $\D$ is a closed set. 
	\end{proof}
	 
%		\begin{proposition}\label{lem:duality_1}
%		(a) Let $y\in\R^q$. Then, $y$ is a weakly $C$-minimal element of $\P$ if and only if $H^{\ast}(y)\cap\D$ is a $K$-maximal \rev{exposed} face of $\D$.
%		%\item
%		(b) For every $K$-maximal \rev{exposed} face $F^\ast$ of $\D$, there exists some $y\in \P$ such that $F^\ast=H^\ast(y)\cap\D$. 
%	\end{proposition}
	
	\begin{proof}[\textbf{Proof of \Cref{lem:duality_1}}]
		\begin{enumerate}[(a)]
			\item Suppose that $y$ is a weakly $C$-minimal element of $\P$. By \Cref{prop:upperimage}, $y=f(\bar{x})+\bar{c}$ for some $\bar{x}\in\X$ and $\bar{c}\in C$. Moreover, as $y$ is a weakly $C$-minimal element of $\P$, we have $y\in \bd\P$ \cite[Corollary 1.48]{lohne_2011}. Then, there exist $\bar{w}\in\R^q $ and $\bar{\alpha}\in\R$ such that $H\coloneqq\{\tilde{y}\in\R^q\mid\bar{w}^{\T}\tilde{y}=\bar{\alpha}\} $ supports $\P$ at $y$. In particular, we have
			\begin{equation}\label{eq:reccCheck}
			\inf_{\tilde{y}\in\P} \bar{w}^{\T}\tilde{y} =\bar{w}^{\T}y=\bar{\alpha}.
			\end{equation}
			By \Cref{rem:w_inC+}, \eqref{eq:reccCheck} implies that $\bar{w}\in C^+$. 
			%\Fir{(İki türlü de yazılabilir aslında ama bu daha direkt gibi geldi.)} %(To see, let $\tilde{y}\in\mathcal{P} $, $c\in C $ be arbitrary. As $\tilde{y}+C\in\mathcal{P} $, we have $\bar{w}^{\T}\tilde{y}+\bar{w}^{\T}c\geq\bar{\alpha} $. Hence, $\bar{w}^{\T}c\geq0 $ for each $c\in C $.) 
			Moreover, using \Cref{prop:upperimage} and the fact that $y=f(\bar{x})+\bar{c}$, \eqref{eq:reccCheck} can be rewritten as %we have $\bar{w}^{\T}y=\inf_{\tilde{y}\in\mathcal{P}}\bar{w}^{\T}\tilde{y} $ as $H$ supports $\mathcal{P} $ at $y$. Then,
			\begin{equation*}
			\inf_{x\in\X}\bar{w}^{\T}f(x)+\inf_{c\in C}\bar{w}^{\T}c=\bar{w}^{\T}f(\bar{x})+\bar{w}^{\T}\bar{c}.
			\end{equation*}
			Noting that we have $\inf_{c\in C}\bar{w}^{\T}c=0$,  $\inf_{x\in\X}\bar{w}^{\T}f(x) \leq \bar{w}^{\T}f(\bar{x})$ and $\bar{w}^{\T}\bar{c}\geq0$, we obtain
			%	\begin{equation*}
			%	\bar{w}^{\T}f(\bar{x})+\bar{w}^{\T}\bar{c}&=\underset{x\in\mathcal{X}}{\textnormal{inf}}\bar{w}^{\T}f(x).
			%	\end{equation*}
			%	We also know that $\bar{w}^{\T}f(\bar{x})\geq \underset{x\in\mathcal{X}}{\textnormal{inf}}\bar{w}^{\T}f(x) $ and $\bar{w}^{\T}\bar{c}\geq0$. Therefore, the equalities 
			$\bar{w}^{\T}\bar{c}=0$ and $\bar{w}^{\T}f(\bar{x})= \inf_{x\in\X}\bar{w}^{\T}f(x)$. It follows that $\H^\ast(f(\bar{x}))= \H^\ast(y)$ since $y=f(\bar{x})+\bar{c}$ and $\bar{w}^{\T}\bar{c} = 0$. Moreover, $\bar{x} $ is an optimal solution to the problem $\text{(WS(}\bar{w}\text{))}$. Hence, %$\bar{x}\in\mathcal{X}$ is a weak minimizer of \eqref{P} by \Cref{prop:WS} and 
			$H^\ast(f(\bar{x}))=H^\ast(y)$ is a supporting hyperplane of $\D$ at $\xi(\bar{w})$ such that $\D\subseteq \H^\ast(f(\bar{x}))$ by \Cref{prop:supp_halfsp}. Therefore, $H^\ast(f(\bar{x}))\cap\D$ is a proper face of $\D$ \cite[Section 18, page 162]{rockafellar_1970}.
			
			To show that $H^\ast(f(\bar{x}))\cap\D$ is a $K$-maximal proper face of $\D$, let $(\tilde{w}^{\T},\tilde{\alpha})^{\T}\in H^\ast(f(\bar{x}))\cap\D$ be arbitrary. Note that $\varphi(f(\bar{x}),\tilde{w},\tilde{\alpha})=\tilde{w}^{\T}f(\bar{x})-\tilde{\alpha}=0$. %, for any $\hat{\alpha} > \tilde{\alpha}$, $\varphi(f(\bar{x}),\tilde{w},\hat{\alpha})=\tilde{w}^{\T}f(\bar{x})-\hat{\alpha}<0  $ holds. 
			On the other hand, the fact $\D\subseteq \H^\ast(f(\bar{x}))$ implies that $\varphi(f(\bar{x}),w,\alpha)=w^{\T}f(\bar{x})-\alpha\geq 0$ for each $(w^{\T},\alpha)^{\T}\in\D$.
			%As $H^*(f(\bar{x})) $ is a supporting hyperplane to $\mathcal{D} $, each $(w^{\T},\alpha)^{\T}\in\mathcal{D} $ satisfies $\varphi(f(\bar{x}),w,\alpha)=w^{\T}f(\bar{x})-\alpha\geq 0 $. 
			Together, these imply that $(\tilde{w}^{\T},\hat{\alpha})^{\T} \notin \D$ for every $\hat{\alpha}>\tilde{\alpha}$. %This implies $\big(\{(\tilde{w}^{\T},\tilde{\alpha})^{\T}\}+K\setminus\{0\}\big) \cap \mathcal{D}=\emptyset $, or equivalently,
			Hence, $(\tilde{w}^{\T},\tilde{\alpha})^{\T}$ is a $K$-maximal element of $\D$.
			%Therefore, $H^*(f(\bar{x}))\cap\mathcal{D} $ is a $K$-maximal proper face.
			
			Conversely, suppose that $H^\ast(y)\cap\D$ is a $K$-maximal proper face of $\D$. Hence, $H^\ast(y)$ is a supporting hyperplane of $\D$ \cite[Section 18, page 162]{rockafellar_1970}, and we have either $\D\subseteq \H^\ast(y) $ or $\D\subseteq\{(w^{\T},\alpha)^{\T}\in\R^{q+1}\mid\varphi(y,w,\alpha)\leq 0\} $. We claim that the former relation holds. Indeed, letting $(\tilde{w}^{\T},\tilde{\alpha})^{\T}\in H^\ast(y)\cap\D$, we have $\tilde{\alpha}\leq\inf_{x\in\X}\tilde{w}^{\T}f(x) $ and $\varphi(y,\tilde{w},\tilde{\alpha})=\tilde{w}^{\T}y-\tilde{\alpha}=0$. Then, $(\tilde{w}^{\T},\tilde{\alpha}-1)^{\T}\in \D$ and $\varphi(y,\tilde{w},\tilde{\alpha}-1)=\tilde{w}^{\T}y-\tilde{\alpha}+1>0$. Hence, $(\tilde{w}^{\T},\tilde{\alpha}-1)^{\T}\notin\{(w^{\T},\alpha)^{\T}\in\R^{q+1}\mid\varphi(y,w,\alpha)\leq 0\} $. Therefore, the claim holds and we have $\D\subseteq \H^\ast(y) $.
			
			Next, we show that $y\in\P$. To get a contradiction, suppose that $y\notin\P$. By separation theorem, there exists $\bar{w}\in\mathbb{R}^q\setminus\{0\} $ such that
			\[
			\bar{w}^{\T}y<\inf_{\bar{y}\in\P}\bar{w}^{\T}\bar{y}=:\bar{\alpha}.
			\]
			Note that $\bar{w}\in C^+$ by \Cref{rem:w_inC+}. Therefore, we have $(\bar{w}^{\T},\bar{\alpha})^{\T} \in \D\subseteq \H^\ast(y) $, which contradicts $\bar{w}^{\T}y-\bar{\alpha}<0 $. Hence, we have $y\in\P$.
			
			Finally, we show that $y$ is a weakly $C$-minimal element of $\P$. To get a contradiction, suppose that there exists $c \in \Int C $ with $y-c \in\P$. Without loss of generality, assume that $y-c \in \wMin_C \P$. From the proof of the previous implication, $\D \subseteq \H^\ast(y-c)$. Let $(w^{\T},\alpha)^{\T}\in H^\ast(y)\cap\D\subseteq \H^\ast(y-c)$. Then, we have $w^{\T}(y-c)-\alpha=w^{\T}y-w^{\T}c-\alpha\geq0$. %\Cag{(How do we get this? I couldn't figure out.)} 
			Note that $w^{\T}y-\alpha=0 $ since $(w^{\T},\alpha)^{\T}\in H^\ast(y)\cap\D$. Hence, we have $w^{\T}y-w^{\T}c-\alpha=-w^{\T}c\geq0$, which contradicts $w\in \W$ and $c \in \Int C $. Therefore, $y$ is a weakly $C$-minimal element of $\P$. 	
			
			\item Let $F^\ast$ be a $K$-maximal proper face of $\D$. Then, there exists a supporting hyperplane $H^\ast\subseteq \R^{q+1}$ of $\D$ such that $F^\ast=H^\ast\cap\D$ \cite[Section 18, page 162]{rockafellar_1970}. Then, we may write
			\[
			H^\ast = \{(w^{\T},\alpha)^{\T}\in\R^{q+1}\mid a^{\T}(w^{\T},\alpha)=b\}
			\]
			for some $a\in\R^{q+1}, b\in\R$. Without loss of generality, we may assume that
			\[
			\H^\ast\coloneqq\{(w^{\T},\alpha)^{\T}\in\R^{q+1}\mid a^{\T}(w^{\T},\alpha)^\T\geq b\} \supseteq \D.
			\]
			Since $\D$ is a convex cone, $F^\ast$ is also a convex cone \cite[Lemma 10.2]{glockner_2003}. Moreover, $F^\ast$ is closed as $\D$ is closed by \Cref{prop:Dcone}. %(Is this true, obvious, non-trivial? A citation would be good.)} 
			Therefore, $0\in F^\ast \subseteq H^\ast $. Hence, $b=0$.%, that is, $H^\ast=\{(w^{\T},\alpha)^{\T}\in\mathbb{R}^{q+1}\mid a^{\T}(w^{\T},\alpha)^{\T}=0\} $. 
			
			Next, we show that $a_{q+1}<0 $. For every $(\bar{w}^{\T},\bar{\alpha})^{\T}\in F^\ast$, the point $(\bar{w}^{\T},\bar{\alpha})^{\T}$ is a $K$-maximal element of $\D$ and it holds $a^{\T}(\bar{w}^{\T},\bar{\alpha})=0$. Moreover, $(\bar{w}^{\T},\bar{\alpha})^{\T}\in \H^\ast$ implies %as $H^\ast$ is a supporting hyperplane, for any $(w^{\T},\alpha)^{\T}\in\D $ we have 
			$a^{\T}(w^{\T},\alpha)\geq0$. For every $\gamma >0$, since $(\bar{w}^{\T},\bar{\alpha}-\gamma)^{\T}\in\D  \subseteq \H^\ast$, we have $a^{\T}(\bar{w}^{\T},\bar{\alpha}-\gamma)=a^{\T}(w^{\T},\alpha)-\gamma a_{q+1}\geq 0$. %$a^{\T}(\bar{w}^{\T},\bar{\alpha}-\gamma)=(\sum_{i=1}^{q}a_i\bar{w}_i) + a_{q+1}\bar{\alpha}-a_{q+1}\gamma\geq 0 $.
			Therefore, $a_{q+1}\leq0 $ holds. If $a_{q+1}=0 $, then $a^{\T}(\bar{w}^{\T},\bar{\alpha}-1)^{\T}=0$ implies that $(\bar{w}^{\T},\bar{\alpha}-1)^{\T} \in F^\ast$ contradicting the $K$-maximality of $F^\ast$. Therefore, $a_{q+1}<0 $. 
			
			By setting
			\begin{equation*}
			\begin{aligned}
			y:=\bigg(\frac{-a_1}{a_{q+1}},\frac{-a_2}{a_{q+1}},\cdots,\frac{-a_q}{a_{q+1}}\bigg)^{\T},
			\end{aligned}
			\end{equation*}
			we obtain $H^\ast=H^\ast(y)$ and $\D \subseteq \H^\ast(y)$.
			
			Finally, we show that $y\in \P$. Assuming otherwise, there exists $\tilde{w}\in C^+$ such that $\tilde{w}^\T y<\inf_{x\in\X}\tilde{w}^\T f(x)\eqqcolon \tilde{\alpha}$ by separation arguments. Then, we obtain $(\tilde{w}^\T,\tilde{\alpha})^\T \in \D\setminus \H^\ast(y)$, which contradicts with $\D \subseteq \H^\ast(y)$. %Hence, there exists \Fir{$y\in\P$} such that $F^\ast=H^\ast(y)\cap\D $.
		\end{enumerate}
	\end{proof}

%	\begin{proposition}\label{lem:duality_2}
%		%	The following statements hold true:
%		%	\begin{enumerate}[(a)]
%		%	\item 
%		(a) Let $(w^{\T},\alpha)^{\T}\in\R^{q+1}\rev{\setminus\{0\}}$. Then, $(w^{\T},\alpha)^{\T}$ is a $K$-maximal element of $\D$ if and only if $H(w,\alpha)\cap\P$ is a weakly $C$-minimal \rev{exposed} face of $\P$ satisfying $\H(w,\alpha)\supseteq\P$.
%		%	\item 
%		(b) For every $C$-minimal \rev{exposed} face $F$ of $\P$, there exists some $(w^{\T},\alpha)^{\T}\in \D$ such that $F=H(w,\alpha)\cap\P $. %\Fir{(Teorem 4.4 için $(w^{\T},\alpha)^{\T}\in\D$ olduğunu göstermemiz gerekiyor sanıyorum. O yüzden ifadeyi değiştirip ispatı ekledim.)}
%		%\end{enumerate}
%	\end{proposition}
	
	\begin{proof}[\textbf{Proof of \Cref{lem:duality_2}}]
		\begin{enumerate}[(a)]
			\item Suppose that $(w^{\T},\alpha)^{\T}$ is a $K$-maximal point of $\D$. Clearly, $w\in C^+$, $\alpha=\inf_{x\in\X}w^{\T}f(x)$ and $(w^{\T},\alpha)^{\T}=\xi(w)$. Since $\X$ is a compact set, there exists an optimal solution $x^w\in\X$ to \eqref{WS(w)}. By \Cref{prop:supp_halfsp}, $H(\xi(w))=H(w,\alpha) $ is a supporting hyperplane of $\P$ at $f(x^{w})$ satisfying $\H(w,\alpha) \supseteq \P$. Then, $H(w,\alpha)\cap\P$ is a proper face of $\P$ \cite[Section 18, page 162]{rockafellar_1970}. To show that $H(w,\alpha)\cap\P$ is weakly $C$-minimal, let $\bar{y}\in H(w,\alpha)\cap\P$ be arbitrary. Since $\bar{y}\in H(w,\alpha)$ and $w \in C^+$, we have $\varphi(\bar{y}-c,w,\alpha) = w^{\T}\bar{y} - w^{\T}c - \alpha <0$ for every $c\in\Int C$. Note that each $y\in\P \subseteq \H(w,\alpha)$ satisfies $\varphi(y,w,\alpha)=w^{\T}y-\alpha\geq0 $. Then, $(\{\bar{y}\}-\Int C)\cap\P=\emptyset$, hence $\bar{y}$ is weakly $C$-minimal. %As any $\bar{y}\in H(w,\alpha)\cap\P$ is weakly $C$-minimal, $H(w,\alpha)\cap\P$ is a weakly $C$-minimal proper face of $\P$.
			
			Conversely, suppose that $H(w,\alpha)\cap\P$ is a weakly $C$-minimal proper face of $\P$ such that $\H(w,\alpha)\supseteq\P$. Then, $H(w,\alpha) $ is a supporting hyperplane of $\P$ \cite[Section 18, page 162]{rockafellar_1970}. By \Cref{rem:w_inC+}, we have $w\in (\rec \P)^+\subseteq C^+$. Moreover, for each $y\in\P$, $\varphi(y,w,\alpha) = w^{\T}y - \alpha\geq 0$. This implies $\alpha\leq \inf_{y\in\P}w^{\T}y$, hence $(w^{\T}, \alpha)^{\T} \in \D$. On the other hand, let 
			%Assume that $(w^{\T},\alpha)^{\T} $ is not a $K$-maximal point of $\D$, that is, there exists $\epsilon>0 $ such that $(w^{\T},\alpha+\epsilon)^{\T}\in\D$. Let 
			$y=f(x)+c \in H(w,\alpha)\cap\P$ for some $x\in\X $ and $c\in C$. Since $\varphi(y,w,\alpha)=w^{\T}f(x)+w^{\T}c-\alpha=0$, we have  $\varphi(y,w,\alpha+\epsilon) = w^{\T}f(x) + w^{\T}c -\alpha -\epsilon<0$ for every $\epsilon >0$. This implies 
			\[
			\alpha+\epsilon>\inf_{x\in\X}w^{\T}f(x)+\inf_{c\in C}w^{\T}c=\inf_{x\in\X}w^{\T}f(x).
			\]
			Since $\epsilon >0$ is arbitrary, we have $\alpha \geq \inf_{x\in\X}w^{\T}f(x)$. Together, we obtain $\alpha = \inf_{x\in\X}w^{\T}f(x)$, which implies that $(w^{\T},\alpha)^{\T} $ is $K$-maximal.
			%Therefore we have a contradiction and $(w^{\T},\alpha+\epsilon)^{\T}\notin\D$. So $(w^{\T},\alpha)^{\T} $ is a $K$-maximal point of $\D$.
			\item Let $F$ be a $C$-minimal proper face of $\P$. Then, there exists a supporting hyperplane $H $ of $\P$ such that  $F=H\cap\P $ \cite[Section 18, page 162]{rockafellar_1970}. We may write $H=\{y\in\R^{q} \mid w^{\T}y=\alpha\} $ for some $w\in \R^q $ and $\alpha\in\R$, and assume that $\P\subseteq \H \coloneqq \{y\in \R^q \mid w^{\T}y\geq \alpha\}$ without loss of generality. By \Cref{rem:w_inC+}, we have $w \in (\rec\P)^+ \subseteq C^+$. Moreover, as $\P\subseteq\H$, it holds true that $\alpha \leq \inf_{x\in \X}w^{\T}f(x)$. Hence, $(w^{\T},\alpha)^{\T}\in\D$. %For any $y\in\P $, $c\in C $, we have $w^{\T}(y+c)=w^{\T}y+w^{\T}c\geq\alpha $. \Fir{Since $C$ is a cone,} $w^{\T}c\geq0$ holds for any $c\in C$, hence $w\in C^+$.
		\end{enumerate}
	\end{proof}
	
	\begin{proof}[\textbf{Alternative proof of \Cref{thm:duality}}]
		First, for a $K$-maximal proper face $F^\ast$ of $\D$, we show that $\Psi(F^\ast) $ is a weakly $C$-minimal proper face of $\P$. By \Cref{lem:duality_2} (a), $H(w,\alpha)\cap\P$ is a weakly $C$-minimal proper face of $\P$ for each $(w^{\T},\alpha)^{\T}\in F^\ast$. From the definition given by \eqref{eqn:Psi}, $\Psi(F^\ast)$ is a weakly $C$-minimal proper face of $\P$ if it is nonempty. %Now we show that $\Psi(F^\ast) $ is nonempty. 
		By \Cref{lem:duality_1} (b), there exists some $y\in\P$ such that $F^\ast=H^\ast(y)\cap\D$. Therefore, for each $(w^{\T},\alpha)^{\T}\in F^\ast$, we have $(w^{\T}\alpha)^{\T}\in H^\ast(y)$, equivalently, $y\in H(w,\alpha)$, see \eqref{eqn:H_H*}. Then, $\Psi(F^\ast)$ is nonempty as $y \in\Psi(F^\ast)$.
		
		For a weakly $C$-minimal proper face $F$ of $\P$, define $\Psi^\ast(F)\coloneqq\bigcap_{ y\in F}(H^\ast(y)\cap\D)$. To show that $\Psi^\ast(F)$ is a $K$-maximal proper face of $\D$, let $y\in F$. By \Cref{lem:duality_1} (a), $H^\ast(y)\cap\D$ is a $K$-maximal proper face of $\D$. Therefore, $\Psi(F^\ast)$ is a $K$-maximal proper face of $\D$, if it is nonempty. %Now we show $\Psi^\ast(F)$ is nonempty. 
		From \Cref{lem:duality_2} (b), there exists some $(w^{\T},\alpha)^{\T}\in\D$ such that $F=H(w,\alpha)\cap\P$. Therefore, for each $y\in F$, we have $y\in H(w,\alpha)$, equivalently, $(w^{\T},\alpha)^{\T}\in H^\ast(y) $, see \eqref{eqn:H_H*}. Then, $\Psi^\ast(F)$ is nonempty as $(w^{\T},\alpha)^{\T}\in\Psi^\ast(F)$.
		
		In order to show that $\Psi$ is a bijection and %the equality
		%	\begin{equation*}
		%	\Psi^{-1}(F) =\bigcap_{y\in F} H^\ast(y)\cap \D=\Psi^\ast(F)
		%	\end{equation*}
		$\Psi^{-1} =\Psi^\ast$, we will show the following two statements:
		\begin{itemize}
			\item[(a)] $\Psi^\ast(\Psi(F^\ast))=F^\ast$ for every $K$-maximal proper face $F^\ast$ of $\D$,
			\item[(b)] $\Psi(\Psi^\ast(F))=F $ for every weakly $C$-minimal proper face $F$ of $\P$.
		\end{itemize}
		(a) Let $F^\ast$ be a $K$-maximal proper face of $\D$. Assume for a contradiction that $ F^\ast \nsubseteq \Psi^\ast(\Psi(F^\ast))$. Let $(w^{\T},\alpha)^{\T}\in F^\ast\setminus\Psi^\ast(\Psi(F^\ast))$. Since $\Psi(F^\ast)$ is nonempty, $(w^{\T},\alpha)^{\T}\notin \Psi^\ast(\Psi(F^\ast))$ means that there exists $\bar{y}\in \Psi(F^\ast)$ such that $(w^{\T},\alpha)^{\T}\notin H^\ast(\bar{y})\cap\D$. This implies $(w^{\T},\alpha)^{\T}\notin H^\ast(\bar{y})$ since $(w^{\T},\alpha)^{\T}\in\D$. Using \eqref{eqn:H_H*}, we have $\bar{y}\notin H(w,\alpha)$. Therefore $\bar{y}\notin\Psi(F^\ast)$, a contradiction. Hence, $F^\ast\subseteq\Psi^\ast(\Psi(F^\ast)) $. For the reverse inclusion, first note that, from \Cref{lem:duality_1} (b), there exists $\bar{y}\in\P$ such that $F^\ast=H^\ast(\bar{y})\cap\D$. Therefore, for each $(w^{\T},\alpha)^{\T}\in F^\ast$, we have $(w^{\T},\alpha)^{\T}\in H^\ast(\bar{y})$, equivalently, $\bar{y}\in H(w,\alpha)$, see \eqref{eqn:H_H*}. %Then, by \eqref{eqn:H_H*}, we have $\bar{y}\in H(w,\alpha)$ for each $(w^{\T},\alpha)^{\T}\in F^\ast $. 
		Hence,	
		\begin{equation*}
		\bar{y}\in\bigcap_{(w^{\T},\alpha)^{\T} \in F^\ast} \big(H(w,\alpha) \cap \P\big) =\Psi(F^\ast).
		\end{equation*}
		Therefore,
		\begin{equation*}
		\Psi^\ast(\Psi(F^\ast))=\bigcap\limits_{y\in\Psi(F^\ast)}\big(H^\ast(y)\cap\D\big)\subseteq H^\ast(\bar{y})\cap\D=F^\ast.
		\end{equation*}
		Hence, the equality $\Psi^\ast(\Psi(F^\ast))=F^\ast$ holds.
		
		\noindent
		(b) Let $F$ be a weakly $C$-minimal proper face of $\P$. Assume for a contradiction that $F \nsubseteq \Psi(\Psi^\ast(F))$. Let $y\in F\setminus\Psi(\Psi^\ast(F)) $. Then, there exists $(\bar{w}^{\T},\bar{\alpha})^{\T} \in \Psi^\ast(F)$ such that $y\notin H(\bar{w},\bar{\alpha})\cap\P $. This implies $y\notin H(\bar{w},\bar{\alpha}) $ since $y\in\P$. By \eqref{eqn:H_H*}, $(\bar{w}^{\T},\bar{\alpha})^{\T}\notin H^\ast(y) $, which implies $(\bar{w}^{\T},\bar{\alpha})^{\T}\notin\Psi^\ast(F)$, a contradiction. Hence, $\Psi(\Psi^\ast(F))\supseteq F $. For the reverse inclusion, first note that, by \Cref{lem:duality_2} (b), there exists $(\bar{w}^{\T},\bar{\alpha})^{\T}\in\D$ such that $F=H(\bar{w},\bar{\alpha})\cap\P$. Then, for each $y\in F$, we have $y\in H(\bar{w},\bar{\alpha})$, equivalently, $(\bar{w}^{\T},\bar{\alpha})^{\T}\in H^\ast(y) $, see \eqref{eqn:H_H*}. Hence,
		\begin{equation*}
		(\bar{w}^{\T},\bar{\alpha})^{\T}\in \bigcap\limits_{y\in F}\big(H^\ast(y)\cap\D\big)=\Psi^\ast(F). 
		\end{equation*}
		Therefore, 
		\begin{equation*}
		\Psi(\Psi^\ast(F))=\bigcap\limits_{(w^{\T},\alpha)^{\T}\in\Psi^\ast(F)}\big(H(w,\alpha)\cap\P\big)\subseteq H(\bar{w},\bar{\alpha})\cap\P=F.
		\end{equation*}
		Hence, the equality $\Psi(\Psi^\ast(F))=F $ holds.
	\end{proof}
	
\begin{proof}[\textbf{Proof of \Cref{lem:6}}]
	Let $\tilde{\P}:=\{y \in \R^q\mid 
	\forall(w^{\T},\alpha)^{\T}\in\xi(\bar{\W})\colon \varphi(y,w,\alpha)\geq0\}$. Since $\bar{\D}\supseteq \xi(\bar{\W})$, we obtain $\tilde{\P} \supseteq \P_{\bar{\D}}$. To show the reverse inclusion, let us fix $y\in\tilde{\P}$. Let $(w^{\T},\alpha)^{\T} \in \bar{\D}$, that is, there exist $n\in \N$, $\lambda_i\geq0 $, $(\tilde{w}_i^{\T},\tilde{\alpha}_i)^{\T} \in \xi(\bar{\W})$ for $i\in\{1,\ldots,n\}$, $\beta\geq 0$ such that $(w^{\T},\alpha+\beta)^{\T}=\sum_{i=1}^{n}\lambda_i(\tilde{w}_i^{\T},\tilde{\alpha}_i)^{\T}$. In particular, $\tilde{w}_i^{\T}y-\tilde{\alpha}_i\geq 0$ for each $i\in\{1,\ldots,n\}$. Then,
	\[
	\varphi(y,w,\alpha)=(w^{\T}y-\alpha-\beta)+\beta=\sum_{i=1}^{n}\lambda_i\tilde{w}_i^{\T}y-\sum_{i=1}^{n}\lambda_i\tilde{\alpha}_i+\beta=\sum_{i=1}^{n}\lambda_i(\tilde{w}_i^{\T}y-\tilde{\alpha}_i)+\beta\geq 0.
	\]
	%	since $\tilde{w}_i^{\T}y-\tilde{\alpha}_i\geq 0$ for every $(\tilde{w}_i^{\T},\tilde{\alpha}_i)^{\T} \in \xi(\bar{\W})$ and $\lambda_i\geq0 $ for $i\in\{1,\ldots,N\}$. 
	Since $\varphi(y,w,\alpha)\geq 0$ for each $(w^{\T},\alpha)^{\T}\in \bar{\D}$, we conclude that $y \in \P_{\bar{\D}}$.
\end{proof}

	\begin{proof}[\textbf{Proof of \Cref{prop:eps_tilde_2}}]
		To start with the proof, we have the following observation: 
		\[
		\D_\epsilon= \cone \conv\left( (\xi(\bar{\W})+\epsilon \{e^{q+1}\}) \cup \{-e^{q+1}\} \right).
		\]
		This implies that the set of extreme directions of $\D_\epsilon$ is a subset of $(\xi(\bar{\W})+\epsilon \{e^{q+1}\}) \cup \{-e^{q+1}\}$. Let $i\in\{1,\ldots,T\}$. By \cite[Section 18, page 162]{rockafellar_1970}, we also have 
		\[
		\{((w^{i1})^\T,\alpha_{i1})^\T,\ldots,((w^{iJ_i})^\T,\alpha_{iJ_i})^\T \}\subseteq (\xi(\bar{\W})+\epsilon \{e^{q+1}\}) \cup \{-e^{q+1}\}.
		\]
		In particular, for $j\in\{1,\ldots,J_i\}$, we have $w^{ij}\in \bar{\W} \cup\{0\}$ and $\alpha_{ij} = \inf_{x\in\X}(w^{ij})^\T f(x)+\epsilon = p^{w^{ij}}+\epsilon$ if $w^{ij}\neq 0$.
		
		%	\Fir{(Yine dimension problemi vardı, önceki ifade bu şekildeydi: } Hence, $w_{ij} \in \xi(\bar{\mathcal{W}})+\epsilon\{e^{q+1}\} $ for each facet.)	
		
		By \Cref{prop:eps_tilde_1}, we have $\P\supseteq\P_\epsilon $. %Since $\bar{\W} $ is a finite $\epsilon $-solution of \eqref{D}, $\D_\epsilon $ is an outer approximation of the lower image $\D$ by \Cref{def:solnDual}, that is, $\D_\epsilon\supseteq\D$. Therefore, by \Cref{prop:DPequalsD} (b), the inclusion $\P\supseteq\P_\epsilon $ holds.
		Following similar steps as in the proof of \Cref{prop:eps_tilde_1}, in order to show that $\P_\epsilon + B(0,\tilde{\epsilon}) \supseteq \P$, we assume the contrary. Then, we obtain $\bar{y}\in\P$, $\bar{w}\in \R^q$ with $\norm{\bar{w}}_\ast=1$ such that
		\begin{equation*}
		\bar{w}^\T\bar{y}+\tilde{\epsilon} < \inf_{y\in\P_\epsilon}\bar{w}^\T y \eqqcolon \bar{\alpha},
		\end{equation*}
		and we can check that $(\bar{w}^\T,\bar{\alpha})^\T \in \D_{\epsilon}$.
		
		By the construction of $\D_\epsilon$, there exists $k\geq 0$ such that $(\bar{w}^\T,\bar{\alpha})^\T+k e^{q+1}\in \bd\D_\epsilon$ is a $K$-maximal element of $\D_\epsilon$. Hence, there exists $I\in\{1,\ldots,T\}$ such that $F_I$ is a $K$-maximal facet of $\D_\epsilon$ and
		\begin{equation}\label{eq:FI}
		(\bar{w}^\T,\bar{\alpha})^\T+k e^{q+1}\in F_I = \cone \conv \{((w^{I1})^\T,\alpha_{I1})^\T,\ldots,((w^{IJ_I})^\T,\alpha_{IJ_I})^\T \}.
		\end{equation}
		Using the $K$-maximality of $F_I$, we may assume that $w^{Ij}\neq 0$, hence $\alpha_{Ij} = p^{w^{Ij}}+\epsilon$ for each $j\in\{1,\ldots,J_I\}$. Then, we can rewrite \eqref{eq:FI} as
		\[
		(\bar{w}^\T,\bar{\alpha})^\T+k e^{q+1}\in F_I = \cone \big(\conv \{((w^{I1})^\T,p^{w^{I1}})^\T,\ldots,((w^{IJ_I})^\T,p^{w^{IJ_I}})^\T \}+\epsilon \{e^{q+1}\} \big).
		\]
		Hence, there exist $\delta \geq0 $ and $\mu \in\Delta^{J_I-1}$ such that 
		\begin{equation}\label{eqn:3}
		(\bar{w}^\T,\bar{\alpha})^\T=\delta\bigg(\sum_{j=1}^{J_I}\mu_{j}((w^{Ij})^\T,p^{w^{Ij}})^\T+\epsilon e^{q+1}  \bigg) - k e^{q+1} .
		\end{equation}
		As exactly in the proof of \Cref{prop:eps_tilde_1}, the aim is to show that $(\bar{w}^\T,\bar{\alpha}-\tilde{\epsilon})^\T \in \D$ in order to get a contradiction. For this purpose, it is sufficient to show that
		\begin{equation} \label{eqn:4}
		\bar{\alpha}-\tilde{\epsilon} \leq \inf_{x\in\X}\bar{w}^\T f(x).
		\end{equation}	
		Following the same steps as in the proof of \Cref{prop:eps_tilde_1}, one can use \eqref{eqn:3} in order to obtain $\bar{\alpha}-\delta\epsilon\leq\inf_{x\in\X} \bar{w}^\T f(x)$ as well as
		\[
		\delta\epsilon \leq \frac{\epsilon}{\underset{\lambda\in\Delta^{J_I-1}}{\min} \norm{\sum_{j=1}^{J_I}\lambda_j w_{Ij}}_\ast} = \frac{\epsilon}{f^I_{\min}} \leq \frac{\epsilon}{\min\{f^1_{\min},\ldots,f^T_{\min} \}} =\tilde{\epsilon}.
		\]
		Then, \eqref{eqn:4} follows from $\bar{\alpha}-\delta\epsilon\leq\inf_{x\in\X} \bar{w}^\T f(x)$ and $\delta\epsilon\leq\tilde{\epsilon}$.
		%	\Fir{(Ispatın tüm detaylarını yazmadım bir öncekiyle aynı olduğu için, eski hali hala aşağıda var.)}
	\end{proof}

	%\section*{Acknowledgments}
	%We would like to acknowledge the assistance of volunteers in putting
	%together this example manuscript and supplement.
	
	\bibliographystyle{plain}
	\bibliography{thesis}
\end{document}